\documentclass[11pt]{article}

\usepackage[margin=1.12in]{geometry}
\usepackage{amsmath,amssymb,amsthm,mathtools}
\usepackage{enumitem}
\usepackage{microtype}
\usepackage{needspace}
\usepackage[hidelinks]{hyperref}

\numberwithin{equation}{section}

\newtheorem{theorem}{Theorem}[section]
\newtheorem{proposition}[theorem]{Proposition}
\newtheorem{lemma}[theorem]{Lemma}
\newtheorem{corollary}[theorem]{Corollary}

\newcommand{\T}{\mathbb T}
\newcommand{\Z}{\mathbb Z}
\newcommand{\N}{\mathbb N}

\newcommand{\id}{\mathrm{id}}
\newcommand{\sgn}{\operatorname{sgn}}

\newcommand{\Tr}{\operatorname{Tr}}

\newcommand{\one}{\mathbf 1}

\title{Closed Response Calculus for SLE Weldings and\\
Weil--Petersson K\"ahler Geometry}
\author{Chunhao Cai\\
{\small School of Mathematics (Zhuhai), Sun Yat-Sen University}\\
{\small \texttt{caichh9@mail.sysu.edu.cn}}}
\date{September 2026}

\begin{document}
\maketitle

\begin{abstract}
For $0<\kappa\leq4$, we construct a closed response calculus for
$\mathrm{SLE}_\kappa$ weldings and a canonical Dirichlet form.
The divergence covariance combines the Weil--Petersson and
Velling--Kirillov forms. The integrated response gives exact changes of
measure and canonical Liouville--capacity increments on conformally
removable weldings.
\end{abstract}

\noindent\textit{2020 Mathematics Subject Classification.}
Primary 60J67; Secondary 30C62, 30F60, 31C25, 58B20.

\noindent\textit{Keywords.}
Schramm--Loewner evolution, conformal welding, response calculus,
Dirichlet forms, Weil--Petersson geometry, Liouville action, quasi-invariance.

\tableofcontents
\clearpage

\section{Introduction}
\label{sec:introduction}

The universal Liouville action is a K\"ahler potential for the
Weil--Petersson metric and equals $\pi$ times the Loewner energy of the
welding curve \cite{TT06,Wang19}.  For Schramm--Loewner evolution (SLE)
loop measures, infinitesimal conformal restriction gives Virasoro
divergence formulas \cite{GQW25}.  For their welding homeomorphisms,
analytic adjoint and likelihood formulas are studied in \cite{BJ25},
and quasi-invariance under Weil--Petersson composition is proved in
\cite{FS25} for SLE parameters in $(0,4)$.
We construct a closed response differential on rough SLE weldings,
identify the covariance geometry of its divergence, and integrate the
divergence to finite transformations.  The resulting increments agree
with analytic limits of the Liouville--capacity action on every conformally
removable Jordan welding.  Integration by parts for the actual right
flow connects this differential with the change of the welding law.

\subsection{The welding law and response differential}

Write $\widehat{\mathbb C}=\mathbb C\cup\{\infty\}$ for the Riemann sphere,
and set
\[
 \mathbb D:=\{z\in\mathbb C:|z|<1\},\qquad
 \mathbb S^1:=\partial\mathbb D,\qquad
 \mathbb D^*:=\widehat{\mathbb C}\setminus\overline{\mathbb D}.
\]
We identify the angular coordinate $\theta\in\T:=\mathbb R/(2\pi\mathbb Z)$
with $e^{i\theta}\in\mathbb S^1$.
Let $\operatorname{Homeo}_+(\mathbb S^1)$ denote the space of
orientation-preserving homeomorphisms of $\mathbb S^1$.

Let $\gamma\subset\mathbb C$ be a Jordan curve.  Denote the two
components of $\widehat{\mathbb C}\setminus\gamma$ by $U_+$ and $U_-$,
with $\infty\in U_-$.  Conformal bijections
\[
 f:\mathbb D\longrightarrow U_+,\qquad
 g:\mathbb D^*\longrightarrow U_-
\]
extend to homeomorphisms of the boundaries and define the
\emph{welding homeomorphism}
\[
 h:=g^{-1}\circ f\big|_{\mathbb S^1}
 \in\operatorname{Homeo}_+(\mathbb S^1).
\]
The curve $\gamma$ is \emph{conformally removable} if every homeomorphism
of $\widehat{\mathbb C}$ that is conformal on
$\widehat{\mathbb C}\setminus\gamma$ is a M\"obius transformation.
We denote by $\mathcal W$ the set of welding homeomorphisms obtained from
conformally removable Jordan curves in this way.
For each $h\in\mathcal W$, the welding pair is unique up to a common
M\"obius postcomposition.  We choose its unique representative $(f_h,g_h)$
satisfying
\[
 h=g_h^{-1}\circ f_h\quad\text{on }\mathbb S^1,\qquad
 f_h(0)=0,\quad f_h'(0)=1,\quad g_h(\infty)=\infty,
\]
and write $f,g$ when $h$ is fixed.

Fix $0<\kappa\leq4$ and set
\begin{equation}\label{eq:intro-central-data}
 Q_\kappa:=\frac{\sqrt\kappa}{2}+\frac{2}{\sqrt\kappa},
 \qquad c_{\mathrm L}:=1+6Q_\kappa^2.
\end{equation}
Let $\widetilde\nu_\kappa$ be the SLE welding probability measure of
\cite[Definition~1.1]{BJ25}.  Its exterior parametrization is marked by a
uniform rotation, independent of the normalized SLE loop.
This measure is carried by $\mathcal W$; conformal removability at
$\kappa=4$ follows from \cite[Theorem~1.1]{KMS22}.
We equip $\operatorname{Homeo}_+(\mathbb S^1)$ with the uniform topology
and $\mathcal W$ with the trace of its Borel sigma-algebra.
We use the completion of this sigma-algebra for $L^p(\widetilde\nu_\kappa)$.
Unless otherwise stated, $L^2(\widetilde\nu_\kappa)$ is a real Hilbert space.
The equivalent realization by normalized welding pairs is established in
Appendix~\ref{sec:DR-frame-spaces}.
For a bounded pointwise function $F$, $\|F\|_{\sup}$ denotes the supremum
of $|F|$ over its whole domain.

We identify a smooth real vector field on the circle with its coefficient
$v\in C^\infty(\T;\mathbb R)$ in $v(\theta)\partial_\theta$.
Its Fourier coefficients are
\[
 v_n:=\frac1{2\pi}\int_0^{2\pi}v(\theta)e^{-in\theta}\,d\theta,
 \qquad n\in\mathbb Z.
\]
The field is called \emph{horizontal} if $v_{-1}=v_0=v_1=0$, in which case
\begin{equation}\label{eq:intro-horizontal-field}
 v(\theta)=\sum_{|n|\geq2}v_ne^{in\theta},
 \qquad v_{-n}=\overline{v_n}.
\end{equation}
On the smooth horizontal fields define the inner product
\begin{equation}\label{eq:intro-WP-metric}
 (v,w)_{\mathrm{WP}}
 :=2\operatorname{Re}\sum_{n=2}^{\infty}
 (n^3-n)v_n\overline{w_n},
 \qquad \|v\|_{\mathrm{WP}}^2:=(v,v)_{\mathrm{WP}}.
\end{equation}
Let $H_{\mathrm{WP}}$ be their completion for this norm, and let
$H_{\mathrm{WP}}^{\mathrm{fin}}$ be the subspace of finite Fourier sums.
For a smooth real field $v$, let $\psi_t^v$ be its circle flow.
Smooth right composition preserves $\mathcal W$ by
\cite[Lemma~4.1]{FS25}.  For $F:\mathcal W\to\mathbb R$ define
\begin{equation}\label{eq:intro-right-response}
 T_t^vh:=h\circ\psi_t^v,
 \qquad D_vF(h):=\left.\frac d{dt}\right|_{t=0}F(T_t^vh)
\end{equation}
whenever the derivative exists.

For $a\in C^\infty(\T;\mathbb R)$ and
$b\in C^\infty(\mathbb S^1;\mathbb R)$, set
\begin{equation}\label{eq:intro-averaged-observable}
 J_{a,b}(h):=\frac1{2\pi}\int_0^{2\pi}
 a(\theta)b(h(e^{i\theta}))\,d\theta.
\end{equation}
For $m\geq1$, let $C_b^\infty(\mathbb R^m;\mathbb R)$ denote the smooth
functions whose values and derivatives of every order are bounded.
Choose $a^{(j)}\in C^\infty(\T;\mathbb R)$ and
$b^{(j)}\in C^\infty(\mathbb S^1;\mathbb R)$ for $1\leq j\leq m$.
Let $\mathcal C$ be the algebra of all functions of the form
\begin{equation}\label{eq:intro-core}
 F(h)=\phi\bigl(J_{a^{(1)},b^{(1)}}(h),\ldots,J_{a^{(m)},b^{(m)}}(h)\bigr),
 \qquad \phi\in C_b^\infty(\mathbb R^m;\mathbb R),
\end{equation}
as $m$, $a^{(j)}$, $b^{(j)}$, and $\phi$ vary.
Lemma~\ref{lem:B1-J-response} and the chain rule show that, for each
$F\in\mathcal C$, there is a constant $C_F<\infty$ such that
\[
 \sup_{h\in\mathcal W}|D_vF(h)|\leq C_F\|v\|_{\mathrm{WP}}
\]
for every smooth horizontal field $v$.
Thus $v\mapsto D_vF(h)$ extends to a bounded real linear functional on
$H_{\mathrm{WP}}$.  Its Riesz representative $DF(h)\in H_{\mathrm{WP}}$
is characterized by
\begin{equation}\label{eq:intro-gradient-identity}
 D_vF(h)=(DF(h),v)_{\mathrm{WP}},\qquad v\in H_{\mathrm{WP}}.
\end{equation}
This defines the response operator
\begin{equation}\label{eq:intro-gradient-operator}
 D:\mathcal C\subset L^2(\widetilde\nu_\kappa)
 \longrightarrow L^2(\widetilde\nu_\kappa;H_{\mathrm{WP}}).
\end{equation}
We use the same notation $D$ for its closure, whose existence is proved
in the next theorem, and write $\operatorname{Dom}(D)$ for its domain.
On this domain define the bilinear form
\begin{equation}\label{eq:intro-Dirichlet-form}
 \mathcal E(F,G):=\int(DF,DG)_{\mathrm{WP}}\,d\widetilde\nu_\kappa,
 \qquad F,G\in\operatorname{Dom}(D).
\end{equation}
A \emph{normal contraction} is a function $\eta:\mathbb R\to\mathbb R$
satisfying
\[
 \eta(0)=0,\qquad |\eta(x)-\eta(y)|\leq|x-y|\quad(x,y\in\mathbb R).
\]
\begin{theorem}\label{thm:B1-main}
The response operator $D$ in \eqref{eq:intro-gradient-operator} is densely
defined and closable, with
\begin{equation}\label{eq:intro-core-density}
 \overline{\mathcal C}^{\,L^2(\widetilde\nu_\kappa)}
 =L^2(\widetilde\nu_\kappa).
\end{equation}
Its closure defines the symmetric Dirichlet form
$(\mathcal E,\operatorname{Dom}(D))$ on $L^2(\widetilde\nu_\kappa)$.
For $F\in\operatorname{Dom}(D)$ and every normal contraction $\eta$,
\begin{equation}\label{eq:intro-energy-contraction}
 \eta\circ F\in\operatorname{Dom}(D),\qquad
 \mathcal E(\eta\circ F,\eta\circ F)\leq\mathcal E(F,F).
\end{equation}
\end{theorem}

\subsection{Stress covariance and geometry}

We express integration by parts for the right response $D_v$ through the
coefficients of the normalized interior map.  For $h\in\mathcal W$, write $f=f_h$ and define
$a_m(h)$ by
\begin{equation}\label{eq:intro-interior-map}
 f(z)=z\left(1+\sum_{m\geq1}a_m(h)z^m\right),\qquad z\in\mathbb D.
\end{equation}
Its Schwarzian derivative is
\[
 \mathcal Sf:=\frac{f'''}{f'}-\frac32\left(\frac{f''}{f'}\right)^2.
\]
For an integer $j$, let $[z^j]$ denote the coefficient of $z^j$ in a
Laurent expansion at zero.  Define the stress function $q_f$ and its
coefficients $t_n$ by
\begin{align}
 q_f(z)&:=\frac{c_{\mathrm L}}{12}\mathcal Sf(z)
   +\frac{f'(z)^2}{f(z)^2}-\frac1{z^2},
 \label{eq:intro-stress}\\
 \shortintertext{and}
 t_n(h)&:=[z^{n-2}]q_f(z),\qquad n\geq1.
 \label{eq:intro-stress-mode}
\end{align}
The normalization of $f$ gives
\[
 \frac{f'(z)^2}{f(z)^2}-\frac1{z^2}
 =\frac{2a_1(h)}z+O(1),\qquad z\to0.
\]
Thus $q_f$ has at most a simple pole at zero, and
$q_f(z)=\sum_{n\geq1}t_n(h)z^{n-2}$.
Each $t_n$ is a polynomial in finitely many coefficients of $f$ and is
bounded on $\mathcal W$; see Lemma~\ref{lem:DR-stress-polynomial}.

Every $v\in H_{\mathrm{WP}}^{\mathrm{fin}}$ can be written, for some
integer $N\geq2$ and real coefficients $x_n,y_n$, as
\begin{equation}\label{eq:intro-real-finite-field}
 v(\theta)=\sum_{n=2}^{N}
 \{x_n\cos(n\theta)+y_n\sin(n\theta)\}.
\end{equation}
Its stress score is the bounded real-valued function
\begin{equation}\label{eq:intro-score}
 \rho_v(h):=\sum_{n=2}^{N}
 \{x_n\operatorname{Im}t_n(h)+y_n\operatorname{Re}t_n(h)\}.
\end{equation}
We write $\rho(v)=\rho_v$ for the resulting linear map on
$H_{\mathrm{WP}}^{\mathrm{fin}}$ and use the same notation for its continuous
extension below.
Proposition~\ref{prop:DR-right-IBP}, proved in Appendix~\ref{app:right-IBP},
identifies $\rho_v$ as the divergence of the actual right flow:
\begin{equation}\label{eq:intro-right-IBP}
 \int D_vF\,d\widetilde\nu_\kappa
 =\int F\rho_v\,d\widetilde\nu_\kappa,
 \qquad F\in\mathcal C,\quad v\in H_{\mathrm{WP}}^{\mathrm{fin}}.
\end{equation}
This identity gives the closability in Theorem~\ref{thm:B1-main} and the
adjoint relation in the next theorem.

\begin{samepage}
To state the covariance formula, define on $H_{\mathrm{WP}}$ the
Velling--Kirillov form and the operator $J_0$ by
\begin{align}
 (v,w)_{\mathrm{VK}}&:=\operatorname{Re}\sum_{n=2}^{\infty}
 n v_n\overline{w_n},\label{eq:intro-VK-form}\\
 \shortintertext{and}
 (J_0v)_n&:=i\,\sgn(n)v_n,\qquad |n|\geq2.
 \label{eq:intro-complex-structure}
\end{align}
\end{samepage}
Set
\begin{align}
 G_\kappa(v,w)&:=\frac{c_{\mathrm L}}{12}(v,w)_{\mathrm{WP}}
                 +4(v,w)_{\mathrm{VK}},
 \label{eq:intro-geometric-form}\\
 \omega_{\kappa,0}(v,w)&:=G_\kappa(J_0v,w).
 \label{eq:intro-symplectic-form}
\end{align}
Let $L_0^2(\widetilde\nu_\kappa;\mathbb R)$ denote the real square-integrable
functions with zero integral, and let $D^*$ be the Hilbert-space adjoint of
the closed operator $D$ from Theorem~\ref{thm:B1-main}.
We identify $v\in H_{\mathrm{WP}}$ with the constant function $h\mapsto v$
in $L^2(\widetilde\nu_\kappa;H_{\mathrm{WP}})$.

\begin{theorem}\label{thm:intro-Fisher-Kahler}
The score map $\rho$ extends uniquely to a bounded real-linear map
\begin{equation}\label{eq:intro-score-extension}
 \rho:H_{\mathrm{WP}}\longrightarrow
 L_0^2(\widetilde\nu_\kappa;\mathbb R).
\end{equation}
For $v,w\in H_{\mathrm{WP}}$, we have
\begin{equation}\label{eq:intro-adjoint-score}
 v\in\operatorname{Dom}(D^*),\qquad D^*v=\rho_v,
\end{equation}
\noindent and
\begin{equation}\label{eq:intro-Fisher-form}
 \int\rho_v\rho_w\,d\widetilde\nu_\kappa=G_\kappa(v,w).
\end{equation}
\end{theorem}

The covariance metric $G_\kappa$ is equivalent to the Weil--Petersson
metric.  Indeed, for $v,w\in H_{\mathrm{WP}}$, the definitions give
\begin{align}
 \frac{c_{\mathrm L}}{12}\|v\|_{\mathrm{WP}}^2
 &\leq G_\kappa(v,v)
 \leq\left(\frac{c_{\mathrm L}}{12}+\frac23\right)
       \|v\|_{\mathrm{WP}}^2,
 \label{eq:intro-score-norm-bounds}\\
 \shortintertext{and}
 G_\kappa(J_0v,J_0w)&=G_\kappa(v,w).
 \label{eq:intro-score-J-invariance}
\end{align}
Together with $J_0^2v=-v$, these identities make
$(G_\kappa,J_0,\omega_{\kappa,0})$ a constant strong K\"ahler structure on
$H_{\mathrm{WP}}$.

\subsection{Integrated response and the rough action}

We next integrate the finite-mode scores along the right flows.
Fix an integer $\ell\geq1$, times $t_1,\ldots,t_\ell\in\mathbb R$, and
fields $v_1,\ldots,v_\ell\in H_{\mathrm{WP}}^{\mathrm{fin}}$.
Let
\begin{equation}\label{eq:intro-finite-composition}
 \Phi:=\psi_{t_1}^{v_1}\circ\cdots\circ\psi_{t_\ell}^{v_\ell},
\end{equation}
and write $\Phi_0=\id$ and
$\Phi_j=\psi_{t_1}^{v_1}\circ\cdots\circ\psi_{t_j}^{v_j}$ for
$1\leq j\leq\ell$.
For $h\in\mathcal W$, define the integrated response by
\begin{equation}\label{eq:intro-integrated-response}
 B_\Phi(h):=\sum_{j=1}^{\ell}\int_0^{t_j}
 \rho_{v_j}(h\circ\Phi_{j-1}\circ\psi_s^{v_j})\,ds.
\end{equation}
This expression is initially attached to the chosen sequence of flows.
Theorem~\ref{thm:DR-radial-closure} shows that it depends only on the
resulting circle diffeomorphism $\Phi$.

For a normalized welding pair $(f,g)$, write
$g'(\infty):=\lim_{z\to\infty}g(z)/z$ and define
\begin{equation}\label{eq:DR-capacity-potential-definition}
 k(h):=\log\left|\frac{g'(\infty)}{f'(0)}\right|
      =\log|g'(\infty)|.
\end{equation}
Let $dA$ denote planar Lebesgue area.
For a smooth circle diffeomorphism $h\in\mathcal W$, the universal Liouville action
in the normalization of \cite{TT06} is
\begin{equation}\label{eq:DR-universal-Liouville-definition}
 S_1(h):=\int_{\mathbb D}\left|\frac{f''}{f'}\right|^2dA
 +\int_{\mathbb D^*}\left|\frac{g''}{g'}\right|^2dA-4\pi k(h).
\end{equation}
The corresponding Liouville--capacity action is
\begin{equation}\label{eq:intro-smooth-action}
 H_\kappa(h):=\frac{c_{\mathrm L}}{24\pi}S_1(h)+2k(h).
\end{equation}
For $h\in\mathcal W$ and $0<r<1$, define
\begin{equation}\label{eq:intro-radial-map}
 f^{(r)}(z):=r^{-1}f(rz)
 =z\left(1+\sum_{m\geq1}r^ma_m(h)z^m\right).
\end{equation}
The map $f^{(r)}$ extends conformally across $\mathbb S^1$, so its boundary
curve is analytic.  Choose a conformal bijection $g^{(r)}$ from
$\mathbb D^*$ onto the component of
$\widehat{\mathbb C}\setminus f^{(r)}(\mathbb S^1)$ containing infinity,
with $g^{(r)}(\infty)=\infty$, and set
\[
 h^{(r)}:=(g^{(r)})^{-1}\circ f^{(r)}\big|_{\mathbb S^1}.
\]
Then $h^{(r)}\in\mathcal W$ is a real-analytic circle diffeomorphism.
The choice of $g^{(r)}$ is unique up to a rotation of its argument;
this is the exterior rotational mark.  We denote the limit of
$H_\kappa(h^{(r)}\circ\Phi)-H_\kappa(h^{(r)})$ as $r\uparrow1$, when it
exists, by $B_\Phi^{\mathrm{rad}}(h)$.

For the same $h\in\mathcal W$, let $(\widehat h_N)_{N\geq1}\subset\mathcal W$
be any sequence of real-analytic circle diffeomorphisms, with normalized
interior maps $\widehat f_N$, such that
\[
 a_m(\widehat h_N)=[z^{m+1}]\widehat f_N\longrightarrow a_m(h)
 \qquad\text{for every }m\geq1.
\]

\begin{theorem}\label{thm:DR-radial-closure}
For every $h\in\mathcal W$, every $\Phi$ as in
\eqref{eq:intro-finite-composition}, and every sequence $(\widehat h_N)$
specified above, both limits below exist and satisfy
\begin{align}
 \lim_{r\uparrow1}
 \{H_\kappa(h^{(r)}\circ\Phi)-H_\kappa(h^{(r)})\}
 &=B_\Phi^{\mathrm{rad}}(h)=B_\Phi(h),
 \label{eq:intro-radial-equality}\\
 \shortintertext{and}
 \lim_{N\to\infty}
 \{H_\kappa(\widehat h_N\circ\Phi)-H_\kappa(\widehat h_N)\}
 &=B_\Phi(h).
 \label{eq:intro-analytic-approximation}
\end{align}
The common value $B_\Phi(h)$ is independent of the representation of $\Phi$ as a
composition of flows and of the exterior rotational marks.
\end{theorem}

The same increment determines the change of the welding law under right
composition.  Let $R_\Phi h:=h\circ\Phi$, and let
$(R_\Phi)_\#\widetilde\nu_\kappa$ denote the image measure.
In the target variable $y\in\mathcal W$,
Proposition~\ref{prop:DR-finite-RN} and
Theorem~\ref{thm:DR-radial-closure} give the Radon--Nikodym density
\begin{equation}\label{eq:intro-RN-identity}
 \frac{d(R_\Phi)_\#\widetilde\nu_\kappa}{d\widetilde\nu_\kappa}(y)
 =\exp\{B_\Phi^{\mathrm{rad}}(y\circ\Phi^{-1})\},
 \qquad \widetilde\nu_\kappa\text{-a.e. }y.
\end{equation}
For $v,w\in H_{\mathrm{WP}}^{\mathrm{fin}}$ and $s,t\in\mathbb R$, put
$\Phi_{s,t}:=\psi_s^v\circ\psi_t^w$.
The negative logarithm of this density is
\[
 A_{s,t}^{\mathrm{rad}}(y):=
 -B_{\Phi_{s,t}}^{\mathrm{rad}}(y\circ\Phi_{s,t}^{-1}).
\]
Its ordered mixed derivative, first in $s$ and then in $t$, exists in
$L^\infty(\widetilde\nu_\kappa)$ at $(s,t)=(0,0)$.
Corollary~\ref{cor:DR-radial-Kahler-Hessian} gives
\begin{equation}\label{eq:intro-action-Hessian}
 \int\left.\partial_t\left(
 \left.\partial_s A_{s,t}^{\mathrm{rad}}\right|_{s=0}
 \right)\right|_{t=0}d\widetilde\nu_\kappa=G_\kappa(v,w).
\end{equation}
Thus the covariance in Theorem~\ref{thm:intro-Fisher-Kahler} is also the
averaged second derivative of the negative log density.

Section~\ref{sec:deterministic-interface} fixes the deterministic
covariance and complex-structure normalizations using the welding operator
of \cite{FVW26}.  Section~\ref{sec:sle-closed} constructs the response core
on rough weldings.  Section~\ref{sec:direct-right} states the coefficient responses and right
integration-by-parts formula; their proofs occupy
Appendices~\ref{app:coefficient-response}--\ref{app:right-IBP}.
The final proofs of Theorems~\ref{thm:B1-main},
\ref{thm:intro-Fisher-Kahler}, and~\ref{thm:DR-radial-closure} are given in
Sections~\ref{sec:proof-closed-response},
\ref{sec:proof-stress-covariance}, and~\ref{sec:proof-radial-action},
respectively.

\section{Deterministic calibration of the response calculus}
\label{sec:deterministic-interface}

The covariance response of the welding operator fixes the
Weil--Petersson and complex-structure normalizations used below.

\subsection{Round-welding differential and Hilbert--Schmidt response}
\label{sec:round}

Let $\varphi$ be an orientation-preserving quasisymmetric homeomorphism
of $\mathbb S^1$.  For $s,t\in\T$ with $s\ne t$, define the logarithmic
kernel of Fan--Viklund--Wang \cite{FVW26} by
\begin{equation}\label{eq:Lphi}
 L_\varphi(s,t)
 =\log\left|
 \frac{\varphi(e^{is})-\varphi(e^{it})}{e^{is}-e^{it}}
 \right|,
\end{equation}
and its Fourier coefficients
\begin{equation}\label{eq:fourier-lambda}
 \widehat\lambda_{k,\ell}(\varphi)
 =\frac{1}{(2\pi)^2}\int_0^{2\pi}\int_0^{2\pi}
 L_\varphi(s,t)e^{-i(ks+\ell t)}\,ds\,dt,
 \qquad k,\ell\in\Z.
\end{equation}
For $k,\ell\neq0$, put
\begin{equation}\label{eq:normalized-lambda}
 \lambda_{k,\ell}(\varphi)
 =\sqrt{|k\ell|}\,\widehat\lambda_{k,\ell}(\varphi).
\end{equation}
Write $\mathbb N=\{1,2,\ldots\}$.  On
$\ell^2(\N)\oplus\ell^2(\N)$, the welding Grunsky operator has the block form
\begin{equation}\label{eq:Lambda-block}
 \Lambda_\varphi
 =\begin{pmatrix}M_\varphi&N_\varphi\\
 \overline{N_\varphi}&\overline{M_\varphi}\end{pmatrix},
 \qquad
 \begin{aligned}
 (M_\varphi)_{k\ell}&=\lambda_{k,-\ell}(\varphi),\\[3pt]
 (N_\varphi)_{k\ell}&=\lambda_{k,\ell}(\varphi),
 \end{aligned}
\end{equation}
for $k,\ell\geq1$.  Quasisymmetry ensures integrability of the kernel and
boundedness of this self-adjoint operator \cite[Theorem~1.3]{FVW26}.

Let $v,w\in C^\infty(\T;\mathbb R)$, with the Fourier convention
of Section~\ref{sec:introduction}.
For $|\varepsilon|$ sufficiently small,
\begin{equation}\label{eq:phi-eps}
 \varphi_\varepsilon(e^{i\theta})
 =e^{i(\theta+\varepsilon v(\theta))}
\end{equation}
defines a smooth orientation-preserving circle diffeomorphism.  We write
\begin{equation}\label{eq:response-def}
 \dot\Lambda_v
 =\left.\frac{d}{d\varepsilon}\right|_{\varepsilon=0}
 \Lambda_{\varphi_\varepsilon},
\end{equation}
where the derivative will be proved below to exist in Hilbert--Schmidt norm.

The modes $-1,0,1$ are the infinitesimal M\"obius directions.
On horizontal fields we use the Weil--Petersson inner product
\eqref{eq:intro-WP-metric}; see \cite[Eq.~(11)]{Nag92} for its Fourier
description at the identity.

We write $\one$, $A^*$, $\|A\|$, and $\Tr A$ for the identity,
adjoint, operator norm, and trace, respectively.  Write $S_2$ for the
Hilbert--Schmidt class between the Hilbert spaces in use, with norm
$\|\cdot\|_{\mathrm{HS}}$, and $S_2^{\mathrm{sa}}=\{A\in S_2:A^*=A\}$
for the real vector space of self-adjoint operators on a fixed Hilbert
space.  Complex inner products are linear in the first argument.
On either Fourier block, and on their direct sum, the Hilbert--Schmidt
pairing is $\langle A,B\rangle_{\mathrm{HS}}:=\Tr(AB^*)$.  Its real part is
\begin{equation}\label{eq:real-HS}
 (A,B)_{\mathrm{HS},\mathbb R}
 =\operatorname{Re}\Tr(AB^*).
\end{equation}

Define the matrix $\dot N_v$ by
\[
 (\dot N_v)_{k\ell}:=i\sqrt{k\ell}\,v_{k+\ell},\qquad k,\ell\geq1.
\]

\begin{proposition}\label{prop:det-round-response}
The response along \eqref{eq:phi-eps} has the following properties.

\begin{enumerate}[label=\textup{(\roman*)},leftmargin=2.3em]
\item The Hilbert--Schmidt derivative satisfies
\begin{equation}\label{eq:det-round-HS-derivative-limit}
 \lim_{\varepsilon\to0}
 \left\|\frac{\Lambda_{\varphi_\varepsilon}-\Lambda_{\id}}{\varepsilon}
          -\dot\Lambda_v\right\|_{\mathrm{HS}}=0.
\end{equation}

\item For all $k,\ell\in\Z\setminus\{0\}$,
\begin{equation}\label{eq:coeff-response-main}
 \left.\frac{d}{d\varepsilon}\right|_{0}
 \widehat\lambda_{k,\ell}(\varphi_\varepsilon)
 =\frac{i}{2}\bigl(\sgn k+\sgn\ell\bigr)v_{k+\ell}.
\end{equation}
Consequently
\begin{equation}\label{eq:block-response-main}
 \dot\Lambda_v
 =\begin{pmatrix}0&\dot N_v\\ \overline{\dot N_v}&0\end{pmatrix}.
\end{equation}
In particular, every infinitesimal M\"obius vector field is annihilated by
$v\mapsto\dot\Lambda_v$.

\item If $v,w$ are horizontal, then
\begin{equation}\label{eq:isometry-main}
 (v,w)_{\mathrm{WP}}
 =6\,(\dot\Lambda_v,\dot\Lambda_w)_{\mathrm{HS},\mathbb R}.
\end{equation}
Equivalently,
\begin{equation}\label{eq:norm-isometry-main}
 \|v\|_{\mathrm{WP}}^2
 =6\|\dot\Lambda_v\|_{\mathrm{HS}}^2.
\end{equation}

\item The response map has
a unique continuous extension
\begin{equation}\label{eq:extension-main}
 v\longmapsto\dot\Lambda_v:
 H_{\mathrm{WP}}\longrightarrow S_2^{\mathrm{sa}}
\end{equation}
for which
\[
 v\longmapsto\sqrt6\,\dot\Lambda_v
 \quad\text{is an isometric embedding}.
\]
Its range is closed.
\end{enumerate}
\end{proposition}

Section~\ref{sec:regularity} establishes the Hilbert--Schmidt
differentiability in part~\textup{(i)}.  Section~\ref{sec:proof-main}
computes the Fourier response and then derives the block formula, the
Weil--Petersson isometry, and the continuous extension, completing the
proof of Proposition~\ref{prop:det-round-response}.

\subsubsection{Kernel regularity and Hilbert--Schmidt control}\label{sec:regularity}

Choose $\varepsilon_0>0$ with
$\varepsilon_0\|v'\|_\infty<1/2$, and put
$a_\varepsilon(s)=s+\varepsilon v(s)$ and
$L_\varepsilon:=L_{\varphi_\varepsilon}$, using
\eqref{eq:Lphi} and \eqref{eq:phi-eps}.

\begin{lemma}\label{lem:smooth-kernel}
The kernels extend smoothly across the diagonal.  For
$|\varepsilon|<\varepsilon_0$ and every integer $m\geq0$,
\begin{equation}\label{eq:det-kernel-smooth-dependence}
 L_\varepsilon\in C^\infty(\T^2),\qquad
 (\varepsilon\mapsto L_\varepsilon)
 \in C^\infty\bigl(({-}\varepsilon_0,\varepsilon_0);C^m(\T^2)\bigr).
\end{equation}
\end{lemma}

\begin{proof}
The choice of $\varepsilon_0$ gives
\begin{equation}\label{eq:positive-derivative}
 a_\varepsilon'(s)=1+\varepsilon v'(s)\geq\frac12
 \qquad\text{for all }s\in\mathbb R,
 \quad |\varepsilon|<\varepsilon_0.
\end{equation}
Then $a_\varepsilon$ is a strictly increasing lift of an orientation-preserving
circle diffeomorphism and
$a_\varepsilon(s+2\pi)=a_\varepsilon(s)+2\pi$.
For $s\neq t$ modulo $2\pi$,
\begin{equation}\label{eq:sine-ratio}
 \left|
 \frac{e^{ia_\varepsilon(s)}-e^{ia_\varepsilon(t)}}
 {e^{is}-e^{it}}
 \right|
 =\left|
 \frac{\sin((a_\varepsilon(s)-a_\varepsilon(t))/2)}
 {\sin((s-t)/2)}
 \right|.
\end{equation}
It remains to check the diagonal.  Work in a local lift for which
$|s-t|<\pi$.  Put
\[
 b_\varepsilon(s,t)
 =\begin{cases}
 \dfrac{a_\varepsilon(s)-a_\varepsilon(t)}{s-t},&s\neq t,\\[1.2ex]
 a_\varepsilon'(s),&s=t.
 \end{cases}
\]
The identity
\begin{equation}\label{eq:b-int}
 b_\varepsilon(s,t)
 =\int_0^1 a_\varepsilon'\bigl(t+r(s-t)\bigr)\,dr
\end{equation}
shows that $b_\varepsilon$ is smooth jointly in
$(\varepsilon,s,t)$ and, by \eqref{eq:positive-derivative}, is strictly
positive.  With
\[
 \operatorname{sinc}(x)=\begin{cases}\sin x/x,&x\neq0,\\1,&x=0,\end{cases}
\]
we have near the diagonal
\begin{equation}\label{eq:sinc-factorization}
 \frac{\sin((a_\varepsilon(s)-a_\varepsilon(t))/2)}
 {\sin((s-t)/2)}
 =b_\varepsilon(s,t)
 \frac{\operatorname{sinc}((a_\varepsilon(s)-a_\varepsilon(t))/2)}
 {\operatorname{sinc}((s-t)/2)}.
\end{equation}
On a fixed neighborhood of the diagonal the factors are positive, and
\begin{align*}
 L_\varepsilon(s,t)
 &=\log b_\varepsilon(s,t)
   +\log\operatorname{sinc}\frac{a_\varepsilon(s)-a_\varepsilon(t)}2
   -\log\operatorname{sinc}\frac{s-t}2,\\
 \shortintertext{with diagonal value}
 L_\varepsilon(s,s)&=\log\{1+\varepsilon v'(s)\}.
\end{align*}
This gives a jointly smooth extension in $(\varepsilon,s,t)$.  Off the
diagonal, the two sine factors in \eqref{eq:sine-ratio} do not vanish.
A finite cover of $\T^2$ therefore gives, for every compact interval
$I\subset(-\varepsilon_0,\varepsilon_0)$ and integers $j,m\geq0$,
\[
 \sup_{\varepsilon\in I}
 \max_{a+b\leq m}
 \|\partial_\varepsilon^j\partial_s^a\partial_t^b
       L_\varepsilon\|_\infty<\infty.
\]
For $\varepsilon,\varepsilon+h\in I$, the integral remainder satisfies
\[
 \|L_{\varepsilon+h}-L_\varepsilon
       -h\partial_\varepsilon L_\varepsilon\|_{C^m}
 \leq\frac{h^2}{2}
       \sup_{r\in I}\|\partial_r^2 L_r\|_{C^m}.
\]
Applying the same estimate to every parameter derivative proves
\eqref{eq:det-kernel-smooth-dependence}.
\end{proof}

To pass from kernels to operators, use the normalized norm
\[
 \|F\|_{L^2(\T^2)}^2
 =\frac{1}{(2\pi)^2}\int_{\T^2}|F(s,t)|^2\,ds\,dt.
\]

Let $F\in H^2(\T^2)$ and write
\[
 \widehat F_{k,\ell}
 =\frac{1}{(2\pi)^2}\int_{\T^2}
 F(s,t)e^{-i(ks+\ell t)}\,ds\,dt.
\]
Define the matrix $A_F$ on $\ell^2(\Z\setminus\{0\})$ by
\begin{equation}\label{eq:AF}
 (A_F)_{k\ell}=\sqrt{|k\ell|}\,\widehat F_{k,-\ell},
 \qquad k,\ell\neq0.
\end{equation}

\begin{lemma}\label{lem:fourier-HS}
The matrix $A_F$ defines a Hilbert--Schmidt operator, and
\begin{equation}\label{eq:HS-bound-F}
 \|A_F\|_{\mathrm{HS}}^2
 =\sum_{k,\ell\neq0}|k\ell|\,|\widehat F_{k,\ell}|^2
 \leq
 \|\partial_s\partial_tF\|_{L^2(\T^2)}^2.
\end{equation}
\end{lemma}

\begin{proof}
The equality follows from the change of index $\ell\mapsto-\ell$ in
\eqref{eq:AF}.  Since $k,\ell\neq0$,
$|k\ell|\leq k^2\ell^2$.  Parseval's identity gives
\[
 \sum_{k,\ell\in\Z}k^2\ell^2|\widehat F_{k,\ell}|^2
 =\|\partial_s\partial_tF\|_{L^2(\T^2)}^2,
\]
which proves \eqref{eq:HS-bound-F}.
\end{proof}

Set $\dot L_v:=\left.\partial_\varepsilon L_\varepsilon\right|_{\varepsilon=0}$;
Lemma~\ref{lem:smooth-kernel} gives $\dot L_v\in C^\infty(\T^2)$.

\begin{proposition}\label{prop:HS-diff}
As $\varepsilon\to0$,
\begin{equation}\label{eq:C2-Taylor}
 \|L_{\varphi_\varepsilon}-\varepsilon\dot L_v\|_{C^2(\T^2)}
 =O_v(\varepsilon^2),
\end{equation}
and
\begin{equation}\label{eq:S2-Taylor}
 \|\Lambda_{\varphi_\varepsilon}-\varepsilon A_{\dot L_v}\|_{\mathrm{HS}}
 =O_v(\varepsilon^2).
\end{equation}
In particular, $\dot\Lambda_v=A_{\dot L_v}$ in $S_2$.
\end{proposition}

\begin{proof}
Since $L_0=0$, Lemma~\ref{lem:smooth-kernel} and the integral form of
Taylor's theorem give, in $C^2(\T^2)$,
\begin{align*}
 L_\varepsilon-\varepsilon\dot L_v
 &=\varepsilon^2\int_0^1(1-r)
       \left.\partial_u^2L_u\right|_{u=r\varepsilon}\,dr,\\
 \shortintertext{and hence}
 \|L_\varepsilon-\varepsilon\dot L_v\|_{C^2}
 &\leq\frac{\varepsilon^2}{2}
       \sup_{|u|\leq\varepsilon_0/2}
       \|\partial_u^2L_u\|_{C^2}
 \qquad (|\varepsilon|\leq\varepsilon_0/2).
\end{align*}
This proves \eqref{eq:C2-Taylor}.  Under the positive/negative Fourier
decomposition, \eqref{eq:fourier-lambda}--\eqref{eq:Lambda-block} give
\[
 \Lambda_{\varphi_\varepsilon}=A_{L_\varepsilon},\qquad
 \Lambda_{\varphi_\varepsilon}-\varepsilon A_{\dot L_v}
 =A_{L_\varepsilon-\varepsilon\dot L_v}.
\]
Lemma~\ref{lem:fourier-HS} therefore yields
\begin{align*}
 \|\Lambda_{\varphi_\varepsilon}-\varepsilon A_{\dot L_v}\|_{\mathrm{HS}}
 &\leq\|\partial_s\partial_t
       (L_\varepsilon-\varepsilon\dot L_v)\|_{L^2(\T^2)}\\
 &\leq\|L_\varepsilon-\varepsilon\dot L_v\|_{C^2(\T^2)}
 =O_v(\varepsilon^2).
\end{align*}
Dividing by $|\varepsilon|$ proves the derivative assertion.
\end{proof}

Fix a family $(a_t)$ of lifts of smooth orientation-preserving circle
diffeomorphisms such that, for a smooth real periodic $v$,
\begin{equation}\label{eq:general-flow-lift-expansion}
 \|a_t-\id-tv\|_{C^4(\T)}=O(t^2)
 \qquad(t\to0).
\end{equation}
For a lift $a$, the notation $L_a$ and $\Lambda_a$ refers to the induced
circle map $e^{i\theta}\mapsto e^{ia(\theta)}$.

\begin{lemma}\label{lem:flow-HS-diff}
The kernel and operator satisfy
\begin{align}
 \|L_{a_t}-t\dot L_v\|_{C^2(\T^2)}&=O(t^2),
 \label{eq:general-flow-kernel-expansion}\\
 \shortintertext{and}
 \|\Lambda_{a_t}-t\dot\Lambda_v\|_{\mathrm{HS}}&=O(t^2).
 \label{eq:general-flow-S2-expansion}
\end{align}
\end{lemma}

\begin{proof}
For a lift $a$ sufficiently close to the identity in $C^1$, define
\begin{equation}\label{eq:general-flow-b}
 b_a(s,t):=\int_0^1a'\bigl(t+r(s-t)\bigr)\,dr.
\end{equation}
The factorization \eqref{eq:sinc-factorization}, with $b_\varepsilon$
replaced by $b_a$, expresses $L_a$ near the diagonal as
\begin{equation}\label{eq:general-flow-L-factorization}
 L_a(s,t)=\log b_a(s,t)
 +\log\left|
 \frac{\operatorname{sinc}((a(s)-a(t))/2)}
      {\operatorname{sinc}((s-t)/2)}\right|.
\end{equation}
Away from the diagonal the defining logarithmic difference ratio is a
smooth function of $(a(s),a(t))$.  A finite diagonal/off-diagonal cover of
$\T^2$ and \eqref{eq:general-flow-L-factorization} therefore show that
\begin{equation}\label{eq:general-flow-Banach-map}
 a\longmapsto L_a
 \quad\text{is }C^2\text{ from a }C^4\text{-neighborhood of }\id
 \text{ into }C^2(\T^2).
\end{equation}
Its derivative at $\id$ in direction $u$ is $\dot L_u$ by definition.
The factorization gives a constant $C<\infty$ such that, for every
sufficiently small smooth real periodic $u$,
\[
 \|\dot L_u\|_{C^2}\leq C\|u\|_{C^4},\qquad
 \sup_{0\leq r\leq1}
 \|\partial_r^2L_{\id+ru}\|_{C^2}\leq C\|u\|_{C^4}^2,
\]
where $C$ is uniform on a fixed $C^4$ neighborhood of zero.
Taylor's formula along $\id+r(a_t-\id)$ now gives
\begin{align*}
 L_{a_t}-t\dot L_v
 &=\dot L_{a_t-\id-tv}
   +\int_0^1(1-r)\partial_r^2
            L_{\id+r(a_t-\id)}\,dr,\\
 \shortintertext{so that}
 \|L_{a_t}-t\dot L_v\|_{C^2}
 &\leq C\|a_t-\id-tv\|_{C^4}
       +\frac C2\|a_t-\id\|_{C^4}^2
 =O(t^2).
\end{align*}
Here \eqref{eq:general-flow-lift-expansion} controls both terms.
Finally,
\[
 \|\Lambda_{a_t}-t\dot\Lambda_v\|_{\mathrm{HS}}
 \leq\|L_{a_t}-t\dot L_v\|_{C^2}=O(t^2),
\]
by Lemma~\ref{lem:fourier-HS} and $\dot\Lambda_v=A_{\dot L_v}$.
This proves both assertions.
\end{proof}

\subsubsection{Exact computation of the first response}\label{sec:proof-main}

We now calculate $\dot L_v$ and its Fourier coefficients.

\begin{lemma}\label{lem:kernel-response}
For $s\neq t$ modulo $2\pi$,
\begin{equation}\label{eq:kernel-response}
 \dot L_v(s,t)
 =\frac12\cot\frac{s-t}{2}\,[v(s)-v(t)].
\end{equation}
The extension across the diagonal satisfies
\begin{equation}\label{eq:det-kernel-response-diagonal}
 \dot L_v\in C^\infty(\T^2),\qquad \dot L_v(s,s)=v'(s).
\end{equation}
\end{lemma}

\begin{proof}
Using \eqref{eq:sine-ratio}, for $s\neq t$ we have
\[
 L_{\varphi_\varepsilon}(s,t)
 =\log\left|
 \frac{\sin((s-t+\varepsilon(v(s)-v(t)))/2)}
 {\sin((s-t)/2)}
 \right|.
\]
Differentiating at $\varepsilon=0$ gives
\eqref{eq:kernel-response}.  As $t\to s$,
\[
 \frac12\cot\frac{s-t}{2}\,[v(s)-v(t)]
 =\left(\frac{1}{s-t}+O(s-t)\right)
 \left(v'(s)(s-t)+O((s-t)^2)\right),
\]
which tends to $v'(s)$.  The smooth extension follows from
Lemma~\ref{lem:smooth-kernel}.
\end{proof}

For $0<u<2\pi$ and $m\in\mathbb Z\setminus\{0\}$, set
\[
 q(u):=\frac12\cot\frac u2,\qquad
 j_m:=\frac1{2\pi}\int_0^{2\pi}q(u)(e^{imu}-1)\,du.
\]
The integral is absolutely convergent because
\[
 |e^{imu}-1|\leq2|m|\left|\sin\frac u2\right|,
 \qquad |q(u)(e^{imu}-1)|\leq|m|.
\]

\begin{lemma}\label{lem:cot-integral}
For every nonzero integer $m$,
\begin{equation}\label{eq:Jm}
 j_m=\frac i2\sgn(m).
\end{equation}
\end{lemma}

\begin{proof}
Assume first that $m\geq1$.  Under the change $u\mapsto2\pi-u$,
$q(u)$ changes sign whereas $\cos(mu)-1$ does not.  Hence
\begin{equation}\label{eq:real-zero}
 \int_0^{2\pi}q(u)(\cos(mu)-1)\,du=0.
\end{equation}
For the imaginary part, the trigonometric identity
\begin{equation}\label{eq:dirichlet-id}
 \cot\frac u2\,\sin(mu)
 =1+2\sum_{j=1}^{m-1}\cos(ju)+\cos(mu)
\end{equation}
holds for $0<u<2\pi$.  Indeed, multiply
$2\sin(mu)\cos(u/2)=\sin((m+\tfrac12)u)+\sin((m-\tfrac12)u)$
by $(2\sin(u/2))^{-1}$ and use the Dirichlet-kernel formulas.
Integrating \eqref{eq:dirichlet-id} over $[0,2\pi]$ gives
\[
 \int_0^{2\pi}\cot\frac u2\,\sin(mu)\,du=2\pi.
\]
Since $q=\frac12\cot(u/2)$, this and \eqref{eq:real-zero} imply
$j_m=i/2$ for $m>0$.  Complex conjugation, or replacing $m$ by $-m$,
gives $j_m=-i/2$ for $m<0$.
\end{proof}

\begin{proposition}\label{prop:fourier-response}
For $k,\ell\in\Z\setminus\{0\}$,
\begin{equation}\label{eq:hat-response}
 \widehat{\dot L_v}_{k,\ell}
 =\frac{i}{2}\bigl(\sgn k+\sgn\ell\bigr)v_{k+\ell}.
\end{equation}
Hence
\begin{equation}\label{eq:lambda-response}
 \left.\frac{d}{d\varepsilon}\right|_0
 \lambda_{k,\ell}(\varphi_\varepsilon)
 =\frac{i}{2}\sqrt{|k\ell|}
 \bigl(\sgn k+\sgn\ell\bigr)v_{k+\ell}.
\end{equation}
\end{proposition}

\begin{proof}
Set $u=s-t$.  Smoothness of $v$ and the preceding bound imply
\[
 \sup_{0<u<2\pi}\sum_{n\in\mathbb Z}
 |v_n q(u)(e^{inu}-1)|
 \leq\sum_{n\in\mathbb Z}|n|\,|v_n|<\infty.
\]
We may therefore integrate the Fourier series term by term:
\begin{align*}
 \widehat{\dot L_v}_{k,\ell}
 &=\frac{1}{(2\pi)^2}\int_0^{2\pi}\int_0^{2\pi}
 q(u)\,[v(t+u)-v(t)]
 e^{-iku}e^{-i(k+\ell)t}\,du\,dt\\
 &=\sum_{n\in\Z}v_n
 \left(\frac{1}{2\pi}\int_0^{2\pi}
 e^{i(n-k-\ell)t}\,dt\right)
 \left(\frac{1}{2\pi}\int_0^{2\pi}
 q(u)(e^{inu}-1)e^{-iku}\,du\right).
\end{align*}
The $t$-integral forces $n=k+\ell$.  Therefore
\begin{align}
 \widehat{\dot L_v}_{k,\ell}
 &=v_{k+\ell}\frac{1}{2\pi}\int_0^{2\pi}
 q(u)\bigl(e^{i\ell u}-e^{-iku}\bigr)\,du\notag\\
 &=v_{k+\ell}\bigl(j_\ell-j_{-k}\bigr).
 \label{eq:response-J}
\end{align}
Both differences in the last integral cancel the singularity of $q$.
Lemma~\ref{lem:cot-integral} now gives
\[
 j_\ell-j_{-k}
 =\frac{i}{2}\sgn\ell+\frac{i}{2}\sgn k,
\]
which proves \eqref{eq:hat-response}.  Multiplication by
$\sqrt{|k\ell|}$ gives \eqref{eq:lambda-response}.
\end{proof}

\begin{proof}[Proof of Proposition~\ref{prop:det-round-response}]
Part (i) is Proposition~\ref{prop:HS-diff}.  Part (ii) follows from
Proposition~\ref{prop:fourier-response}.  Indeed, for $k,\ell\geq1$,
\eqref{eq:lambda-response} gives
\begin{equation}\label{eq:N-dot}
 (\dot N_v)_{k\ell}=i\sqrt{k\ell}\,v_{k+\ell},
\end{equation}
whereas
\begin{equation}\label{eq:M-dot}
 \left.\frac{d}{d\varepsilon}\right|_0(M_{\varphi_\varepsilon})_{k\ell}
 =\left.\frac{d}{d\varepsilon}\right|_0
 \lambda_{k,-\ell}(\varphi_\varepsilon)
 =\frac{i}{2}\sqrt{k\ell}(1-1)v_{k-\ell}=0.
\end{equation}
This proves \eqref{eq:block-response-main}.  Since $k+\ell\geq2$ in
\eqref{eq:N-dot}, the modes $v_{-1},v_0,v_1$ make no contribution to either
off-diagonal block; hence infinitesimal M\"obius directions are killed.

For part (iii), the block $\dot N_v$ in \eqref{eq:block-response-main} satisfies
\begin{align}
 \langle \dot N_v,\dot N_w\rangle_{\mathrm{HS}}
 &=\sum_{k,\ell\geq1}
 k\ell\,v_{k+\ell}\overline{w_{k+\ell}}\notag\\
 &=\sum_{n=2}^\infty
 \left(\sum_{k=1}^{n-1}k(n-k)\right)
 v_n\overline{w_n}.
 \label{eq:H-HS-1}
\end{align}
The finite sum is
\begin{align}
 \sum_{k=1}^{n-1}k(n-k)
 &=n\sum_{k=1}^{n-1}k-\sum_{k=1}^{n-1}k^2\notag\\
 &=n\frac{n(n-1)}2-\frac{(n-1)n(2n-1)}6
 =\frac{n^3-n}{6}.
 \label{eq:finite-sum}
\end{align}
Thus
\begin{equation}\label{eq:H-HS-2}
 \langle \dot N_v,\dot N_w\rangle_{\mathrm{HS}}
 =\frac16\sum_{n=2}^\infty
 (n^3-n)v_n\overline{w_n}.
\end{equation}
The block formula \eqref{eq:block-response-main} gives
\begin{align}
 (\dot\Lambda_v,\dot\Lambda_w)_{\mathrm{HS},\mathbb R}
 &=2\operatorname{Re}\langle \dot N_v,\dot N_w\rangle_{\mathrm{HS}}\notag\\
 &=\frac13\operatorname{Re}\sum_{n=2}^\infty
 (n^3-n)v_n\overline{w_n}.
 \label{eq:full-HS}
\end{align}
Comparison with \eqref{eq:intro-WP-metric} gives
\eqref{eq:isometry-main}, and taking $w=v$ gives
\eqref{eq:norm-isometry-main}.

For part (iv), take $v\in H_{\mathrm{WP}}$ and smooth horizontal fields
$v_j\to v$ in $H_{\mathrm{WP}}$.  Equation \eqref{eq:norm-isometry-main}
gives
\[
 \|\dot\Lambda_{v_j}-\dot\Lambda_{v_k}\|_{\mathrm{HS}}
 =6^{-1/2}\|v_j-v_k\|_{\mathrm{WP}}\longrightarrow0.
\]
Define $\dot\Lambda_v:=\lim_j\dot\Lambda_{v_j}$ in $S_2$.  The same
estimate applied to two approximating sequences proves independence of
the sequence and uniqueness of the continuous extension.  Passing to
the limit gives
\[
 \|\dot\Lambda_v\|_{\mathrm{HS}}
 =6^{-1/2}\|v\|_{\mathrm{WP}},\qquad
 \dot\Lambda_v=0\ \Longrightarrow\ v=0.
\]
If $\dot\Lambda_{v_j}\to A$ in $S_2$, then
\[
 \|v_j-v_k\|_{\mathrm{WP}}
 =\sqrt6\,\|\dot\Lambda_{v_j}-\dot\Lambda_{v_k}\|_{\mathrm{HS}}
 \longrightarrow0.
\]
Completeness gives $v_j\to v\in H_{\mathrm{WP}}$, and
\[
 \|A-\dot\Lambda_v\|_{\mathrm{HS}}
 \leq\|A-\dot\Lambda_{v_j}\|_{\mathrm{HS}}
      +6^{-1/2}\|v_j-v\|_{\mathrm{WP}}\longrightarrow0.
\]
Thus the range is closed.
\end{proof}

\subsection{Covariance response at a general welding}\label{sec:global}

\subsubsection{The positive composition coordinate}

Let $H_\partial:=H^{1/2}(\mathbb S^1,\mathbb C)/\mathbb C$.
We use the mean-zero representative of each class.  For $u,w\in H_\partial$,
let $u_n,w_n$ be their Fourier coefficients and set
\begin{equation}\label{eq:Hhalf-inner-product}
 \langle u,w\rangle_{H_\partial}
 :=\sum_{n\ne0}|n|u_n\overline{w_n}.
\end{equation}
We identify $H_\partial$ with $\ell^2(\mathbb Z\setminus\{0\})$ using
the orthonormal Fourier basis $e^{in\theta}/\sqrt{|n|}$, $n\ne0$,
with positive modes in the first block and negative modes in the second.
For a quasisymmetric circle homeomorphism $\varphi$, initially define on
smooth functions modulo constants
\begin{equation}\label{eq:Cphi-global}
 C_\varphi u
 =u\circ\varphi-\frac{1}{2\pi}\int_0^{2\pi}
 u\circ\varphi(e^{it})\,dt,
 \qquad u\in C^\infty(\mathbb S^1;\mathbb C)/\mathbb C.
\end{equation}
Quasisymmetry gives a bounded extension to $H_\partial$, with inverse
$C_{\varphi^{-1}}$ \cite{NS95,FVW26}.  In the Fourier basis used above,
\cite[Theorem~1.10]{FVW26} gives
\begin{equation}\label{eq:P-def}
 P_\varphi:=\one-2\Lambda_\varphi=C_\varphi C_\varphi^*.
\end{equation}
For quasisymmetric $\varphi_1,\varphi_2$, their Corollary~1.11 gives
\begin{equation}\label{eq:Lambda-cocycle-global}
 \Lambda_{\varphi_1\circ\varphi_2}
 =C_{\varphi_2}\Lambda_{\varphi_1}C_{\varphi_2}^*
  +\Lambda_{\varphi_2}.
\end{equation}
Using \eqref{eq:P-def}, we obtain
\begin{equation}\label{eq:P-congruence}
 P_{\varphi_1\circ\varphi_2}
 =C_{\varphi_2}(\one-2\Lambda_{\varphi_1})C_{\varphi_2}^*
 =C_{\varphi_2}P_{\varphi_1}C_{\varphi_2}^*.
\end{equation}
For a quasiconformal map $H$, write $\mu_H:=H_{\bar z}/H_z$ for its
Beltrami coefficient.  Write $\operatorname{WP}(\mathbb S^1)$ for the
quasisymmetric maps admitting
a quasiconformal extension $\widehat\varphi:\mathbb D\to\mathbb D$ whose
Beltrami coefficient satisfies
\[
 \int_{\mathbb D}\frac{|\mu_{\widehat\varphi}(z)|^2}
 {(1-|z|^2)^2}\,dA(z)<\infty.
\]
For $\varphi\in\operatorname{WP}(\mathbb S^1)$,
\cite[Theorem~1.5]{FVW26} and \eqref{eq:P-def} give
\begin{gather*}
 P_\varphi-\one=-2\Lambda_\varphi\in S_2,\\
 \shortintertext{and}
 \|C_\varphi^{-1}\|^{-2}\one
 \leq P_\varphi\leq\|C_\varphi\|^2\one.
\end{gather*}
In particular, $P_\varphi$ is positive and boundedly invertible.

The positive operators and their tangent operators below act on the
real-valued subspace of $H_\partial$.  We identify them with their
complex-linear extensions to $H_\partial$, which commute with complex
conjugation.  This identification preserves operator and Hilbert--Schmidt
norms, as well as traces of trace-class operators.  We use the real Hilbert
manifold of positive Hilbert--Schmidt perturbations
\begin{equation}\label{eq:P2-space}
 \mathcal P_2
 =\{P=P^*>0:P^{-1}\text{ is bounded and }P-\one\in S_2\}.
\end{equation}
For $P\in\mathcal P_2$ and self-adjoint Hilbert--Schmidt $A,B$, set
\begin{equation}\label{eq:G-P}
 G_P(A,B)
 =\frac32\operatorname{Re}\Tr(P^{-1}AP^{-1}B).
\end{equation}
Since
\[
 \operatorname{Re}\Tr(P^{-1}AP^{-1}A)
 =\|P^{-1/2}AP^{-1/2}\|_{\mathrm{HS}}^2,
\]
this is a positive definite Hilbert metric on each tangent space.

Let $S$ be bounded and boundedly invertible, preserving the real-valued
subspace.  For
$P\in\mathcal P_2$ and self-adjoint Hilbert--Schmidt operators $A,B$, put
$P'=SPS^*$, $A'=SAS^*$, and $B'=SBS^*$, and assume $P'\in\mathcal P_2$.

\begin{lemma}\label{lem:G-congruence}
\begin{equation}\label{eq:G-invariance}
 G_{P'}(A',B')=G_P(A,B).
\end{equation}
\end{lemma}

\begin{proof}
We have $(P')^{-1}=(S^{-1})^*P^{-1}S^{-1}$ and therefore
\[
 (P')^{-1}A'(P')^{-1}B'
 =(S^{-1})^*P^{-1}AP^{-1}BS^*.
\]
Writing $\|\cdot\|_{\mathrm{tr}}$ for the trace norm, the ideal inequality
gives
\[
 \|P^{-1}AP^{-1}B\|_{\mathrm{tr}}
 \leq\|P^{-1}\|^2\|A\|_{\mathrm{HS}}\|B\|_{\mathrm{HS}}<\infty.
\]
Cyclicity of the trace therefore gives
\[
 \Tr((S^{-1})^*P^{-1}AP^{-1}BS^*)
 =\Tr(P^{-1}AP^{-1}B),
\]
which proves the claim.
\end{proof}

\subsubsection{Response at a general Weil--Petersson welding}

We view the universal Weil--Petersson Teichm\"uller space as
the left M\"obius quotient
$T_0(1):=\operatorname{Mob}(\mathbb S^1)\backslash
\operatorname{WP}(\mathbb S^1)$, where
$\operatorname{Mob}(\mathbb S^1)$ denotes the orientation-preserving circle
M\"obius transformations.  At the identity, its real tangent Hilbert space is
$H_{\mathrm{WP}}$.  Let $g_{\mathrm{WP}}$ be the right-invariant
Weil--Petersson metric, normalized by \eqref{eq:intro-WP-metric}
\cite{Nag92,TT06}.  For a smooth real field $v$, let $\Psi_t$ be
the lifted flow determined by
\begin{equation}\label{eq:wp-flow}
 \partial_t\Psi_t(\theta)=v(\Psi_t(\theta)),
 \qquad \Psi_0(\theta)=\theta,
 \qquad \psi_t^v(e^{i\theta})=e^{i\Psi_t(\theta)}.
\end{equation}
Differentiation in time gives
$\partial_t^2\Psi_t=(v'v)\circ\Psi_t$.  Hence
\[
 \Psi_t-\id-tv=\int_0^t(t-r)(v'v)\circ\Psi_r\,dr.
\]
For each integer $m\geq0$, the functions $(v'v)\circ\Psi_r$ have bounded
$C^m$ norms on compact time intervals, so
\begin{equation}\label{eq:wp-flow-Cm-expansion}
 \|\Psi_t-\id-tv\|_{C^m(\T)}
 \leq\frac{t^2}{2}\sup_{|r|\leq|t|}
       \|(v'v)\circ\Psi_r\|_{C^m(\T)}
 =O_{v,m}(t^2).
\end{equation}
For a fixed representative $\varphi\in\operatorname{WP}(\mathbb S^1)$
and a smooth horizontal $v$, put
\begin{equation}\label{eq:right-tangent}
 X_{\varphi,v}
 =\left.\frac{d}{dt}\right|_{0}[\psi_t^v\circ\varphi].
\end{equation}
For right translation $R_\varphi(\psi):=\psi\circ\varphi$ and smooth
horizontal $v,w$, these vectors satisfy
\begin{align*}
 X_{\varphi,v}&=(dR_\varphi)_{\id}v,\\[3pt]
 g_{\mathrm{WP},[\varphi]}(X_{\varphi,v},X_{\varphi,w})
 &=(v,w)_{\mathrm{WP}}.
\end{align*}
The derivative in this subsection follows $\psi_t^v\circ\varphi$.
The right action \eqref{eq:intro-right-response} on random weldings follows
$\varphi\circ\psi_t^v$; the two paths have the same tangent at $\varphi=\id$.

For a tangent path in \eqref{eq:right-tangent}, write
$(d\Lambda)_\varphi(X_{\varphi,v})$
for its derivative in Hilbert--Schmidt norm when this derivative exists,
and write $(dP)_\varphi=-2(d\Lambda)_\varphi$.

\begin{proposition}\label{prop:det-global-response}
For every smooth horizontal $v$, the Hilbert--Schmidt derivative at the
fixed welding $\varphi$ satisfies
\begin{equation}\label{eq:global-response-formula}
 (d\Lambda)_\varphi(X_{\varphi,v})
 =\left.\frac{d}{dt}\right|_{0}
 \Lambda_{\psi_t^v\circ\varphi}
 =C_\varphi\dot\Lambda_v C_\varphi^*.
\end{equation}
Equivalently,
\begin{equation}\label{eq:global-P-response}
 (dP)_\varphi(X_{\varphi,v})
 =-2C_\varphi\dot\Lambda_vC_\varphi^*.
\end{equation}
For smooth horizontal tangent vectors $X,Y$ at $[\varphi]$,
\begin{equation}\label{eq:global-WP-P}
 g_{\mathrm{WP},[\varphi]}(X,Y)
 =G_{P_\varphi}((dP)_\varphi X,(dP)_\varphi Y),
\end{equation}
or, equivalently,
\begin{equation}\label{eq:global-WP-Lambda}
 g_{\mathrm{WP},[\varphi]}(X,Y)
 =6\operatorname{Re}\Tr\!\left(
 P_\varphi^{-1}(d\Lambda)_\varphi(X)
 P_\varphi^{-1}(d\Lambda)_\varphi(Y)\right).
\end{equation}
The response has a unique continuous extension from smooth horizontal
vectors to the full tangent Hilbert space
\begin{equation}\label{eq:global-response-extension}
 (d\Lambda)_\varphi:T_{[\varphi]}T_0(1)\longrightarrow S_2^{\mathrm{sa}},
\end{equation}
and its range is closed in $S_2^{\mathrm{sa}}$.
\end{proposition}

\begin{proof}
For a circle M\"obius transformation $m$, \cite[Theorem~1.5]{FVW26}
gives $\Lambda_m=0$.  Thus \eqref{eq:Lambda-cocycle-global} yields
\[
 \Lambda_{m\circ\varphi}
 =C_\varphi\Lambda_mC_\varphi^*+\Lambda_\varphi=\Lambda_\varphi,
 \qquad P_{m\circ\varphi}=P_\varphi.
\]
Both operators are therefore defined on the left M\"obius quotient.

Apply \eqref{eq:Lambda-cocycle-global} with
$\varphi_1=\psi_t^v$ and $\varphi_2=\varphi$:
\begin{equation}\label{eq:global-diff-quotient}
 \frac{\Lambda_{\psi_t^v\circ\varphi}-\Lambda_\varphi}{t}
 =C_\varphi\frac{\Lambda_{\psi_t^v}}{t}C_\varphi^*.
\end{equation}
By \eqref{eq:wp-flow-Cm-expansion} and
Lemma~\ref{lem:flow-HS-diff},
\begin{equation}\label{eq:global-flow-S2-limit}
 \left\|\frac{\Lambda_{\psi_t^v}}t-\dot\Lambda_v\right\|_{\mathrm{HS}}
 =O_v(|t|).
\end{equation}
The operator ideal estimate now gives
\begin{align*}
 &\left\|
 \frac{\Lambda_{\psi_t^v\circ\varphi}-\Lambda_\varphi}{t}
       -C_\varphi\dot\Lambda_vC_\varphi^*
 \right\|_{\mathrm{HS}}\\
 &\hspace{2em}\leq\|C_\varphi\|^2
       \left\|\frac{\Lambda_{\psi_t^v}}t-\dot\Lambda_v\right\|_{\mathrm{HS}}
 =O_{\varphi,v}(|t|).
\end{align*}
This proves \eqref{eq:global-response-formula};
\eqref{eq:global-P-response} follows from $P=\one-2\Lambda$.

Now take $X=X_{\varphi,v}$ and $Y=X_{\varphi,w}$.  Equation
\eqref{eq:P-congruence} gives
\[
 P_\varphi=C_\varphi P_{\id}C_\varphi^*,\qquad P_{\id}=\one,
\]
and \eqref{eq:global-P-response} gives
\begin{align*}
 (dP)_\varphi X&=C_\varphi(-2\dot\Lambda_v)C_\varphi^*,\\[3pt]
 (dP)_\varphi Y&=C_\varphi(-2\dot\Lambda_w)C_\varphi^*.
\end{align*}
Lemma~\ref{lem:G-congruence}, followed by the round identity
\eqref{eq:isometry-main}, yields
\begin{align*}
 G_{P_\varphi}((dP)_\varphi X,(dP)_\varphi Y)
 &=G_{\one}(-2\dot\Lambda_v,-2\dot\Lambda_w)\\
 &=6(\dot\Lambda_v,\dot\Lambda_w)_{\mathrm{HS},\mathbb R}\\
 &=(v,w)_{\mathrm{WP}}.
\end{align*}
Right invariance of the Weil--Petersson metric identifies the last quantity
with $g_{\mathrm{WP},[\varphi]}(X,Y)$, proving
\eqref{eq:global-WP-P}.  Substituting $dP=-2d\Lambda$ into
\eqref{eq:global-WP-P} gives \eqref{eq:global-WP-Lambda}.

It remains to pass from smooth vectors to the tangent completion.  Define
\[
 \|A\|_{\varphi}
 =\|P_\varphi^{-1/2}AP_\varphi^{-1/2}\|_{\mathrm{HS}}.
\]
Because $P_\varphi$ and $P_\varphi^{-1}$ are bounded,
\begin{equation}\label{eq:moving-HS-equivalence}
 \|P_\varphi\|^{-1}\|A\|_{\mathrm{HS}}
 \leq \|A\|_\varphi
 \leq \|P_\varphi^{-1}\|\|A\|_{\mathrm{HS}}.
\end{equation}
For smooth horizontal $X$, \eqref{eq:global-WP-Lambda} gives
\begin{equation}\label{eq:global-norm-response}
 \|X\|_{\mathrm{WP}}^2
 =6\|(d\Lambda)_\varphi X\|_\varphi^2.
\end{equation}
For an arbitrary tangent vector $X$, choose
smooth horizontal $X_n\to X$ in the Weil--Petersson norm.  Then
\[
 \|(d\Lambda)_\varphi X_n-(d\Lambda)_\varphi X_m\|_{\mathrm{HS}}
 \leq\frac{\|P_\varphi\|}{\sqrt6}\|X_n-X_m\|_{\mathrm{WP}}
 \longrightarrow0.
\]
The limit defines the unique continuous extension.
Equation \eqref{eq:global-norm-response} passes to the limit and proves
injectivity.  Conversely, if $(d\Lambda)_\varphi X_n\to A$ in $S_2$, then
\[
 \|X_n-X_m\|_{\mathrm{WP}}
 \leq\sqrt6\,\|P_\varphi^{-1}\|
       \|(d\Lambda)_\varphi X_n-(d\Lambda)_\varphi X_m\|_{\mathrm{HS}}
 \longrightarrow0.
\]
Hence $X_n\to X$ in the tangent Hilbert space, and
\[
 \|A-(d\Lambda)_\varphi X\|_{\mathrm{HS}}
 \leq\|A-(d\Lambda)_\varphi X_n\|_{\mathrm{HS}}
      +\frac{\|P_\varphi\|}{\sqrt6}\|X_n-X\|_{\mathrm{WP}}
 \longrightarrow0.
\]
Thus the range is closed.
\end{proof}

\subsubsection{Covariance interpretation}

The normalized composition operator in \eqref{eq:Cphi-global} is the same
pullback operator denoted $\Pi(\varphi)$ by Fan--Sung \cite{FS25}.
For quasisymmetric $\varphi$, Fan--Sung's Gaussian equivalence criterion
\cite[Theorem~3.3]{FS25}, with zero mean shift, is $P_\varphi-\one\in S_2$.
The derivative $(dP)_{\id}v=-2\dot\Lambda_v$ therefore identifies the
covariance tangent, with its Weil--Petersson normalization given by
Proposition~\ref{prop:det-round-response}.

\subsection{K\"ahler and Fisher calibration}
\label{sec:action-kahler}

On $H_\partial$, define the Fourier Hilbert transform by
\begin{equation}\label{eq:det-Hilbert-transform}
 J_\partial e^{in\theta}
 :=i\,\sgn(n)e^{in\theta},
 \qquad n\in\mathbb Z\setminus\{0\}.
\end{equation}
The normalized composition operator $C_\varphi$ in
\eqref{eq:Cphi-global} is symplectic on the real form of $H_\partial$:
\begin{equation}\label{eq:det-composition-symplectic}
 C_\varphi^*J_\partial C_\varphi=J_\partial.
\end{equation}
This is the composition-operator realization of the universal period map
\cite{NS95,TT06}.  Invertibility of $C_\varphi$ and $J_\partial^{-1}=-J_\partial$
also give $C_\varphi J_\partial C_\varphi^*=J_\partial$.  Hence
\begin{equation}\label{eq:det-P-symplectic}
 P_\varphi J_\partial P_\varphi
 =C_\varphi(C_\varphi^*J_\partial C_\varphi)C_\varphi^*
 =C_\varphi J_\partial C_\varphi^*=J_\partial.
\end{equation}

Define the restricted positive symplectic cone by
$\mathcal P_2^{\mathrm{sp}}:=\{P\in\mathcal P_2:PJ_\partial P=J_\partial\}$.
Its tangent space at $P$ is
\begin{equation}\label{eq:det-cone-tangent}
 T_P\mathcal P_2^{\mathrm{sp}}
 =\{A\in S_2^{\mathrm{sa}}:
 AJ_\partial P+PJ_\partial A=0\}.
\end{equation}
The inclusion from left to right follows by differentiating
$P_tJ_\partial P_t=J_\partial$.  For the converse, functional calculus gives
\[
 J_\partial P^s=P^{-s}J_\partial,\qquad s\in\mathbb R.
\]
If $A\in S_2^{\mathrm{sa}}$ satisfies the constraint in
\eqref{eq:det-cone-tangent}, set $K=P^{-1/2}AP^{-1/2}$.  Then
\[
 KJ_\partial+J_\partial K
 =P^{-1/2}(AJ_\partial P+PJ_\partial A)P^{-1/2}=0.
\]
The curve $P_t=P^{1/2}e^{tK}P^{1/2}$ satisfies
\[
 \begin{gathered}
 P_0=P,\qquad \dot P_0=A,\\[3pt]
 P_tJ_\partial P_t
 =P^{1/2}e^{tK}J_\partial e^{tK}P^{1/2}=J_\partial.
 \end{gathered}
\]
The remainder satisfies
\[
 \|P_t-P-tA\|_{\mathrm{HS}}
 \leq\|P\|\,\|e^{tK}-\one-tK\|_{\mathrm{HS}}
 \leq\frac{t^2}{2}\|P\|e^{|t|\|K\|}\|K\|\,\|K\|_{\mathrm{HS}}.
\]
Thus $P_t\in\mathcal P_2^{\mathrm{sp}}$, proving
\eqref{eq:det-cone-tangent}.

For $A,B\in T_P\mathcal P_2^{\mathrm{sp}}$, put
\begin{align}
 J_PA&:=PJ_\partial A,
 \label{eq:det-target-complex}\\
 \Omega_P(A,B)&:=G_P(J_PA,B).
 \label{eq:det-target-symplectic}
\end{align}
The identities $J_\partial^*=-J_\partial$ and
$PJ_\partial P=J_\partial$ give
\begin{gather*}
 (PJ_\partial A)^*=-AJ_\partial P=PJ_\partial A,\\
 \shortintertext{and}
 (PJ_\partial A)J_\partial P+PJ_\partial(PJ_\partial A)
 =P J_\partial A J_\partial P-A=0.
\end{gather*}
For the last equality, multiply $AJ_\partial P=-PJ_\partial A$ by
$PJ_\partial$ on the left.  Thus $J_P$ preserves the tangent space.
Moreover,
\[
 J_P^2A=PJ_\partial PJ_\partial A=-A,\qquad
 J_\partial A J_\partial=P^{-1}AP^{-1},
\]
and hence
\begin{align*}
 G_P(J_PA,J_PB)
 &=\frac32\operatorname{Re}\Tr(J_\partial A J_\partial B)
 =G_P(A,B),\\
 \shortintertext{and}
 G_P(J_PA,B)
 &=G_P(J_P^2A,J_PB)=-G_P(A,J_PB).
\end{align*}
Together these identities give
\begin{equation}\label{eq:det-target-Kahler-identities}
 \begin{gathered}
 J_P^2=-\one,\qquad G_P(J_PA,J_PB)=G_P(A,B),\\[3pt]
 \Omega_P(A,B)=-\Omega_P(B,A).
 \end{gathered}
\end{equation}
Transport $J_0$ from \eqref{eq:intro-complex-structure} by right translation:
$J_0X_{\varphi,v}:=X_{\varphi,J_0v}$.  On the source tangent space, set
\begin{equation}\label{eq:det-source-symplectic}
 \omega_{\mathrm{WP},[\varphi]}(X,Y)
 :=g_{\mathrm{WP},[\varphi]}(J_0X,Y).
\end{equation}

\begin{proposition}
\label{prop:det-Kahler-response}
For every $\varphi\in\operatorname{WP}(\mathbb S^1)$ and all
Weil--Petersson tangent vectors $X,Y$ at $[\varphi]$,
\begin{align}
 (dP)_\varphi(J_0X)
 &=J_{P_\varphi}(dP)_\varphi(X),
 \label{eq:det-complex-intertwining}\\
 g_{\mathrm{WP},[\varphi]}(X,Y)
 &=G_{P_\varphi}((dP)_\varphi X,(dP)_\varphi Y),
 \label{eq:det-metric-pullback}\\
 \shortintertext{and}
 \omega_{\mathrm{WP},[\varphi]}(X,Y)
 &=\Omega_{P_\varphi}((dP)_\varphi X,(dP)_\varphi Y).
 \label{eq:det-symplectic-pullback}
\end{align}
\end{proposition}

\begin{proof}
For a smooth horizontal field $v$, put $\dot P_v=-2\dot\Lambda_v$.
Since $(J_0v)_n=iv_n$ for $n>0$, the block response
\eqref{eq:block-response-main} gives
\[
 \dot N_{J_0v}=i\dot N_v,\qquad
 J_\partial=\begin{pmatrix}i\one&0\\0&-i\one\end{pmatrix}.
\]
Therefore
\begin{equation}\label{eq:det-round-complex-response}
 \dot P_{J_0v}
 =-2\begin{pmatrix}0&i\dot N_v\\-i\overline{\dot N_v}&0\end{pmatrix}
 =J_\partial\dot P_v.
\end{equation}
Equations \eqref{eq:global-P-response},
\eqref{eq:det-composition-symplectic}, and
\eqref{eq:det-round-complex-response} yield
\begin{align}
 (dP)_\varphi(J_0X_{\varphi,v})
 &=C_\varphi J_\partial\dot P_vC_\varphi^*\notag\\
 &=C_\varphi C_\varphi^*J_\partial
   C_\varphi\dot P_vC_\varphi^*
 =J_{P_\varphi}(dP)_\varphi(X_{\varphi,v}).
 \label{eq:det-complex-response-proof}
\end{align}
The metric identity is \eqref{eq:global-WP-P}.  Thus, for smooth $X,Y$,
\begin{align*}
 \Omega_{P_\varphi}((dP)_\varphi X,(dP)_\varphi Y)
 &=G_{P_\varphi}(J_{P_\varphi}(dP)_\varphi X,(dP)_\varphi Y)\\
 &=G_{P_\varphi}((dP)_\varphi(J_0X),(dP)_\varphi Y)\\
 &=g_{\mathrm{WP},[\varphi]}(J_0X,Y)
 =\omega_{\mathrm{WP},[\varphi]}(X,Y).
\end{align*}
Proposition~\ref{prop:det-global-response}, the isometry of $J_0$, and
\[
 \|J_PA\|_{\mathrm{HS}}\leq\|P\|\,\|A\|_{\mathrm{HS}}
\]
extend all three identities to the tangent completion.
\end{proof}

Let $H_\partial^+$ and $H_\partial^-$ be the closed subspaces spanned by
the positive and negative Fourier modes, respectively.  For any bounded,
positive, invertible $P$ satisfying $PJ_\partial P=J_\partial$, define
the Cayley operator
\begin{equation}\label{eq:det-Cayley-operator}
 Y_P:=(P-\one)(P+\one)^{-1}.
\end{equation}
The spectral theorem and $J_\partial P=P^{-1}J_\partial$ give
\[
 Y_P^*=Y_P,\qquad
 \|Y_P\|=\max_{\lambda\in\operatorname{spec}(P)}
 \left|\frac{\lambda-1}{\lambda+1}\right|<1,
\]
and
\[
 J_\partial Y_P
 =(P^{-1}-\one)(P^{-1}+\one)^{-1}J_\partial
 =-Y_PJ_\partial.
\]
Thus $Y_P$ has the block form
\begin{equation}\label{eq:det-Cayley-block}
 Y_P=
 \begin{pmatrix}0&Z_P\\ Z_P^*&0\end{pmatrix}
 \quad\text{on }H_\partial^+\oplus H_\partial^-,
 \qquad Z_P:H_\partial^-\longrightarrow H_\partial^+.
\end{equation}
Since $P-\one=2Y_P(\one-Y_P)^{-1}$, the Hilbert--Schmidt ideal property gives
\begin{gather*}
 P-\one\in S_2
 \quad\Longleftrightarrow\quad Y_P\in S_2
 \quad\Longleftrightarrow\quad Z_P\in S_2,\\
 \shortintertext{with}
 \|Y_P\|_{\mathrm{HS}}^2=2\|Z_P\|_{\mathrm{HS}}^2.
\end{gather*}
For $\varphi\in\operatorname{WP}(\mathbb S^1)$, the block $Z_{P_\varphi}$
is the restricted period coordinate of Kirillov--Yuriev--Nag--Sullivan
\cite{Nag92,NS95,TT06}.  Let $D_{\mathrm{FVW}}$ be the nonnegative diagonal
operator on $\ell^2(\mathbb N)$ whose entries are the reciprocals of the positive Fredholm
eigenvalues of the welded curve, as in \cite[Theorem~1.6]{FVW26}.
The symmetric-welding hypothesis of that theorem follows here from
$\Lambda_\varphi\in S_2$: this operator is compact, so
\cite[Theorem~1.5]{FVW26} implies that $\varphi$ is symmetric.  That is,
an increasing lift $a$ of $\varphi$ satisfies
\[
 \lim_{t\downarrow0}\sup_{\theta\in\mathbb R}
 \left|\frac{a(\theta+t)-a(\theta)}
 {a(\theta)-a(\theta-t)}-1\right|=0.
\]
The Weil--Petersson case of their Theorem~1.6 gives
$D_{\mathrm{FVW}}\in S_2$ and $\|D_{\mathrm{FVW}}\|<1$.
Writing $\simeq$ for unitary equivalence, it follows from
$P_\varphi=2(\one-\Lambda_\varphi)-\one$ that
\begin{align}
 P_\varphi
 &\simeq
 \begin{pmatrix}
  (\one+D_{\mathrm{FVW}})(\one-D_{\mathrm{FVW}})^{-1}&0\\
  0&(\one-D_{\mathrm{FVW}})(\one+D_{\mathrm{FVW}})^{-1}
 \end{pmatrix},
 \label{eq:det-FVW-P-diagonalization}\\
 \shortintertext{and hence}
 Y_{P_\varphi}
 &\simeq
 \begin{pmatrix}D_{\mathrm{FVW}}&0\\0&-D_{\mathrm{FVW}}\end{pmatrix}.
 \label{eq:det-FVW-Cayley-diagonalization}
\end{align}
On the other hand,
\[
 Y_{P_\varphi}^2
 =\begin{pmatrix}
 Z_{P_\varphi}Z_{P_\varphi}^*&0\\
 0&Z_{P_\varphi}^*Z_{P_\varphi}
 \end{pmatrix}.
\]
For every integer $m\geq1$, cyclicity of the trace now yields
\[
 \Tr\bigl((Z_{P_\varphi}Z_{P_\varphi}^*)^m\bigr)
 =\frac12\Tr(Y_{P_\varphi}^{2m})
 =\Tr(D_{\mathrm{FVW}}^{2m}).
\]
Write $I^L(\varphi)$ for the Loewner energy of the welded curve, and
$\det$ for the Fredholm determinant of a trace-class perturbation of the
identity.  Since $Z_{P_\varphi}\in S_2$ and $\|Z_{P_\varphi}\|<1$,
\begin{align*}
 -\log\det(\one-Z_{P_\varphi}Z_{P_\varphi}^*)
 &=\sum_{m=1}^\infty\frac1m
       \Tr\bigl((Z_{P_\varphi}Z_{P_\varphi}^*)^m\bigr)\\
 &=\sum_{m=1}^\infty\frac1m\Tr(D_{\mathrm{FVW}}^{2m})
 =-\log\det(\one-D_{\mathrm{FVW}}^2).
\end{align*}
The energy formula in \cite[Theorem~1.6]{FVW26} therefore gives
\begin{equation}\label{eq:det-action-potential}
 I^L(\varphi)
 =-12\log\det(\one-Z_{P_\varphi}Z_{P_\varphi}^*).
\end{equation}
For an arbitrary quasisymmetric welding, the finite-energy characterization
of \cite{Wang19}, \cite[Theorem~1.5]{FVW26}, and the Cayley identities above
give
\begin{equation}\label{eq:det-finite-energy-equivalence}
 I^L(\varphi)<\infty
 \Longleftrightarrow\varphi\in\operatorname{WP}(\mathbb S^1)
 \Longleftrightarrow P_\varphi-\one\in S_2
 \Longleftrightarrow Z_{P_\varphi}\in S_2.
\end{equation}

Equip the real-valued subspace of $H_\partial$ with
$(u,w)_{H_\partial}:=\operatorname{Re}\langle u,w\rangle_{H_\partial}$.
On an auxiliary probability space $(\Omega_G,\mathcal F_G,\mathbb P_G)$,
let $W$ be an isonormal Gaussian field, viewed as a real-linear map from
this Hilbert space into $L^2(\mathbb P_G)$.  Its finite collections are
jointly Gaussian with mean zero and, with expectation under $\mathbb P_G$,
\[
 \mathbb E[W(u)W(w)]=(u,w)_{H_\partial}.
\]
Choose a real orthonormal basis $(u_j)_{j\geq1}$, write $W_j:=W(u_j)$,
and let $\delta_{ij}$ be the Kronecker delta.
For a self-adjoint Hilbert--Schmidt operator $K$ on this real space,
let $K^{(N)}$ be its compression to $\operatorname{span}\{u_1,\ldots,u_N\}$,
extended by zero on the orthogonal complement, and set
\[
 I_{2,N}(K):=\sum_{i,j=1}^N
 (Ku_i,u_j)_{H_\partial}(W_iW_j-\delta_{ij}).
\]
For another such operator $L$, the Gaussian fourth-moment identity
\[
 \mathbb E[(W_iW_j-\delta_{ij})(W_kW_l-\delta_{kl})]
 =\delta_{ik}\delta_{jl}+\delta_{il}\delta_{jk}
\]
gives
\begin{align*}
 \mathbb E[I_{2,N}(K)I_{2,M}(L)]
 &=2\Tr(K^{(N)}L^{(M)}),\\
 \shortintertext{and hence}
 \|I_{2,N}(K)-I_{2,M}(K)\|_{L^2(\mathbb P_G)}^2
 &=2\|K^{(N)}-K^{(M)}\|_{\mathrm{HS}}^2\longrightarrow0
 \quad(N,M\to\infty).
\end{align*}
Consequently the second Wiener integral is well defined by
\[
 I_2(K):=\lim_{N\to\infty}I_{2,N}(K)
 \quad\text{in }L^2(\mathbb P_G).
\]
Its first two moments satisfy
\[
 \mathbb E[I_2(K)]=0,\qquad
 \mathbb E[I_2(K)I_2(L)]=2\Tr(KL).
\]
To check basis independence, fix a real-valued $a\in H_\partial$, take
$Ku=(u,a)_{H_\partial}a$, and put
$a_N=\sum_{j=1}^N(a,u_j)_{H_\partial}u_j$.  Then
\[
 I_{2,N}(K)=W(a_N)^2-\|a_N\|_{H_\partial}^2
 \longrightarrow W(a)^2-\|a\|_{H_\partial}^2
 \quad\text{in }L^2(\mathbb P_G),
\]
because the squared $L^2$ distance is
\[
 2\bigl(\|a_N\|_{H_\partial}^4+\|a\|_{H_\partial}^4
       -2(a_N,a)_{H_\partial}^2\bigr)\longrightarrow0.
\]
The spectral decomposition gives basis independence for finite-rank
self-adjoint operators.  Their density in $S_2^{\mathrm{sa}}$, together with
\[
 \|I_2(K)-I_2(L)\|_{L^2(\mathbb P_G)}
 =\sqrt2\,\|K-L\|_{\mathrm{HS}},
\]
gives the same conclusion for every self-adjoint Hilbert--Schmidt $K$.

For $P\in\mathcal P_2^{\mathrm{sp}}$ and $A\in T_P\mathcal P_2^{\mathrm{sp}}$,
put
\begin{equation}\label{eq:det-covariance-score}
 S_{P,A}^{\mathrm{cov}}:=\frac12 I_2(P^{-1/2}AP^{-1/2}).
\end{equation}
For $A,B\in T_P\mathcal P_2^{\mathrm{sp}}$, the covariance identity and
cyclicity of the trace give
\begin{align}
 \mathbb E[S_{P,A}^{\mathrm{cov}}S_{P,B}^{\mathrm{cov}}]
 &=\frac12\operatorname{Re}\Tr\bigl(
       (P^{-1/2}AP^{-1/2})(P^{-1/2}BP^{-1/2})\bigr)\notag\\
 &=\frac12\operatorname{Re}\Tr(P^{-1}AP^{-1}B)
 =\frac13G_P(A,B).
 \label{eq:det-Gaussian-Fisher}
\end{align}
For $\varphi\in\operatorname{WP}(\mathbb S^1)$ and
$X,Y\in T_{[\varphi]}T_0(1)$, combining \eqref{eq:det-Gaussian-Fisher}
with Proposition~\ref{prop:det-Kahler-response} yields
\begin{align}
 g_{\mathrm{WP},[\varphi]}(X,Y)
 &=3\mathbb E[
 S_{P_\varphi,(dP)_\varphi X}^{\mathrm{cov}}
 S_{P_\varphi,(dP)_\varphi Y}^{\mathrm{cov}}],
 \label{eq:det-Fisher-pullback}\\
 \shortintertext{and}
 \omega_{\mathrm{WP},[\varphi]}(X,Y)
 &=3\mathbb E[
 S_{P_\varphi,J_{P_\varphi}(dP)_\varphi X}^{\mathrm{cov}}
 S_{P_\varphi,(dP)_\varphi Y}^{\mathrm{cov}}].
 \label{eq:det-Fisher-symplectic}
\end{align}

\section{A pointwise response core on rough SLE weldings}
\label{sec:sle-closed}

\subsection{The welding law and horizontal flows}

For deterministic continuity statements, use on $\mathcal W$ the
Carath\'eodory topology of joint local uniform convergence of
$(f_h,g_h)$.  Left and right
composition by a fixed quasisymmetric homeomorphism are continuous in
this topology \cite[Proposition~4.3]{FS25}.
Rotation invariance of the loop shape and the independent uniform
exterior mark make $\widetilde\nu_\kappa$ invariant under left and right
rotations.

The real Hilbert space $H_{\mathrm{WP}}$ has the orthonormal Fourier basis
\begin{equation}\label{eq:B1-WP-basis}
 e_{n,c}(\theta)=\sqrt{\frac{2}{n^3-n}}\cos(n\theta),\qquad
 e_{n,s}(\theta)=\sqrt{\frac{2}{n^3-n}}\sin(n\theta),
 \qquad n\ge2.
\end{equation}
Its finite real span is $H_{\mathrm{WP}}^{\mathrm{fin}}$.

Let $\varphi$ be a quasisymmetric circle homeomorphism, and let
$v_1,\ldots,v_\ell\in H_{\mathrm{WP}}^{\mathrm{fin}}$ and
$t_1,\ldots,t_\ell\in\mathbb R$, with $\ell\geq1$.

\begin{lemma}\label{lem:B1-unique-welding-flow-stability}
For every $h\in\mathcal W$,
\begin{equation}\label{eq:B1-welding-stability}
 h\circ\varphi\in\mathcal W,
\end{equation}
and, in particular,
\begin{equation}\label{eq:B1-finite-composition-stability}
 h\circ\psi_{t_1}^{v_1}\circ\cdots\circ\psi_{t_\ell}^{v_\ell}
 \in\mathcal W.
\end{equation}
\end{lemma}

\begin{proof}
The first assertion is \cite[Lemma~4.1]{FS25}.  Each $v_j$ is smooth,
so its flow is a smooth circle diffeomorphism and hence quasisymmetric.
Their finite composition is quasisymmetric; applying the same lemma
proves \eqref{eq:B1-finite-composition-stability}.
\end{proof}

For adjoint identities we use the complexifications of the real Hilbert
spaces, with the same notation and the pairing
$\langle F,G\rangle=\int F\overline G\,d\widetilde\nu_\kappa$.

\subsection{A bounded finite-energy core on rough weldings}

For $J_{a,b}$ from \eqref{eq:intro-averaged-observable}, fix
$v\in C^\infty(\T;\mathbb R)$.  All suprema in this subsection
are over $h\in\mathcal W$.  We use normalized angular measure
$d\theta/(2\pi)$ in the $L^2$ and $H^1$ norms on $\T$, and set
\[
 C_{a,b}:=\|b\|_\infty
 \bigl(\|a'\|_\infty^2+\|a\|_\infty^2\bigr)^{1/2}.
\]

\begin{lemma}
\label{lem:B1-J-response}
For every $h\in\mathcal W$,
\begin{equation}\label{eq:B1-J-response}
 D_vJ_{a,b}(h)
 =-\frac{1}{2\pi}\int_0^{2\pi}
 (av)'(\theta)b(h(e^{i\theta}))\,d\theta.
\end{equation}
Moreover,
\begin{equation}\label{eq:B1-uniform-dq}
 \lim_{t\to0}\sup_h
 \left|
 \frac{J_{a,b}(T_t^v h)-J_{a,b}(h)}{t}-D_vJ_{a,b}(h)
 \right|=0.
\end{equation}
For horizontal $v$,
\begin{equation}\label{eq:B1-J-WP-bound}
 \sup_h|D_vJ_{a,b}(h)|\le C_{a,b}\|v\|_{\mathrm{WP}}.
\end{equation}
\end{lemma}

\begin{proof}
For the lift $\Psi_t$ in \eqref{eq:wp-flow},
\eqref{eq:wp-flow-Cm-expansion} gives
\begin{equation}\label{eq:B1-flow-expansion}
 \Psi_t(\theta)=\theta+t v(\theta)+O(t^2)
 \quad\text{in }C^1(\T),
\end{equation}
and therefore
\begin{equation}\label{eq:B1-inverse-expansion}
 \begin{aligned}
 \Psi_t^{-1}(u)&=u-tv(u)+O(t^2),\\[3pt]
 (\Psi_t^{-1})'(u)&=1-tv'(u)+O(t^2),
 \end{aligned}
\end{equation}
uniformly in $u$.  Changing variables $u=\Psi_t(\theta)$ gives
\begin{align}
 J_{a,b}(T_t^v h)
 &=\frac{1}{2\pi}\int_0^{2\pi}
 a(\Psi_t^{-1}(u))(\Psi_t^{-1})'(u)
 b(h(e^{iu}))\,du.\label{eq:B1-change-var}
\end{align}
Taylor expansion of the smooth factor yields
\begin{equation}\label{eq:B1-weight-expansion}
 \begin{aligned}
 a(\Psi_t^{-1}(u))(\Psi_t^{-1})'(u)
 &=a(u)-t\{a'(u)v(u)+a(u)v'(u)\}+O(t^2)\\
 &=a(u)-t(av)'(u)+O(t^2)
 \end{aligned}
\end{equation}
uniformly in $u$.  Thus, for a constant $C_{a,v}$ and all sufficiently
small $t$,
\[
 \sup_{h\in\mathcal W}
 \left|J_{a,b}(T_t^vh)-J_{a,b}(h)
       +\frac{t}{2\pi}\int_0^{2\pi}(av)'(u)b(h(e^{iu}))\,du\right|
 \leq C_{a,v}\|b\|_\infty t^2.
\]
Dividing by $t$ proves \eqref{eq:B1-J-response} and
\eqref{eq:B1-uniform-dq}.

For the norm estimate,
\begin{align}
 |D_vJ_{a,b}(h)|
 &\le \|b\|_\infty
 \bigl(\|a'\|_\infty\|v\|_{L^2}
      +\|a\|_\infty\|v'\|_{L^2}\bigr)\notag\\
 &\le C_{a,b}\|v\|_{H^1}.\label{eq:B1-H1-bound}
\end{align}
For a horizontal real field $v$, Parseval's identity and
$1+n^2\le n^3-n$ for $n\ge2$ give
\begin{align}
 \|v\|_{H^1}^2
 &=2\sum_{n\ge2}(1+n^2)|v_n|^2\notag\\
 &\le2\sum_{n\ge2}(n^3-n)|v_n|^2
 =\|v\|_{\mathrm{WP}}^2.\label{eq:B1-H1-WP}
\end{align}
Combining \eqref{eq:B1-H1-bound} and \eqref{eq:B1-H1-WP} proves
\eqref{eq:B1-J-WP-bound}.
\end{proof}

Since $J_{a,b}$ is bounded, a smooth cutoff equal to the identity on
its range gives $J_{a,b}\in\mathcal C$.  For a cylinder $F$ represented
as in \eqref{eq:intro-core}, write $J_r=J_{a^{(r)},b^{(r)}}$ and set
\[
 C_F:=\sum_{r=1}^m\|\partial_r\phi\|_\infty C_{a^{(r)},b^{(r)}}.
\]
The chain rule and Lemma~\ref{lem:B1-J-response} give
\begin{equation}\label{eq:B1-cylinder-response}
 D_vF(h)
 =\sum_{r=1}^m
 \partial_r\phi(J_1(h),\ldots,J_m(h))D_vJ_r(h),
\end{equation}
and, for every smooth horizontal $v$,
\begin{equation}\label{eq:B1-cylinder-WP-bound}
 \sup_h|D_vF(h)|\le C_F\|v\|_{\mathrm{WP}}.
\end{equation}
Riesz representation applied to \eqref{eq:B1-cylinder-WP-bound}
gives the differential \eqref{eq:intro-gradient-identity}, with
\begin{equation}\label{eq:B1-gradient-def}
 \begin{aligned}
 D_vF(h)&=(DF(h),v)_{\mathrm{WP}},\\[3pt]
 \|DF(h)\|_{\mathrm{WP}}&\le C_F.
 \end{aligned}
\end{equation}
Enumerate the basis \eqref{eq:B1-WP-basis} as $(e_j)_{j\geq1}$.
Each $D_{e_j}F$ is measurable by \eqref{eq:B1-cylinder-response}, and
\[
 \sum_{j=1}^N D_{e_j}F(h)e_j\longrightarrow DF(h)
 \quad\text{in }H_{\mathrm{WP}},\qquad h\in\mathcal W.
\]
The finite sums are strongly measurable, so their pointwise limit is
strongly measurable.  The bound in \eqref{eq:B1-gradient-def} gives
\begin{equation}\label{eq:B1-gradient-L2}
 DF\in L^2(\widetilde\nu_\kappa;H_{\mathrm{WP}}).
\end{equation}

\begin{lemma}
\label{lem:B1-density}
\begin{equation}\label{eq:B1-core-L2-density}
 \overline{\mathcal C}^{\,L^2(\widetilde\nu_\kappa)}
 =L^2(\widetilde\nu_\kappa).
\end{equation}
\end{lemma}

\begin{proof}
We first show that $\{J_{a,b}\}$ separates points of
$\operatorname{Homeo}_+(\mathbb S^1)$.  If $h_1\ne h_2$, choose $\theta_0$
with $h_1(e^{i\theta_0})\ne h_2(e^{i\theta_0})$.  Choose a real smooth $b$
with
\[
 b(h_1(e^{i\theta_0}))\ne b(h_2(e^{i\theta_0})).
\]
By continuity, $b\circ h_1-b\circ h_2$ has a fixed nonzero sign on a small
arc $I$ about $\theta_0$.  Choosing $0\le a\in C_c^\infty(I)$,
$a\not\equiv0$, gives
\[
 J_{a,b}(h_1)-J_{a,b}(h_2)
 =\frac1{2\pi}\int_I a(\theta)
 \{b(h_1(e^{i\theta}))-b(h_2(e^{i\theta}))\}\,d\theta\ne0.
\]

Choose countable sets $A_{\mathrm{tr}}$ and $B_{\mathrm{sm}}$ of trigonometric
polynomials with rational coefficients, dense in the uniform norm in the
corresponding smooth real functions.  Because $(a,b)\mapsto J_{a,b}(h)$ is
continuous in the two uniform norms, the countable family
\[
 \{J_{a,b}:a\in A_{\mathrm{tr}},\ b\in B_{\mathrm{sm}}\}
\]
still separates points.  Enumerate it by $(J_k)_{k\ge1}$ and define
\[
 \begin{gathered}
 \operatorname{Obs}:\operatorname{Homeo}_+(\mathbb S^1)\to\mathbb R^{\mathbb N},\\[3pt]
 \operatorname{Obs}(h)=(J_k(h))_{k\ge1}.
 \end{gathered}
\]
The map $\operatorname{Obs}$ is measurable and injective.  Since
$\operatorname{Homeo}_+(\mathbb S^1)$ is a standard Borel space,
Lusin--Souslin implies that its image under $\operatorname{Obs}$ is a Borel set
and that the inverse on the image is measurable.  Consequently
\begin{equation}\label{eq:B1-generated-sigma}
 \sigma(J_k:k\ge1)=\mathcal B(\operatorname{Homeo}_+(\mathbb S^1)).
\end{equation}
On $\mathcal W$, let $\mathcal F_m$ be the completion of
$\sigma(J_1,\ldots,J_m)$ under $\widetilde\nu_\kappa$.  By
\eqref{eq:B1-generated-sigma}, these sigma-algebras increase to the
completed sigma-algebra modulo null sets.  The $L^2$ martingale
convergence theorem gives, for every $F\in L^2(\widetilde\nu_\kappa)$,
\begin{equation}\label{eq:B1-martingale-density}
 \mathbb E[F\mid\mathcal F_m]\longrightarrow F
 \quad\text{in }L^2(\widetilde\nu_\kappa).
\end{equation}
Every $\mathcal F_m$-measurable square-integrable function equals
$f(J_1,\ldots,J_m)$ almost surely for a measurable $f$ on $\mathbb R^m$.  Put
$\mu_m=(J_1,\ldots,J_m)_\#\widetilde\nu_\kappa$.  Since $\mu_m$ is a finite Borel measure on
$\mathbb R^m$, it is regular.  If $E\subset\mathbb R^m$ is a Borel set and
$\varepsilon>0$, inner regularity first gives a compact $K\subset E$ with
$\mu_m(E\setminus K)<\varepsilon$.  Outer regularity, applied to $K$, then
gives a bounded open $O$ such that
\begin{equation}\label{eq:B1-regular-sets}
 K\subset O,
 \qquad
 \mu_m(E\setminus K)<\varepsilon,
 \qquad
 \mu_m(O\setminus K)<\varepsilon.
\end{equation}
By Urysohn's lemma there is $\phi\in C_c(\mathbb R^m)$ with
$0\le\phi\le1$, $\phi=1$ on $K$, and $\operatorname{supp}\phi\subset O$.
Consequently
\begin{equation}\label{eq:B1-indicator-continuous}
 \|1_E-\phi\|_{L^2(\mu_m)}^2
 \le \mu_m(E\setminus K)+\mu_m(O\setminus K)<2\varepsilon.
\end{equation}
Finite linear combinations of indicators are dense in $L^2(\mu_m)$, so
$C_c(\mathbb R^m)$ is dense there.  If $\zeta_\delta$ is a standard smooth
mollifier, then for $\phi\in C_c(\mathbb R^m)$,
$\phi*\zeta_\delta\in C_c^\infty(\mathbb R^m)$ for small $\delta$ and
\begin{equation}\label{eq:B1-mollify-finite}
 \|\phi*\zeta_\delta-\phi\|_{L^2(\mu_m)}
 \le \mu_m(\mathbb R^m)^{1/2}
 \|\phi*\zeta_\delta-\phi\|_\infty
 \longrightarrow0.
\end{equation}
Thus
\begin{equation}\label{eq:B1-Ccinfty-density}
 \overline{C_c^\infty(\mathbb R^m)}^{L^2(\mu_m)}=L^2(\mu_m).
\end{equation}
Choose $\phi_m\in C_c^\infty(\mathbb R^m)$ so that
\[
 \|\phi_m(J_1,\ldots,J_m)-\mathbb E[F\mid\mathcal F_m]\|_{L^2}
 <m^{-1}.
\]
Then $\phi_m(J_1,\ldots,J_m)\in\mathcal C$ and
\[
 \|F-\phi_m(J_1,\ldots,J_m)\|_{L^2}
 \leq\|F-\mathbb E[F\mid\mathcal F_m]\|_{L^2}+m^{-1}
 \longrightarrow0.
\]
This proves \eqref{eq:B1-core-L2-density}.
\end{proof}

\section{Coefficient response and right integration by parts}
\label{sec:direct-right}

The right integration-by-parts formula below identifies the stress score
as the divergence needed for the closure and covariance arguments.
The proofs are given in Appendices~\ref{app:coefficient-response}--\ref{app:right-IBP}.

\subsection{Coefficient response}

Write $t_n(f)=t_n(h)$ when only the interior map is relevant, and set
\begin{equation}\label{eq:DR-mode-constant}
 C_{\kappa,n}:=2n+\frac{c_{\mathrm L}}{12}(n^3-n),\qquad n\geq1.
\end{equation}
On the polynomial algebra $\mathcal P_{\mathrm{hol}}:=\mathbb C[a_1,a_2,\ldots]$,
define the coefficient derivations
\begin{equation}\label{eq:DR-positive-Kirillov}
 \mathcal L_n
 :=\frac{\partial}{\partial a_n}
  +\sum_{m\geq1}(m+1)a_m
    \frac{\partial}{\partial a_{m+n}},\qquad n\geq1.
\end{equation}
The sum is finite on each polynomial.  The opposite modes are determined by
\begin{equation}\label{eq:DR-negative-Kirillov}
 (\mathcal L_{-n}f)(z)
 :=\frac{f(z)^2}{2\pi i}
 \oint_{|w|=r}
 \frac{f'(w)^2w^{1-n}}
 {f(w)^2\{f(w)-f(z)\}}\,dw,
 \qquad |z|<r<1,
\end{equation}
through $\mathcal L_{-n}a_m=[z^{m+1}](\mathcal L_{-n}f)(z)$ and the
Leibniz rule.  The residue theorem makes this definition independent of $r$.
These are the coefficient vector fields of \cite{AN08}.

In the next lemma, the supremum is over normalized univalent maps
$f:\mathbb D\to\mathbb C$, with $f(0)=0$ and $f'(0)=1$;
$j_1,\ldots,j_n$ are fixed nonnegative integers.

\begin{lemma}\label{lem:DR-stress-polynomial}
For every $n\geq1$,
\begin{align}
 t_n&\in\mathbb C[a_1,\ldots,a_n],
 \label{eq:DR-tn-polynomial}\\
 \mathcal L_nt_n&=C_{\kappa,n},
 \label{eq:DR-level-one}\\
 \shortintertext{and}
 \sup_f\left|
 \partial_{a_1}^{j_1}\cdots\partial_{a_n}^{j_n}t_n(f)
 \right|&<\infty.
 \label{eq:DR-stress-polynomial-uniform-bound}
\end{align}
\end{lemma}

Write $D_{\mathrm{rot}}:=D_1$, $D_{n,c}:=D_{\cos(n\cdot)}$, and
$D_{n,s}:=D_{\sin(n\cdot)}$, and put
\begin{equation}\label{eq:DR-complex-response}
 D_n^+:=D_{n,c}+iD_{n,s},\qquad n\geq1.
\end{equation}

\begin{lemma}\label{lem:DR-exact-change-frame}
For $n\geq1$ and $P\in\mathcal P_{\mathrm{hol}}$, at every $h\in\mathcal W$,
\begin{align}
 D_n^+P&=i\mathcal L_nP,
 \label{eq:DR-positive-change-frame}\\
 \shortintertext{and}
 (D_{n,c}-iD_{n,s})P&=i\mathcal L_{-n}P.
 \label{eq:DR-negative-change-frame}
\end{align}
Equivalently,
\begin{align}
 D_{n,c}P&=\frac i2(\mathcal L_n+\mathcal L_{-n})P,
 \label{eq:DR-real-cos-Kirillov}\\
 \shortintertext{and}
 D_{n,s}P&=\frac12(\mathcal L_n-\mathcal L_{-n})P.
 \label{eq:DR-real-sin-Kirillov}
\end{align}
\end{lemma}

\subsection{Right integration by parts}

For the normalized exterior map, define the rotation characters
\begin{equation}\label{eq:DR-exterior-rotation-characters}
 \chi_j(h):=\left(\frac{g'(\infty)}{|g'(\infty)|}\right)^j,
 \qquad j\in\mathbb Z.
\end{equation}
Let $\mathcal A_0$ be the unital complex algebra generated by
$a_m,\overline{a_m}$, the characters $\chi_j$, and the cylinder algebra
$\mathcal C$ from \eqref{eq:intro-core}.
Let $\Phi$ range over the finite compositions
\eqref{eq:intro-finite-composition} and the identity.  Define
\begin{equation}\label{eq:DR-flow-saturated-core}
 \mathcal A_{\mathrm{fl}}
 :=\operatorname{alg}\{F\circ R_\Phi:F\in\mathcal A_0,\ \Phi\text{ as above}\}.
\end{equation}
Lemma~\ref{lem:B1-unique-welding-flow-stability} makes $R_\Phi$ well defined
on $\mathcal W$.  On normalized welding pairs, the generators of
$\mathcal A_0$ are measurable by Appendix~\ref{sec:DR-frame-spaces}, and
$R_\Phi$ is continuous in the Carath\'eodory topology by
Section~\ref{sec:sle-closed}.  Thus $F\circ R_\Phi$ is measurable on the
frame space for $F\in\mathcal A_0$.  The identification of completed
probability spaces in that appendix makes these functions measurable on
$\mathcal W$.
De Branges' bound $|a_m|\leq m+1$ and $|\chi_j|=1$ make them bounded.
The algebra $\mathcal A_{\mathrm{fl}}$ is stable under conjugation and
these right compositions, and
\begin{equation}\label{eq:DR-averaged-core-in-flow-core}
 \mathcal C\subset\mathcal A_0\subset\mathcal A_{\mathrm{fl}}.
\end{equation}
Thus \eqref{eq:B1-generated-sigma} makes it measure determining.

\begin{lemma}\label{lem:DR-flow-core-uniform-differentiability}
For a real trigonometric polynomial $v$ and $F\in\mathcal A_{\mathrm{fl}}$,
\begin{equation}\label{eq:DR-flow-core-uniform-differentiability}
 \begin{gathered}
 \lim_{t\to0}\left\|
 \frac{F\circ T_t^v-F}{t}-D_vF
 \right\|_{\sup}=0,\\[3pt]
 \|D_vF\|_{\sup}<\infty.
 \end{gathered}
\end{equation}
\end{lemma}

For a derivation $X$ and a measure $\nu$, our divergence convention is
\begin{equation}\label{eq:DR-divergence-convention}
 \int XF\,d\nu=\int F\,\operatorname{Div}_\nu X\,d\nu
\end{equation}
on the specified test algebra.

\begin{proposition}\label{prop:DR-right-IBP}
For $v\in H_{\mathrm{WP}}^{\mathrm{fin}}$ and $F\in\mathcal A_{\mathrm{fl}}$,
\begin{equation}\label{eq:DR-hyp-IBP}
 \int D_vF\,d\widetilde\nu_\kappa
 =\int F\rho_v\,d\widetilde\nu_\kappa.
\end{equation}
\end{proposition}

\section{Closed response and Fisher--K\"ahler geometry}
\label{sec:ward-fisher}

\subsection{Rotation response and scalar closability}

For $s\in\mathbb R$, let $r_s(e^{i\theta})=e^{i(\theta+s)}$ and define
\begin{equation}\label{eq:B1-rotation-unitary}
 U_s^{\mathrm{rot}}F(h):=F(h\circ r_s).
\end{equation}
\Needspace{8\baselineskip}
Right-rotation invariance of $\widetilde\nu_\kappa$ gives, on the
complexification of $L^2(\widetilde\nu_\kappa)$,
\begin{gather*}
 U_{s+t}^{\mathrm{rot}}=U_s^{\mathrm{rot}}U_t^{\mathrm{rot}},\qquad
 U_0^{\mathrm{rot}}=\one,\\
 \shortintertext{and}
 \|U_s^{\mathrm{rot}}F\|_{L^2}=\|F\|_{L^2}.
\end{gather*}
Lemma~\ref{lem:B1-J-response} and the cylinder chain rule, applied to the
constant field $1$, give
\begin{equation}\label{eq:B1-rotation-core-continuity}
 \lim_{s\to0}\left\|
 \frac{U_s^{\mathrm{rot}}F-F}{s}-D_1F\right\|_{\sup}=0,
 \qquad F\in\mathcal C+i\mathcal C.
\end{equation}
For $F\in L^2(\widetilde\nu_\kappa;\mathbb C)$ and
$G\in\mathcal C+i\mathcal C$, unitarity gives
\[
 \|U_s^{\mathrm{rot}}F-F\|_{L^2}
 \leq2\|F-G\|_{L^2}+\|U_s^{\mathrm{rot}}G-G\|_{L^2}.
\]
First let $s\to0$, then approximate $F$ by $G\in\mathcal C+i\mathcal C$ using
Lemma~\ref{lem:B1-density}.  Thus $(U_s^{\mathrm{rot}})_{s\in\mathbb R}$ is strongly
continuous.  Its generator, with domain consisting of the functions for
which the following $L^2$ limit exists, extends the pointwise rotation
response:
\begin{equation}\label{eq:B1-rotation-generator}
 D_{\mathrm{rot}}F:=\lim_{s\to0}\frac{U_s^{\mathrm{rot}}F-F}{s}
 \quad\text{in }L^2(\widetilde\nu_\kappa).
\end{equation}
Stone's theorem applies to this strongly continuous unitary group.
Together with \eqref{eq:B1-rotation-core-continuity}, it gives
\begin{equation}\label{eq:B1-rotation-skew-adjoint}
 \begin{gathered}
 D_{\mathrm{rot}}^*=-D_{\mathrm{rot}},\\[3pt]
 \mathcal C\subset\operatorname{Dom}(D_{\mathrm{rot}}),\qquad
 D_{\mathrm{rot}}F=D_1F\quad(F\in\mathcal C).
 \end{gathered}
\end{equation}

Fix $v\in H_{\mathrm{WP}}^{\mathrm{fin}}$.  On the cylinder core we consider
\begin{equation}\label{eq:B1-scalar-map}
 D_v:\mathcal C\subset L^2(\widetilde\nu_\kappa)
 \longrightarrow L^2(\widetilde\nu_\kappa).
\end{equation}
Let $(F_m)_{m\geq1}\subset\mathcal C$ and
$Y\in L^2(\widetilde\nu_\kappa)$.

\begin{lemma}
\label{lem:B1-scalar-close}
The operator $D_v$ is well defined on $L^2$ equivalence classes and closable:
\begin{equation}\label{eq:B1-scalar-close-criterion}
 F_m\to0\text{ in }L^2,
 \qquad D_vF_m\to Y\text{ in }L^2
 \quad\Longrightarrow\quad Y=0.
\end{equation}
\end{lemma}

\begin{proof}
For $F,G\in\mathcal C$, Proposition~\ref{prop:DR-right-IBP} applied to
$F\overline G$ and the Leibniz rule give
\begin{equation}\label{eq:B1-scalar-adjoint-pairing}
 \langle D_vF,G\rangle_{L^2}
 =\langle F,\rho_vG-D_vG\rangle_{L^2}.
\end{equation}
The score and the core response are bounded, so
\[
 \|\rho_vG-D_vG\|_{L^2}
 \leq\|\rho_v\|_{\sup}\|G\|_{L^2}+\|D_vG\|_{\sup}<\infty.
\]
If $F=0$ in $L^2$, the pairing in
\eqref{eq:B1-scalar-adjoint-pairing} vanishes for every $G\in\mathcal C$.
Density of $\mathcal C$ implies $D_vF=0$ in $L^2$, proving that $D_v$
is well defined on equivalence classes.

Under the two convergence hypotheses in
\eqref{eq:B1-scalar-close-criterion}, for every $G\in\mathcal C$,
\begin{align*}
 |\langle Y,G\rangle_{L^2}|
 &\leq\|Y-D_vF_m\|_{L^2}\|G\|_{L^2}
       +|\langle F_m,\rho_vG-D_vG\rangle_{L^2}|\\
 &\leq\|Y-D_vF_m\|_{L^2}\|G\|_{L^2}
       +\|F_m\|_{L^2}\|\rho_vG-D_vG\|_{L^2}
 \longrightarrow0.
\end{align*}
Density again gives $Y=0$.
\end{proof}

\subsection{Final proof of Theorem~\ref{thm:B1-main}}
\label{sec:proof-closed-response}

The initial domain is dense by Lemma~\ref{lem:B1-density}.  Retain the
enumeration $(e_j)_{j\geq1}$ of \eqref{eq:B1-WP-basis} used in
Section~\ref{sec:sle-closed}.  Since every $D_{e_j}$ is well defined on
$L^2$ equivalence classes by Lemma~\ref{lem:B1-scalar-close}, so is $D$.

Let $(F_m)_{m\geq1}\subset\mathcal C$ and
$U\in L^2(\widetilde\nu_\kappa;H_{\mathrm{WP}})$ satisfy
\begin{equation}\label{eq:B1-closability-criterion}
 \begin{aligned}
 F_m&\longrightarrow0\quad\text{in }L^2(\widetilde\nu_\kappa),\\[3pt]
 DF_m&\longrightarrow U\quad\text{in }L^2(\widetilde\nu_\kappa;H_{\mathrm{WP}}).
 \end{aligned}
\end{equation}
For each $j$,
\begin{equation}\label{eq:B1-coordinate-convergence}
 \|D_{e_j}F_m-(U,e_j)_{\mathrm{WP}}\|_{L^2}
 \leq\|DF_m-U\|_{L^2(H_{\mathrm{WP}})}\longrightarrow0.
\end{equation}
Each $e_j$ lies in $H_{\mathrm{WP}}^{\mathrm{fin}}$, so
Lemma~\ref{lem:B1-scalar-close} and $F_m\to0$ imply
\begin{equation}\label{eq:B1-coordinate-zero}
 (U,e_j)_{\mathrm{WP}}=0
 \quad\text{in }L^2(\widetilde\nu_\kappa),\qquad j\geq1.
\end{equation}
Parseval's identity and Tonelli's theorem now give
\[
 \|U\|_{L^2(H_{\mathrm{WP}})}^2
 =\int\sum_{j\geq1}|(U,e_j)_{\mathrm{WP}}|^2\,d\widetilde\nu_\kappa
 =\sum_{j\geq1}\|(U,e_j)_{\mathrm{WP}}\|_{L^2}^2=0.
\]
Thus $D$ is closable.

Let $D$ now denote the closure.  Its graph
\begin{equation}\label{eq:B1-graph}
 \{(F,DF):F\in\operatorname{Dom}(D)\}
\end{equation}
is a closed linear subspace of
$L^2(\widetilde\nu_\kappa)\oplus
 L^2(\widetilde\nu_\kappa;H_{\mathrm{WP}})$.
Thus $\operatorname{Dom}(D)$ is complete for the graph norm
\begin{equation}\label{eq:B1-graph-norm}
 \|F\|_{\operatorname{Dom}(D)}^2
 :=\|F\|_{L^2}^2+\|DF\|_{L^2(H_{\mathrm{WP}})}^2
 =\|F\|_{L^2}^2+\mathcal E(F,F)
\end{equation}
and \eqref{eq:intro-Dirichlet-form} is a densely defined, closed,
symmetric nonnegative form.

Let $\eta\in C^\infty(\mathbb R)$ have bounded derivative.  For
$F\in\mathcal C$, choose a compactly supported smooth function that
agrees with $\eta$ on a neighborhood of the bounded range of $F$.
The cylinder chain rule gives
\begin{equation}\label{eq:B1-chain-core}
 \eta(F)\in\mathcal C,\qquad D(\eta(F))=\eta'(F)DF.
\end{equation}
Now fix $F\in\operatorname{Dom}(D)$ and choose $F_m\in\mathcal C$ with
\begin{equation}\label{eq:B1-core-graph-approx}
 F_m\to F\quad\text{in }L^2,
 \qquad
 DF_m\to DF\quad\text{in }L^2(H_{\mathrm{WP}}).
\end{equation}
Since $\widetilde\nu_\kappa$ is a probability measure,
\begin{align*}
 \|\eta(F)\|_{L^2}
 &\leq|\eta(0)|+\|\eta'\|_\infty\|F\|_{L^2}<\infty,\\
 \shortintertext{and}
 \|\eta(F_m)-\eta(F)\|_{L^2}
 &\leq\|\eta'\|_\infty\|F_m-F\|_{L^2}\longrightarrow0.
\end{align*}
For the gradients,
\begin{align}
 &\|\eta'(F_m)DF_m-\eta'(F)DF\|_{L^2(H_{\mathrm{WP}})}\notag\\
 &\quad\le \|\eta'\|_\infty\|DF_m-DF\|_{L^2(H_{\mathrm{WP}})}
 +\|(\eta'(F_m)-\eta'(F))DF\|_{L^2(H_{\mathrm{WP}})}.
 \label{eq:B1-chain-limit-bound}
\end{align}
To control the second term, put $M=2\|\eta'\|_\infty$ and, for $\delta>0$,
\[
 A_{m,\delta}:=\{h\in\mathcal W:
 |\eta'(F_m(h))-\eta'(F(h))|>\delta\}.
\]
Continuity of $\eta'$ and $F_m\to F$ in measure give
$\widetilde\nu_\kappa(A_{m,\delta})\to0$.  Also,
\begin{equation}\label{eq:B1-chain-tail}
 \lim_{R\to\infty}
 \int_{\{\|DF\|_{\mathrm{WP}}>R\}}
 \|DF\|_{\mathrm{WP}}^2\,d\widetilde\nu_\kappa=0.
\end{equation}
For $R,\delta>0$, splitting over the gradient tail, $A_{m,\delta}$,
and its complement yields the estimate
\begin{align}
 &\|(\eta'(F_m)-\eta'(F))DF\|_{L^2(H_{\mathrm{WP}})}^2\notag\\
 &\quad\leq
 M^2\int_{\{\|DF\|_{\mathrm{WP}}>R\}}
       \|DF\|_{\mathrm{WP}}^2\,d\widetilde\nu_\kappa\notag\\
 &\qquad+M^2R^2\widetilde\nu_\kappa(A_{m,\delta})
 +\delta^2\|DF\|_{L^2(H_{\mathrm{WP}})}^2.
 \label{eq:B1-chain-multiplier-limit}
\end{align}
Letting $m\to\infty$, then $R\to\infty$ and $\delta\downarrow0$, proves
convergence to zero.  Equations \eqref{eq:B1-chain-core} and
\eqref{eq:B1-chain-limit-bound}, together with closedness of $D$, imply
\begin{equation}\label{eq:B1-chain-closed}
 \eta(F)\in\operatorname{Dom}(D),
 \qquad
 D(\eta(F))=\eta'(F)DF.
\end{equation}
If $|\eta'|\le1$, then
\begin{equation}\label{eq:B1-smooth-contraction}
 \mathcal E(\eta(F),\eta(F))
 =\int|\eta'(F)|^2\|DF\|_{\mathrm{WP}}^2\,d\widetilde\nu_\kappa
 \le\mathcal E(F,F).
\end{equation}

Now let $\eta$ be an arbitrary normal contraction.
Choose an even nonnegative $\zeta\in C_c^\infty(\mathbb R)$ with
$\int\zeta=1$ and $\operatorname{supp}\zeta\subset[-1,1]$, and put
\begin{equation}\label{eq:B1-contraction-mollifier}
 \zeta_m(x)=m\zeta(mx),
 \qquad
 \eta_m(x)=(\eta*\zeta_m)(x)-(\eta*\zeta_m)(0).
\end{equation}
For $x,y\in\mathbb R$,
\[
 |\eta_m(x)-\eta_m(y)|
 \leq\int|\eta(x-u)-\eta(y-u)|\zeta_m(u)\,du
 \leq|x-y|.
\]
Consequently,
\begin{equation}\label{eq:B1-mollifier-lipschitz}
 \eta_m(0)=0,
 \qquad
 \|\eta_m'\|_\infty\le1.
\end{equation}
Also
\begin{equation}\label{eq:B1-mollifier-uniform}
 \sup_x|\eta_m(x)-\eta(x)|
 \leq2\int|u|\zeta_m(u)\,du\leq\frac2m.
\end{equation}
Thus $\eta_m(F)\to \eta(F)$ in $L^2$.  By
\eqref{eq:B1-chain-closed}--\eqref{eq:B1-smooth-contraction},
\begin{equation}\label{eq:B1-uniform-energy-contraction}
 \sup_m\|D(\eta_m(F))\|_{L^2(H_{\mathrm{WP}})}
 \le\|DF\|_{L^2(H_{\mathrm{WP}})}.
\end{equation}
Boundedness in this Hilbert space yields a subsequence indexed by
$m_k\to\infty$ and
$V\in L^2(\widetilde\nu_\kappa;H_{\mathrm{WP}})$ such that
$D(\eta_{m_k}(F))\rightharpoonup V$.  Together with the strong convergence
of $\eta_{m_k}(F)$, this gives
\begin{equation}\label{eq:B1-product-weak}
 (\eta_{m_k}(F),D(\eta_{m_k}(F)))
 \rightharpoonup(\eta(F),V)
\end{equation}
in the product Hilbert space.  The graph \eqref{eq:B1-graph} is a closed
linear subspace and therefore weakly closed.  Hence
\begin{equation}\label{eq:B1-contraction-domain}
 \eta(F)\in\operatorname{Dom}(D),
 \qquad D(\eta(F))=V.
\end{equation}
Weak lower semicontinuity gives
\begin{align}
 \mathcal E(\eta(F),\eta(F))
 &=\|V\|_{L^2(H_{\mathrm{WP}})}^2\notag\\
 &\le\liminf_{k\to\infty}\|D(\eta_{m_k}(F))\|_{L^2(H_{\mathrm{WP}})}^2
 \le\mathcal E(F,F),
 \label{eq:B1-contraction-energy}
\end{align}
which proves the Markov property in Theorem~\ref{thm:B1-main}.

\subsection{Fourier scores and the Ward identity}

Fix $n\geq2$ and $F,G\in\mathcal A_{\mathrm{fl}}$.

\begin{corollary}
\label{cor:DR-real-divergence}
The mode divergences satisfy
\begin{align}
 \operatorname{Div}_{\widetilde\nu_\kappa}D_{n,c}
 &=\operatorname{Im}t_n,
 \label{eq:DR-cos-divergence}\\
 \operatorname{Div}_{\widetilde\nu_\kappa}D_{n,s}
 &=\operatorname{Re}t_n,
 \label{eq:DR-sin-divergence}\\
 \shortintertext{and}
 \operatorname{Div}_{\widetilde\nu_\kappa}D_n^+
 &=i\overline{t_n}.
 \label{eq:DR-complex-divergence}
\end{align}
Their adjoint pairings are
\begin{align}
 \langle D_{n,c}F,G\rangle_{L^2}
 &=\langle F,(\operatorname{Im}t_n)G-D_{n,c}G\rangle_{L^2},
 \label{eq:DR-cos-adjoint}\\
 \shortintertext{and}
 \langle D_{n,s}F,G\rangle_{L^2}
 &=\langle F,(\operatorname{Re}t_n)G-D_{n,s}G\rangle_{L^2}.
 \label{eq:DR-sin-adjoint}
\end{align}
\end{corollary}

\begin{proof}
 Proposition~\ref{prop:DR-right-IBP} and
 \eqref{eq:intro-score} imply
\begin{align*}
 \operatorname{Div}_{\widetilde\nu_\kappa}D_{n,c}
 &=\rho_{\cos(n\cdot)}=\operatorname{Im}t_n,\\
 \operatorname{Div}_{\widetilde\nu_\kappa}D_{n,s}
 &=\rho_{\sin(n\cdot)}=\operatorname{Re}t_n.
\end{align*}
Since $\operatorname{Im}t_n+i\operatorname{Re}t_n=i\overline{t_n}$,
complex linearity gives \eqref{eq:DR-complex-divergence}.  For
$v=\cos(n\cdot)$ or $\sin(n\cdot)$, the Leibniz rule gives
\[
 \int(D_vF)\overline G\,d\widetilde\nu_\kappa
 =\int F\overline G\rho_v\,d\widetilde\nu_\kappa
  -\int F\overline{D_vG}\,d\widetilde\nu_\kappa,
\]
which is \eqref{eq:DR-cos-adjoint} or \eqref{eq:DR-sin-adjoint}.
\end{proof}

We next derive the coefficient adjoint relation considered in
\cite{AM01}, with the convention \eqref{eq:DR-negative-Kirillov}.
We take the Hilbert-space adjoint of $\mathcal L_n$ in
$\overline{\mathcal P_{\mathrm{hol}}}^{\,L^2(\widetilde\nu_\kappa;\mathbb C)}$,
with initial domain $\mathcal P_{\mathrm{hol}}$.
Let $\phi,\psi\in\mathcal P_{\mathrm{hol}}$.

\begin{corollary}
\label{cor:DR-AM-identity}
For $n\geq2$,
\begin{equation}\label{eq:DR-AM-identity}
 \langle\mathcal L_n\phi,\psi\rangle_{L^2}
 =\langle\phi,\mathcal L_{-n}\psi+t_n\psi\rangle_{L^2}.
\end{equation}
In particular,
\begin{equation}\label{eq:DR-Ln-adjoint}
 \psi\in\operatorname{Dom}(\mathcal L_n^*),\qquad
 \mathcal L_n^*\psi=\mathcal L_{-n}\psi+t_n\psi.
\end{equation}
\end{corollary}

\begin{proof}
\renewcommand{\qedsymbol}{}
For $F=\phi\overline\psi\in\mathcal A_0\subset\mathcal A_{\mathrm{fl}}$,
Lemma~\ref{lem:DR-exact-change-frame} and
commutation of the real responses with conjugation give
\begin{equation}\label{eq:DR-mixed-mode-action}
 D_n^+(\phi\overline\psi)
 =i(\mathcal L_n\phi)\overline\psi
  -i\phi\,\overline{\mathcal L_{-n}\psi}.
\end{equation}
Integrating and using \eqref{eq:DR-complex-divergence} gives
\begin{align}
 &i\int(\mathcal L_n\phi)\overline\psi\,d\widetilde\nu_\kappa
 -i\int\phi\,\overline{\mathcal L_{-n}\psi}
   \,d\widetilde\nu_\kappa\notag\\
 &\qquad=i\int\phi\overline\psi\,\overline{t_n}
   \,d\widetilde\nu_\kappa.
 \label{eq:DR-AM-from-divergence}
\end{align}
Dividing by $i$ and moving the second term to the right proves
\eqref{eq:DR-AM-identity}.

To verify the adjoint domain, expand \eqref{eq:DR-negative-Kirillov}
in powers of $f(z)/f(w)$ for $z$ sufficiently close to zero.  For $m\geq1$,
\begin{equation}\label{eq:DR-negative-mode-polynomial}
 \begin{aligned}
 \mathcal L_{-n}a_m
 &=\sum_{j=0}^{m-1}
   \bigl([z^{m+1}]f(z)^{j+2}\bigr)
   \biggl([w^{n+j+1}]f'(w)^2
                   \left(\frac{w}{f(w)}\right)^{j+3}\biggr)\\
 &\in\mathbb C[a_1,\ldots,a_{m+n}].
 \end{aligned}
\end{equation}
Thus $\mathcal L_{\pm n}\mathcal P_{\mathrm{hol}}\subset\mathcal P_{\mathrm{hol}}$,
and $\mathcal L_{-n}\psi+t_n\psi$ is a bounded coefficient polynomial.
Equation \eqref{eq:DR-AM-identity} yields
\[
 |\langle\mathcal L_n\phi,\psi\rangle_{L^2}|
 \leq\|\phi\|_{L^2}\|\mathcal L_{-n}\psi+t_n\psi\|_{L^2}.
\]
Density of $\mathcal P_{\mathrm{hol}}$ in its closed span first makes
$\mathcal L_n$ well defined on $L^2$ equivalence classes, and the same
bound gives \eqref{eq:DR-Ln-adjoint}.
\end{proof}

All moments below are finite by Lemma~\ref{lem:DR-stress-polynomial}.

\begin{proposition}
\label{prop:DR-Fourier-Ward}
For all $m,n\geq2$,
\begin{equation}\label{eq:DR-stress-mean-zero}
 \int t_n\,d\widetilde\nu_\kappa=0.
\end{equation}
The second moments satisfy
\begin{align}
 \int t_nt_m\,d\widetilde\nu_\kappa&=0,
 \label{eq:DR-holomorphic-mode-orthogonality}\\
 \shortintertext{and}
 \int t_n\overline{t_m}\,d\widetilde\nu_\kappa
 &=\delta_{nm}C_{\kappa,n}.
 \label{eq:DR-Hermitian-mode-orthogonality}
\end{align}
In real coordinates,
\begin{align}
 \int\operatorname{Re}t_n\operatorname{Re}t_m\,d\widetilde\nu_\kappa
 &=\int\operatorname{Im}t_n\operatorname{Im}t_m\,d\widetilde\nu_\kappa
 =\frac{\delta_{nm}}2C_{\kappa,n},
 \label{eq:DR-real-mode-variance}\\
 \shortintertext{and}
 \int\operatorname{Re}t_n\operatorname{Im}t_m
 \,d\widetilde\nu_\kappa&=0.
 \label{eq:DR-real-mode-cross}
\end{align}
\end{proposition}

\begin{proof}
Take $\phi=t_n$ and $\psi=1$ in
\eqref{eq:DR-AM-identity}.  Since $\mathcal L_{-n}1=0$, Lemma
\ref{lem:DR-stress-polynomial} gives
\begin{equation}\label{eq:DR-Ward-proof}
 C_{\kappa,n}
 =\int\mathcal L_nt_n\,d\widetilde\nu_\kappa
 =\int t_n\overline{t_n}\,d\widetilde\nu_\kappa.
\end{equation}
For $\alpha\in\mathbb R$, uniqueness of the normalized welding pair and
the definition \eqref{eq:intro-stress} give
\begin{align*}
 f_{h\circ r_\alpha}(z)&=e^{-i\alpha}f_h(e^{i\alpha}z),\\[3pt]
 a_m(h\circ r_\alpha)&=e^{im\alpha}a_m(h),\\[3pt]
 q_{f_{h\circ r_\alpha}}(z)&=e^{2i\alpha}q_{f_h}(e^{i\alpha}z).
\end{align*}
Taking coefficients yields
\begin{equation}\label{eq:DR-stress-rotation-phase}
 t_n(h\circ r_\alpha)=e^{in\alpha}t_n(h).
\end{equation}
Right-rotation invariance gives, for every $\alpha\in\mathbb R$,
\begin{align*}
 (1-e^{in\alpha})\int t_n\,d\widetilde\nu_\kappa&=0,\\
 (1-e^{i(n+m)\alpha})\int t_nt_m\,d\widetilde\nu_\kappa&=0,\\
 (1-e^{i(n-m)\alpha})\int t_n\overline{t_m}\,
 d\widetilde\nu_\kappa&=0.
\end{align*}
Taking $\alpha=\pi/n$, $\pi/(n+m)$, and, when $n\ne m$,
$\pi/(n-m)$, respectively, proves all vanishing moments.
The diagonal Hermitian moment is \eqref{eq:DR-Ward-proof}.  Expanding
real and imaginary parts now gives
\begin{align*}
 \int\operatorname{Re}t_n\operatorname{Re}t_m\,d\widetilde\nu_\kappa
 &=\int\operatorname{Im}t_n\operatorname{Im}t_m\,d\widetilde\nu_\kappa\\
 &=\frac14\int(t_n\overline{t_m}+\overline{t_n}t_m)
                    \,d\widetilde\nu_\kappa
 =\frac{\delta_{nm}}2C_{\kappa,n},\\
 \shortintertext{and}
 \int\operatorname{Re}t_n\operatorname{Im}t_m\,d\widetilde\nu_\kappa
 &=\frac1{4i}\int(\overline{t_n}t_m-t_n\overline{t_m})
                    \,d\widetilde\nu_\kappa=0.\qedhere
\end{align*}
\end{proof}

\subsection{Final proof of Theorem~\ref{thm:intro-Fisher-Kahler}}
\label{sec:proof-stress-covariance}

For $v\in H_{\mathrm{WP}}^{\mathrm{fin}}$ written as in
\eqref{eq:intro-real-finite-field}, the Fourier coefficients and the
score satisfy
\begin{equation}\label{eq:DR-score-complex-field}
 \rho_v=2\operatorname{Im}\sum_{n=2}^N\overline{v_n}t_n,\qquad
 v_n=\frac{x_n-iy_n}{2}\quad(2\leq n\leq N).
\end{equation}
Let $v,w\in H_{\mathrm{WP}}^{\mathrm{fin}}$, and choose $N\geq2$ such that
$v_n=w_n=0$ for $|n|>N$.  The real covariance identities in
Proposition~\ref{prop:DR-Fourier-Ward} give
\begin{align}
 \int\rho_v\rho_w\,d\widetilde\nu_\kappa
 &=2\sum_{n=2}^N C_{\kappa,n}
   \{\operatorname{Re}v_n\operatorname{Re}w_n
    +\operatorname{Im}v_n\operatorname{Im}w_n\}\notag\\
 &=2\operatorname{Re}\sum_{n=2}^N C_{\kappa,n}v_n\overline{w_n}\notag\\
 &=\frac{c_{\mathrm L}}6\operatorname{Re}\sum_{n=2}^N
       (n^3-n)v_n\overline{w_n}
   +4\operatorname{Re}\sum_{n=2}^N n v_n\overline{w_n}\notag\\
 &=G_\kappa(v,w).
 \label{eq:DR-finite-Fisher-computation}
\end{align}
For $n\geq2$,
\begin{equation}\label{eq:DR-VK-controlled-by-WP}
 n\leq\frac13(n^3-n),
\end{equation}
so the forms defined in \eqref{eq:intro-VK-form} and
\eqref{eq:intro-geometric-form} satisfy, for every $v\in H_{\mathrm{WP}}$,
\begin{equation}\label{eq:DR-Fisher-equivalent-WP}
 \frac{c_{\mathrm L}}{12}\|v\|_{\mathrm{WP}}^2
 \leq G_\kappa(v,v)
 \leq\left(\frac{c_{\mathrm L}}{12}+\frac23\right)
      \|v\|_{\mathrm{WP}}^2.
\end{equation}

For $v\in H_{\mathrm{WP}}$ and an integer $N\geq2$, define its Fourier
truncation by
\[
 v^{(N)}(\theta):=\sum_{2\leq|n|\leq N}v_ne^{in\theta}.
\]
Then
\[
 \|v-v^{(N)}\|_{\mathrm{WP}}^2
 =2\sum_{n>N}(n^3-n)|v_n|^2\longrightarrow0.
\]
Equations \eqref{eq:DR-finite-Fisher-computation} and
\eqref{eq:DR-Fisher-equivalent-WP} yield, for $M,N\geq2$,
\begin{align}
 \|\rho_{v^{(N)}}-\rho_{v^{(M)}}\|_{L^2}^2
 &=G_\kappa(v^{(N)}-v^{(M)},v^{(N)}-v^{(M)})\notag\\
 &\leq\left(\frac{c_{\mathrm L}}{12}+\frac23\right)
      \|v^{(N)}-v^{(M)}\|_{\mathrm{WP}}^2.
 \label{eq:DR-score-Cauchy}
\end{align}
We therefore define
\begin{equation}\label{eq:DR-score-L2-limit}
 \rho_v:=\lim_{N\to\infty}\rho_{v^{(N)}}
 \quad\text{in }L^2(\widetilde\nu_\kappa;\mathbb R).
\end{equation}
The same estimate applied to any finite-mode approximation of $v$
shows that the limit is independent of the approximation.  It defines
the unique bounded real-linear extension of $\rho$.  Since
$\int\rho_{v^{(N)}}\,d\widetilde\nu_\kappa=0$,
\[
 \left|\int\rho_v\,d\widetilde\nu_\kappa\right|
 \leq\|\rho_v-\rho_{v^{(N)}}\|_{L^2}\longrightarrow0.
\]
For $v,w\in H_{\mathrm{WP}}$,
\begin{align*}
 &\left|\int\rho_{v^{(N)}}\rho_{w^{(N)}}\,d\widetilde\nu_\kappa
             -\int\rho_v\rho_w\,d\widetilde\nu_\kappa\right|\\
 &\quad\leq\|\rho_{v^{(N)}}-\rho_v\|_{L^2}
             \|\rho_{w^{(N)}}\|_{L^2}
       +\|\rho_v\|_{L^2}\|\rho_{w^{(N)}}-\rho_w\|_{L^2}
 \longrightarrow0.
\end{align*}
Together with continuity of $G_\kappa$, this extends
\eqref{eq:DR-finite-Fisher-computation} to
\eqref{eq:intro-Fisher-form}.

For $F\in\mathcal C$ and $v\in H_{\mathrm{WP}}^{\mathrm{fin}}$,
\eqref{eq:intro-gradient-identity} and
Proposition~\ref{prop:DR-right-IBP} give
\begin{equation}\label{eq:DR-adjoint-finite-v}
 \int(DF,v)_{\mathrm{WP}}\,d\widetilde\nu_\kappa
 =\int F\rho_v\,d\widetilde\nu_\kappa.
\end{equation}
For $v\in H_{\mathrm{WP}}$, apply this identity to $v^{(N)}$ and use
\begin{align*}
 &\left|\int(DF,v)_{\mathrm{WP}}\,d\widetilde\nu_\kappa
             -\int F\rho_v\,d\widetilde\nu_\kappa\right|\\
 &\quad\leq\|DF\|_{L^2(H_{\mathrm{WP}})}
             \|v-v^{(N)}\|_{\mathrm{WP}}
       +\|F\|_{L^2}\|\rho_v-\rho_{v^{(N)}}\|_{L^2}
 \longrightarrow0.
\end{align*}
Now let $F\in\operatorname{Dom}(D)$ and choose $F_m\in\mathcal C$ with
$F_m\to F$ in the graph norm \eqref{eq:B1-graph-norm}.  The identity
just proved for $F_m$ gives
\begin{align*}
 &\left|\int(DF,v)_{\mathrm{WP}}\,d\widetilde\nu_\kappa
             -\int F\rho_v\,d\widetilde\nu_\kappa\right|\\
 &\quad\leq\|DF-DF_m\|_{L^2(H_{\mathrm{WP}})}\|v\|_{\mathrm{WP}}
       +\|F-F_m\|_{L^2}\|\rho_v\|_{L^2}
 \longrightarrow0.
\end{align*}
Thus the constant vector field $h\mapsto v$ satisfies
\[
 \left|\int(DF,v)_{\mathrm{WP}}\,d\widetilde\nu_\kappa\right|
 =|\langle F,\rho_v\rangle_{L^2}|
 \leq\|\rho_v\|_{L^2}\|F\|_{L^2},
 \qquad F\in\operatorname{Dom}(D).
\]
By the definition of the Hilbert-space adjoint,
\begin{equation}\label{eq:DR-gradient-adjoint-score}
 v\in\operatorname{Dom}(D^*),\qquad D^*v=\rho_v,
 \qquad v\in H_{\mathrm{WP}}.
\end{equation}

The covariance also realizes the complex and symplectic structures
defined in the Introduction.  For $v,w\in H_{\mathrm{WP}}$, the Fourier
formula \eqref{eq:intro-complex-structure} gives
\begin{equation}\label{eq:DR-Fisher-J-invariant}
 J_0^2v=-v,\qquad G_\kappa(J_0v,J_0w)=G_\kappa(v,w).
\end{equation}
Consequently,
\begin{align}
 \omega_{\kappa,0}(v,w)
 &=\frac{c_{\mathrm L}}{12}(J_0v,w)_{\mathrm{WP}}
   +4(J_0v,w)_{\mathrm{VK}}\notag\\
 &=\int\rho_{J_0v}\rho_w\,d\widetilde\nu_\kappa,
 \label{eq:DR-symplectic-WP-VK}\\
 \shortintertext{and}
 \omega_{\kappa,0}(w,v)&=-\omega_{\kappa,0}(v,w).\notag
\end{align}
The norm equivalence \eqref{eq:DR-Fisher-equivalent-WP}, Riesz
representation for $G_\kappa$, and $J_0^{-1}=-J_0$ show that
$v\mapsto\omega_{\kappa,0}(v,\cdot)$ is a bounded isomorphism from
$H_{\mathrm{WP}}$ onto its continuous dual.  Thus $\omega_{\kappa,0}$
is strong.  It is constant, so $d\omega_{\kappa,0}=0$.
The linear coordinates
\[
 v\longmapsto\bigl(\sqrt{2(n^3-n)}\,v_n\bigr)_{n\geq2}
 \in\ell^2(\{2,3,\ldots\};\mathbb C)
\]
identify $J_0$ with multiplication by $i$.  They are global complex
coordinates, proving that $(G_\kappa,J_0,\omega_{\kappa,0})$ is a
strong K\"ahler structure.

Finally, \eqref{eq:DR-Fisher-equivalent-WP} and
\eqref{eq:intro-Fisher-form} give
\[
 \|v-w\|_{\mathrm{WP}}^2
 \leq\frac{12}{c_{\mathrm L}}\|\rho_v-\rho_w\|_{L^2}^2,
 \qquad v,w\in H_{\mathrm{WP}}.
\]
Hence $\rho$ is injective, and every convergent sequence in its image
comes from a Cauchy sequence in $H_{\mathrm{WP}}$; its image is therefore
closed.  Define on this image
\[
 J_\rho(\rho_v):=\rho_{J_0v},\qquad v\in H_{\mathrm{WP}}.
\]
The covariance identity and \eqref{eq:DR-Fisher-J-invariant} imply
\begin{gather*}
 J_\rho^2\rho_v=-\rho_v,\\
 \shortintertext{and}
 \|J_\rho\rho_v\|_{L^2}^2
 =G_\kappa(J_0v,J_0v)=G_\kappa(v,v)=\|\rho_v\|_{L^2}^2.
\end{gather*}
In particular, $J_0\cos(n\cdot)=-\sin(n\cdot)$ and
$J_0\sin(n\cdot)=\cos(n\cdot)$ give
\begin{equation}\label{eq:DR-score-quarter-turn}
 J_\rho\rho_{\cos(n\cdot)}=-\rho_{\sin(n\cdot)},\qquad
 J_\rho\rho_{\sin(n\cdot)}=\rho_{\cos(n\cdot)},\qquad n\geq2.
\end{equation}

\section{Integrated response and the radial rough action}
\label{sec:integrated-response}

\subsection{Integration of the right divergence}

For $u\in H_{\mathrm{WP}}^{\mathrm{fin}}$ and $t\in\mathbb R$, define
the composition operator on bounded functions by
\begin{equation}\label{eq:DR-flow-graph-Koopman}
 U_t^uF:=F\circ T_t^u.
\end{equation}
The algebra $\mathcal A_{\mathrm{fl}}$ is invariant under $U_t^u$ by
\eqref{eq:DR-flow-saturated-core}.  Moreover,
\begin{equation}\label{eq:DR-flow-uniform-isometry}
 \|U_t^uF\|_{\sup}=\|F\|_{\sup},\qquad
 D_uU_t^uF=U_t^uD_uF,
 \quad F\in\mathcal A_{\mathrm{fl}}.
\end{equation}
The first identity follows from bijectivity of $T_t^u$; the second follows
by differentiating $U_s^uU_t^uF=U_{s+t}^uF$ in the uniform norm.
Lemma~\ref{lem:DR-stress-polynomial}, the coefficient response formulas,
and \eqref{eq:DR-negative-mode-polynomial} give
\begin{equation}\label{eq:DR-bv-iterated-core}
 \rho_v,\quad D_{u_1}\rho_v,\quad D_{u_1}D_{u_2}\rho_v\in\mathcal A_0,
 \qquad u_1,u_2,v\in H_{\mathrm{WP}}^{\mathrm{fin}}.
\end{equation}
Lemma~\ref{lem:DR-flow-core-uniform-differentiability} therefore makes
$t\mapsto U_t^u\rho_v$ continuously differentiable in the uniform norm,
with derivative $U_t^uD_u\rho_v$.  In particular,
\begin{equation}\label{eq:DR-score-time-Lipschitz}
 \|U_t^u\rho_v-U_s^u\rho_v\|_{\sup}
 \leq|t-s|\,\|D_u\rho_v\|_{\sup},
 \qquad s,t\in\mathbb R.
\end{equation}

For a fixed $u$, abbreviate the increment
\eqref{eq:intro-integrated-response} and define its target density by
\begin{align}
 B_t&:=B_{\psi_t^u}=\int_0^tU_s^u\rho_u\,ds,
 \label{eq:DR-abstract-source-action}\\
 \shortintertext{and}
 Z_t(y)&:=\exp\{B_t(T_{-t}^uy)\},\qquad A_t(y):=-\log Z_t(y).
 \label{eq:DR-one-flow-target-definition}
\end{align}
The integral exists in the uniform norm by
\eqref{eq:DR-score-time-Lipschitz}.  Let $F$ be a bounded measurable
function on $\mathcal W$.

\begin{lemma}\label{lem:DR-integrated-divergence}
For $t\in\mathbb R$,
\begin{align}
 \int F(T_t^uh)e^{-B_t(h)}\,d\widetilde\nu_\kappa(h)
 &=\int F\,d\widetilde\nu_\kappa,
 \label{eq:DR-abstract-source-identity}\\
 \shortintertext{and}
 \frac{d(T_t^u)_\#\widetilde\nu_\kappa}{d\widetilde\nu_\kappa}(y)
 &=Z_t(y)=\exp\left\{\int_{-t}^0U_s^u\rho_u(y)\,ds\right\},
 \label{eq:DR-abstract-target-density}\\
 \shortintertext{with}
 e^{-|t|\|\rho_u\|_{\sup}}&\leq Z_t(y)\leq e^{|t|\|\rho_u\|_{\sup}}.
 \label{eq:DR-abstract-two-sided-bound}
\end{align}
The derivatives of $A_t$ exist in the uniform norm and satisfy
\begin{equation}\label{eq:DR-abstract-action-derivatives}
 \partial_tA_t=-U_{-t}^u\rho_u,\qquad
 \partial_t^2A_t=U_{-t}^uD_u\rho_u.
\end{equation}
\end{lemma}

\begin{proof}
We first take $F\in\mathcal A_{\mathrm{fl}}$.
For an integer $N\geq1$, the Riemann sum
\begin{equation}\label{eq:DR-integrated-score-Riemann-sum}
 B_{t,N}:=\frac{t}{N}\sum_{j=0}^{N-1}U_{jt/N}^u\rho_u
 \in\mathcal A_{\mathrm{fl}}
\end{equation}
satisfies, by \eqref{eq:DR-bv-iterated-core} and the uniform fundamental
theorem of calculus,
\begin{align}
 \|B_{t,N}-B_t\|_{\sup}
 &\leq\frac{t^2}{2N}\|D_u\rho_u\|_{\sup},
 \label{eq:DR-integrated-score-Riemann-bound}\\
 \shortintertext{and}
 \|D_uB_{t,N}-(U_t^u\rho_u-\rho_u)\|_{\sup}
 &\leq\frac{t^2}{2N}\|D_u^2\rho_u\|_{\sup}.
 \label{eq:DR-integrated-score-derivative-bound}
\end{align}
For the second estimate we used
\begin{align*}
 D_uB_{t,N}&=\frac{t}{N}\sum_{j=0}^{N-1}U_{jt/N}^uD_u\rho_u,\\[3pt]
 \int_0^tU_s^uD_u\rho_u\,ds&=U_t^u\rho_u-\rho_u.
\end{align*}
Also $\|B_{t,N}\|_{\sup},\|B_t\|_{\sup}\leq|t|\|\rho_u\|_{\sup}$.
Choose $M>|t|\|\rho_u\|_{\sup}$ and set
\[
 p_K(x):=\sum_{k=0}^{K}\frac{(-x)^k}{k!},\qquad K\geq1.
\]
These polynomials satisfy
\begin{equation}\label{eq:DR-exponential-polynomial-approximation}
 \sup_{|x|\leq M}
 \{|p_K(x)-e^{-x}|+|p_K'(x)+e^{-x}|\}
 \longrightarrow0.
\end{equation}
The tests $F_{N,K}:=(U_t^uF)p_K(B_{t,N})$ belong to
$\mathcal A_{\mathrm{fl}}$.  The product rule and
Proposition~\ref{prop:DR-right-IBP} give
\begin{align}
 &\int\{(U_t^uD_uF)p_K(B_{t,N})
       +(U_t^uF)p_K'(B_{t,N})D_uB_{t,N}\}\,d\widetilde\nu_\kappa\notag\\
 &\qquad=\int(U_t^uF)p_K(B_{t,N})\rho_u\,d\widetilde\nu_\kappa.
 \label{eq:DR-weighted-core-IBP}
\end{align}
First let $K\to\infty$, then $N\to\infty$.  The three uniform estimates
above justify both limits and yield
\begin{equation}\label{eq:DR-weighted-IBP-limit}
 \int e^{-B_t}\{U_t^uD_uF-(U_t^uF)U_t^u\rho_u\}
       \,d\widetilde\nu_\kappa=0.
\end{equation}
Since $\partial_tB_t=U_t^u\rho_u$ in the uniform norm,
\begin{equation}\label{eq:DR-pointwise-weighted-constant}
 \frac d{dt}\int(U_t^uF)e^{-B_t}\,d\widetilde\nu_\kappa
 =\int e^{-B_t}\{U_t^uD_uF-(U_t^uF)U_t^u\rho_u\}
       \,d\widetilde\nu_\kappa=0.
\end{equation}
Evaluation at $t=0$ proves \eqref{eq:DR-abstract-source-identity} on
$\mathcal A_{\mathrm{fl}}$.

Invariance of $\mathcal A_{\mathrm{fl}}$ under $U_{\pm t}^u$ makes
$T_t^u$ and its inverse measurable for $\sigma(\mathcal A_{\mathrm{fl}})$.
The sums \eqref{eq:DR-integrated-score-Riemann-sum} and their uniform
limits are measurable for this sigma-algebra as well.
The two sides of \eqref{eq:DR-abstract-source-identity} are integrals
against finite measures.  The monotone class theorem therefore extends
the identity to all bounded $\sigma(\mathcal A_{\mathrm{fl}})$-measurable
functions.  In particular, it holds on the trace Borel sigma-algebra
generated by $\mathcal C$, by \eqref{eq:B1-generated-sigma}.

Changing variables $y=T_t^uh$ gives
\begin{equation}\label{eq:DR-pointwise-density-derivation}
 e^{-B_t(T_{-t}^uy)}\,d(T_t^u)_\#\widetilde\nu_\kappa(y)
 =d\widetilde\nu_\kappa(y).
\end{equation}
The flow property yields
\begin{equation}\label{eq:DR-pointwise-moving-interval}
 B_t(T_{-t}^uy)=\int_0^tU_{s-t}^u\rho_u(y)\,ds
 =\int_{-t}^0U_s^u\rho_u(y)\,ds.
\end{equation}
This proves the density formula and its two-sided bound.  The measures
are equivalent, so their completions agree.  Applying the same argument
to $-t$ extends both $T_t^u$ and its inverse measurably to the completed
space; the source identity follows for every bounded measurable $F$.
Finally,
\begin{equation}\label{eq:DR-abstract-target-action}
 A_t=-\int_{-t}^0U_s^u\rho_u\,ds.
\end{equation}
Lemma~\ref{lem:DR-flow-core-uniform-differentiability} gives both uniform
derivatives in \eqref{eq:DR-abstract-action-derivatives}.
\end{proof}

\subsection{Finite compositions and the density Hessian}

Retain $\Phi$, its partial compositions $\Phi_j$, and $B_\Phi$ from
\eqref{eq:intro-finite-composition}--\eqref{eq:intro-integrated-response}.
Define
\begin{equation}\label{eq:DR-composition-target-action}
 Z_\Phi(y):=\exp\{B_\Phi(y\circ\Phi^{-1})\},\qquad
 A_\Phi(y):=-\log Z_\Phi(y).
\end{equation}
For another chosen finite composition $\Psi$, use the concatenated
sequence of flows to define $B_{\Phi\circ\Psi}$.
Let $F$ be a bounded measurable function on $\mathcal W$.

\begin{proposition}\label{prop:DR-finite-RN}
We have
\begin{align}
 \int F(h\circ\Phi)e^{-B_\Phi(h)}\,d\widetilde\nu_\kappa(h)
 &=\int F\,d\widetilde\nu_\kappa,
 \label{eq:DR-composition-source-identity}\\
 \shortintertext{and}
 \frac{d(R_\Phi)_\#\widetilde\nu_\kappa}{d\widetilde\nu_\kappa}(y)
 &=Z_\Phi(y).
 \label{eq:DR-composition-target-density}
\end{align}
The measures are equivalent, with
\begin{equation}\label{eq:DR-finite-RN-uniform-bound}
 \max\{\|Z_\Phi\|_{\sup},\|Z_\Phi^{-1}\|_{\sup}\}
 \leq\exp\left\{\sum_{j=1}^{\ell}|t_j|\|\rho_{v_j}\|_{\sup}\right\}.
\end{equation}
For $h,y\in\mathcal W$,
\begin{align}
 B_{\Phi\circ\Psi}(h)&=B_\Phi(h)+B_\Psi(h\circ\Phi),
 \label{eq:DR-source-cocycle}\\
 \shortintertext{and}
 Z_{\Phi\circ\Psi}(y)&=Z_\Phi(y\circ\Psi^{-1})Z_\Psi(y).
 \label{eq:DR-target-cocycle}
\end{align}
\end{proposition}

\begin{proof}
For $0\leq j\leq\ell$, write $h_j=h\circ\Phi_j$, and put $B_{\Phi_0}=0$.
The definition of the increment gives
\begin{equation}\label{eq:DR-composition-orbit-recursion}
 \begin{aligned}
 h_j&=h_{j-1}\circ\psi_{t_j}^{v_j},\\[3pt]
 B_{\Phi_j}(h)&=B_{\Phi_{j-1}}(h)+B_{\psi_{t_j}^{v_j}}(h_{j-1}).
 \end{aligned}
\end{equation}
Assume the source identity holds for $\Phi_{j-1}$ and apply it to
\[
 y\longmapsto F(y\circ\psi_{t_j}^{v_j})
               e^{-B_{\psi_{t_j}^{v_j}}(y)}.
\]
This function is bounded and measurable by
Lemma~\ref{lem:DR-integrated-divergence}, which also gives
\begin{align*}
 \int F(h\circ\Phi_j)e^{-B_{\Phi_j}(h)}\,d\widetilde\nu_\kappa(h)
 &=\int F(y\circ\psi_{t_j}^{v_j})
          e^{-B_{\psi_{t_j}^{v_j}}(y)}\,d\widetilde\nu_\kappa(y)\\
 &=\int F\,d\widetilde\nu_\kappa.
\end{align*}
Induction proves \eqref{eq:DR-composition-source-identity}.
Changing variables $y=h\circ\Phi$ gives
\[
 e^{-B_\Phi(y\circ\Phi^{-1})}\,d(R_\Phi)_\#\widetilde\nu_\kappa(y)
 =d\widetilde\nu_\kappa(y),
\]
which is \eqref{eq:DR-composition-target-density}.
The uniform bound follows from
\begin{equation}\label{eq:DR-composition-density-bound}
 |B_\Phi(h)|\leq\sum_{j=1}^{\ell}|t_j|\|\rho_{v_j}\|_{\sup}.
\end{equation}
Splitting the defining sum at $h\circ\Phi$ proves
\eqref{eq:DR-source-cocycle}.  Since
$(\Phi\circ\Psi)^{-1}=\Psi^{-1}\circ\Phi^{-1}$,
\begin{align*}
 \log Z_{\Phi\circ\Psi}(y)
 &=B_\Phi(y\circ\Psi^{-1}\circ\Phi^{-1})
   +B_\Psi(y\circ\Psi^{-1})\\
 &=\log Z_\Phi(y\circ\Psi^{-1})+\log Z_\Psi(y),
\end{align*}
proving \eqref{eq:DR-target-cocycle}.
\end{proof}

For $v,w\in H_{\mathrm{WP}}^{\mathrm{fin}}$, use $\Phi_{s,t}$ from
the introduction and set
\begin{equation}\label{eq:DR-canonical-two-flow}
 A_{s,t}:=A_{\Phi_{s,t}}=-\log Z_{\Phi_{s,t}}.
\end{equation}
All derivatives and remainders below are taken in the uniform norm on
$\mathcal W$.

\begin{corollary}\label{cor:DR-finite-Hessian}
As $(s,t)\to(0,0)$,
\begin{align}
 A_{s,t}
 &=-s\rho_v-t\rho_w+\frac{s^2}{2}D_v\rho_v\notag\\
 &\quad+stD_w\rho_v+\frac{t^2}{2}D_w\rho_w+o(s^2+t^2).
 \label{eq:DR-canonical-action-expansion}
\end{align}
The ordered mixed derivative satisfies
\begin{align}
 \left.\partial_t\left(
   \left.\partial_sA_{s,t}\right|_{s=0}\right)\right|_{t=0}
 &=D_w\rho_v,
 \label{eq:DR-canonical-mixed-Hessian}\\
 \shortintertext{and}
 \int\left.\partial_t\left(
   \left.\partial_sA_{s,t}\right|_{s=0}\right)\right|_{t=0}
      d\widetilde\nu_\kappa
 &=G_\kappa(v,w).
 \label{eq:DR-averaged-Hessian}
\end{align}
\end{corollary}

\begin{proof}
For $u\in H_{\mathrm{WP}}^{\mathrm{fin}}$ and $\tau\in\mathbb R$,
\eqref{eq:DR-abstract-target-action} gives
\begin{equation}\label{eq:DR-one-flow-action}
 A_{\psi_\tau^u}=-\int_{-\tau}^0U_q^u\rho_u\,dq.
\end{equation}
By Lemma~\ref{lem:DR-flow-core-uniform-differentiability} and
\eqref{eq:DR-bv-iterated-core},
\begin{align}
 U_q^u\rho_u-\rho_u-qD_u\rho_u
 &=\int_0^q(q-r)U_r^uD_u^2\rho_u\,dr,\notag\\
 \shortintertext{and hence}
 \|U_q^u\rho_u-\rho_u-qD_u\rho_u\|_{\sup}
 &\leq\frac{q^2}{2}\|D_u^2\rho_u\|_{\sup}.
 \label{eq:DR-base-core-uniform-second-bound}
\end{align}
With the remainder
\begin{equation}
 r_\tau^u
 :=-\int_{-\tau}^0
       (U_q^u\rho_u-\rho_u-qD_u\rho_u)\,dq,
\end{equation}
we obtain
\begin{align}
 A_{\psi_\tau^u}
 &=-\tau\rho_u+\frac{\tau^2}{2}D_u\rho_u+r_\tau^u,
 \label{eq:DR-one-flow-second-order}\\
 \shortintertext{with}
 \|r_\tau^u\|_{\sup}
 &\leq\frac{|\tau|^3}{6}\|D_u^2\rho_u\|_{\sup}.
 \label{eq:DR-one-flow-third-remainder}
\end{align}
The target cocycle \eqref{eq:DR-target-cocycle} gives
\begin{equation}\label{eq:DR-canonical-target-cocycle}
 A_{s,t}=U_{-t}^wA_{\psi_s^v}+A_{\psi_t^w}.
\end{equation}
Substitution of these expansions into \eqref{eq:DR-canonical-target-cocycle}
gives the quadratic expression in \eqref{eq:DR-canonical-action-expansion}
with remainder
\begin{align}
 r_{s,t}
 ={}&-s\{U_{-t}^w\rho_v-\rho_v+tD_w\rho_v\}\notag\\
 &+\frac{s^2}{2}\{U_{-t}^wD_v\rho_v-D_v\rho_v\}
   +U_{-t}^wr_s^v+r_t^w.
 \label{eq:DR-two-flow-remainder-identity}
\end{align}
\begin{samepage}
The uniform fundamental theorem of calculus gives
\begin{align}
 \|U_{-t}^w\rho_v-\rho_v+tD_w\rho_v\|_{\sup}
 &\leq\frac{t^2}{2}\|D_w^2\rho_v\|_{\sup},
 \label{eq:DR-shift-bv-second}\\
 \shortintertext{and}
 \|U_{-t}^wD_v\rho_v-D_v\rho_v\|_{\sup}
 &\leq|t|\|D_wD_v\rho_v\|_{\sup}.
 \label{eq:DR-shift-Dvbv-first}
\end{align}
\end{samepage}
Since $U_{-t}^w$ preserves the uniform norm,
\begin{align}
 \|r_{s,t}\|_{\sup}
 \leq{}&\frac{|s|t^2}{2}\|D_w^2\rho_v\|_{\sup}
       +\frac{s^2|t|}{2}\|D_wD_v\rho_v\|_{\sup}\notag\\
 &+\frac{|s|^3}{6}\|D_v^2\rho_v\|_{\sup}
       +\frac{|t|^3}{6}\|D_w^2\rho_w\|_{\sup}.
 \label{eq:DR-canonical-explicit-remainder}
\end{align}
Each term divided by $s^2+t^2$ tends to zero, proving
\eqref{eq:DR-canonical-action-expansion}.

To establish the ordered derivative, differentiate the exact identity
\eqref{eq:DR-canonical-target-cocycle}:
\begin{align}
 \left.\partial_sA_{s,t}\right|_{s=0}&=-U_{-t}^w\rho_v,
 \label{eq:DR-canonical-first-then-shift}\\
 \shortintertext{and hence}
 \left.\partial_t\left(
   \left.\partial_sA_{s,t}\right|_{s=0}\right)\right|_{t=0}
 &=D_w\rho_v.
 \label{eq:DR-canonical-ordered-action-response}
\end{align}
Finally, $\rho_v\in\mathcal A_0$, so
Proposition~\ref{prop:DR-right-IBP} and
Theorem~\ref{thm:intro-Fisher-Kahler} give
\begin{equation}\label{eq:DR-Hessian-IBP}
 \begin{aligned}
 \int D_w\rho_v\,d\widetilde\nu_\kappa
 &=\int\rho_v\rho_w\,d\widetilde\nu_\kappa\\
 &=G_\kappa(v,w)
 =\frac{c_{\mathrm L}}{12}(v,w)_{\mathrm{WP}}+4(v,w)_{\mathrm{VK}}.
 \end{aligned}
\end{equation}
\end{proof}

\subsection{Smooth variation of the Liouville--capacity action}

Fix $\mu\in L^\infty(\mathbb D)$ with
$\operatorname{supp}\mu\Subset\mathbb D\setminus\{0\}$.
Let $\widetilde\Phi_\zeta$ be the circle-preserving Ahlfors--Bers family
with interior coefficient $\zeta\mu$, normalized to fix $0,1,\infty$
after reflection across $\mathbb S^1$, for $|\zeta|\|\mu\|_\infty<1$.
When the expansion exists, define $R_\mu$ and $\bar R_\mu$ by
\[
 F(h\circ\widetilde\Phi_\zeta)
 =F(h)+\zeta R_\mu F(h)+\bar\zeta\,\bar R_\mu F(h)+o(|\zeta|).
\]
For a holomorphic function $p$ on $\mathbb D\setminus\{0\}$, write
\begin{equation}\label{eq:DR-holomorphic-pairing}
 (p,\mu):=\frac1\pi\int_{\mathbb D}p(z)\mu(z)\,dA(z).
\end{equation}
\begin{samepage}
For $h=g^{-1}\circ f\in\mathcal W$, put
\begin{align}
 \sigma_h(\mu)&:=-(\mathcal Sf,\mu),
 \label{eq:DR-sigma-pairing}\\
 \tau_h(\mu)&:=-\frac1\pi\int_{\mathbb D}
 \left(\frac{f'(z)^2}{f(z)^2}-\frac1{z^2}\right)\mu(z)\,dA(z),
 \label{eq:DR-tau-pairing}\\
 \shortintertext{and}
 \rho_h(\mu)&:=\frac{c_{\mathrm L}}{12}\sigma_h(\mu)+\tau_h(\mu)
 =-(q_f,\mu).
 \label{eq:DR-rho-pairing}
\end{align}
\end{samepage}

\Needspace{6\baselineskip}
\begin{samepage}
Fix a smooth circle diffeomorphism $h\in\mathcal W$ with normalized pair
$(f,g)$, and write
$\Gamma=f(\mathbb S^1)$ and $\Omega=f(\mathbb D)$.

\begin{lemma}\label{lem:DR-smooth-complex-variation}
\begin{equation}\label{eq:DR-smooth-complex-variation-identities}
 R_\mu S_1(h)=2\pi\sigma_h(\mu),\qquad
 2R_\mu k(h)=\tau_h(\mu).
\end{equation}
\end{lemma}
\end{samepage}

\begin{proof}
The physical compensating map $\Psi_\zeta$ is chosen so that
$\Psi_\zeta\circ f\circ\widetilde\Phi_\zeta$ is conformal on $\mathbb D$
and $\Psi_\zeta\circ g$ is conformal on $\mathbb D^*$.
By \eqref{eq:DR-physical-pushforward}, its infinitesimal Beltrami
differential is
\begin{equation}\label{eq:DR-SW-physical-Beltrami}
 \begin{aligned}
 \beta(f(z))&=-\mu(z)\frac{f'(z)}{\overline{f'(z)}},
 \quad\text{for a.e. }z\in\mathbb D,\\[3pt]
 \beta&=0\quad\text{on }\mathbb C\setminus\Omega.
 \end{aligned}
\end{equation}
The compensated pair has welding $h\circ\widetilde\Phi_\zeta$ and curve
$\Psi_\zeta(\Gamma)$.  Define
\begin{equation}\label{eq:DR-SW-external-variation-data}
 J_\mu:=\int_\Omega\beta(w)\mathcal S(f^{-1})(w)\,dA(w).
\end{equation}
Let $\Psi_s^\beta$ solve the Beltrami equation with coefficient $s\beta$
and fix $0,1,\infty$, where $s\in\mathbb R$ and
$|s|\|\beta\|_\infty<1$.
The support of $\beta$ is compactly contained in $\Omega$, away from
$\Gamma$, and the smooth curve $\Gamma$ has finite Loewner energy.
Thus the first-variation formula of
\cite[Theorem~1.4]{SW24} applies and yields
\begin{equation}\label{eq:DR-SW-real-S1-variation}
 \pi\left.\frac d{ds}\right|_{s=0}I^L(\Psi_s^\beta(\Gamma))
 =-4\operatorname{Re}J_\mu.
\end{equation}
The two families have the same infinitesimal curve deformation.
Using $S_1=\pi I^L$ from \cite[Theorem~1.2]{SW24} and applying the
formula to $\beta$ and $i\beta$ gives
\begin{align}
 (R_\mu+\bar R_\mu)S_1&=-4\operatorname{Re}J_\mu,
 \label{eq:DR-SW-polarization-real}\\
 \shortintertext{and}
 i(R_\mu-\bar R_\mu)S_1&=4\operatorname{Im}J_\mu.
 \label{eq:DR-SW-polarization-imag}
\end{align}
Hence
\begin{equation}\label{eq:DR-SW-complex-S1-variation}
 R_\mu S_1=-2J_\mu.
\end{equation}
The inverse-Schwarzian identity and the area formula give
\begin{align}
 J_\mu
 &=\int_{\mathbb D}
 \left(-\mu\frac{f'}{\overline{f'}}\right)
 \left(-\frac{\mathcal Sf}{f'^2}\right)|f'|^2\,dA\notag\\
 &=\int_{\mathbb D}\mathcal Sf\,\mu\,dA,
 \label{eq:DR-SW-pullback}
\end{align}
so
\begin{equation}\label{eq:DR-TT-potential-variation}
 R_\mu S_1(h)=-2\pi(\mathcal Sf,\mu)=2\pi\sigma_h(\mu).
\end{equation}
The second identity is the capacity response
\eqref{eq:DR-capacity-complex-variation}.
\end{proof}

\begin{samepage}
Let $n\geq2$ and $v\in H_{\mathrm{WP}}^{\mathrm{fin}}$.

\begin{lemma}\label{lem:DR-smooth-stress-variation}
At every smooth circle diffeomorphism in $\mathcal W$,
\begin{equation}\label{eq:DR-smooth-mode-variations}
 D_{n,c}H_\kappa=\operatorname{Im}t_n,\qquad
 D_{n,s}H_\kappa=\operatorname{Re}t_n,
\end{equation}
and consequently
\begin{equation}\label{eq:DR-smooth-general-variation}
 D_vH_\kappa=\rho_v.
\end{equation}
\end{lemma}
\end{samepage}

\begin{proof}
By \eqref{eq:intro-smooth-action} and
Lemma~\ref{lem:DR-smooth-complex-variation},
\begin{align}
 R_\mu H_\kappa
 &=\frac{c_{\mathrm L}}{24\pi}R_\mu S_1+2R_\mu k
 =\rho_h(\mu),
 \label{eq:DR-complex-potential-variation}\\
 \shortintertext{and therefore}
 (R_\mu+\bar R_\mu)H_\kappa
 &=2\operatorname{Re}\rho_h(\mu),
 \label{eq:DR-real-potential-variation}
\end{align}
since $H_\kappa$ is real-valued.
The Fourier representatives from
\eqref{eq:DR-Fourier-Cauchy-vector} satisfy
\eqref{eq:DR-real-cos-response}--\eqref{eq:DR-real-sin-response}.
For the rotation $r_\alpha$, the normalized pair of $h\circ r_\alpha$ is
\[
 f_\alpha(z)=e^{-i\alpha}f(e^{i\alpha}z),\qquad
 g_\alpha(z)=e^{-i\alpha}g(z).
\]
The definitions of $S_1$ and $k$ therefore give
$D_{\mathrm{rot}}H_\kappa=0$.
The pairings \eqref{eq:DR-Fourier-A-values-proof} now give
\begin{align*}
 D_{n,c}H_\kappa&=2\operatorname{Re}(-it_n/2)=\operatorname{Im}t_n,\\
 D_{n,s}H_\kappa&=2\operatorname{Re}(t_n/2)=\operatorname{Re}t_n.
\end{align*}
Real linearity and \eqref{eq:intro-score} prove
\eqref{eq:DR-smooth-general-variation}.
\end{proof}

At the round welding, \eqref{eq:DR-negative-Kirillov} gives
$\mathcal L_{-n}f|_{f(z)=z}=0$.  Lemma~\ref{lem:DR-stress-polynomial}
and \eqref{eq:DR-real-cos-Kirillov}--\eqref{eq:DR-real-sin-Kirillov}
then give
\begin{equation}\label{eq:DR-round-stress-first-jets}
 \left.D_{n,c}t_n\right|_{\id}=\frac i2C_{\kappa,n},\qquad
 \left.D_{n,s}t_n\right|_{\id}=\frac12C_{\kappa,n}.
\end{equation}
In particular, the smooth action has the positive second variations
\begin{equation}\label{eq:DR-smooth-potential-round-Hessian}
 \left.D_{n,c}^2H_\kappa\right|_{\id}
 =\left.D_{n,s}^2H_\kappa\right|_{\id}
 =\frac12C_{\kappa,n}>0.
\end{equation}
\subsection{Coefficient continuity along finite flows}

Let $h_j,h\in\mathcal W$ have normalized pairs $(f_j,g_j)$ and $(f,g)$,
respectively, with
\begin{equation}\label{eq:DR-coefficient-convergence}
 a_m(h_j)\longrightarrow a_m(h)\qquad(m\geq1).
\end{equation}
Fix $v\in H_{\mathrm{WP}}^{\mathrm{fin}}$ and a compact interval $I$.
For $v_1,\ldots,v_\ell\in H_{\mathrm{WP}}^{\mathrm{fin}}$ and compact
intervals $I_1,\ldots,I_\ell$, write
\[
 \Phi_{s_1,\ldots,s_\ell}
 :=\psi_{s_1}^{v_1}\circ\cdots\circ\psi_{s_\ell}^{v_\ell},
 \qquad (s_1,\ldots,s_\ell)\in I_1\times\cdots\times I_\ell.
\]

\begin{lemma}\label{lem:DR-finite-mode-continuity}
For every $m\geq1$,
\begin{align}
 \sup_{s\in I}|a_m(T_s^vh_j)-a_m(T_s^vh)|&\longrightarrow0,
 \label{eq:DR-flow-coefficient-continuity}\\
 \shortintertext{and}
 \sup_{\substack{s_r\in I_r\\1\leq r\leq\ell}}
 |a_m(h_j\circ\Phi_{s_1,\ldots,s_\ell})
       -a_m(h\circ\Phi_{s_1,\ldots,s_\ell})|&\longrightarrow0.
 \label{eq:DR-composition-coefficient-continuity}
\end{align}
\end{lemma}

\begin{proof}
All the transformed weldings lie in $\mathcal W$ by
Lemma~\ref{lem:B1-unique-welding-flow-stability}.
The coefficient bound $|a_m|\leq m+1$ gives, for $0<r<1$ and $M\geq1$,
\begin{align}
 \sup_{|z|\leq r}|f_j(z)-f(z)|
 &\leq\sum_{m=1}^M r^{m+1}|a_m(h_j)-a_m(h)|
       +2\sum_{m>M}(m+1)r^{m+1}.
 \label{eq:DR-coefficients-to-local-uniform}
\end{align}
Letting $j\to\infty$ and then $M\to\infty$ proves local uniform
convergence on $\mathbb D$; Cauchy's formula also gives $f_j'\to f'$.

The finite Laurent polynomial
\[
 V(z):=i\sum_{|n|\geq2}v_nz^{n+1}
\]
is holomorphic near $\mathbb S^1$ and equals $izv(\arg z)$ on the circle.
Multiply it by a smooth radial cutoff equal to one near $\mathbb S^1$
and zero near $0$, and let $\widehat\psi_s^v$ be the resulting smooth
flow on $\overline{\mathbb D}$.  For $s\in I$, a sufficiently small
boundary collar stays in the region where the vector field is
holomorphic.  Thus, for a compact set $K\Subset\mathbb D$ and $q_0<1$,
\begin{align}
 \widehat\psi_s^v|_{\mathbb S^1}&=\psi_s^v,
 \label{eq:DR-exact-finite-flow-extension}\\
 \shortintertext{with}
 \operatorname{supp}\mu_{\widehat\psi_s^v}&\subset K,\qquad
 \sup_{s\in I}\|\mu_{\widehat\psi_s^v}\|_\infty\leq q_0.
 \label{eq:DR-finite-flow-common-support}
\end{align}
Smooth dependence on the compact set $I\times\overline{\mathbb D}$ and
positivity of the flow Jacobian give
\[
 \inf_{\substack{s\in I\\z\in\overline{\mathbb D}}}
 \left\{|(\widehat\psi_s^v)_z(z)|^2
       -|(\widehat\psi_s^v)_{\bar z}(z)|^2\right\}>0.
\]
Consequently, the continuous ratio
$|(\widehat\psi_s^v)_{\bar z}|/|(\widehat\psi_s^v)_z|$ has a maximum
strictly below one, which gives the uniform constant $q_0$ above.

Put
\begin{equation}\label{eq:DR-finite-continuity-qjs}
 f_{j,s}^{\rm src}:=f_j\circ\widehat\psi_s^v
\end{equation}
and define a Beltrami coefficient by
\begin{equation}\label{eq:DR-finite-continuity-betajs}
 \beta_{j,s}(f_{j,s}^{\rm src}(z))
 :=-\mu_{f_{j,s}^{\rm src}}(z)
 \frac{(f_{j,s}^{\rm src})_z(z)}
 {\overline{(f_{j,s}^{\rm src})_z(z)}},\qquad z\in\mathbb D,
\end{equation}
extended by zero outside $f_j(\mathbb D)$.
Let $\Psi_{j,s}$ solve the Beltrami equation with coefficient
$\beta_{j,s}$ and fix $0,1,\infty$.  Put
$F_{j,s}:=\Psi_{j,s}\circ f_{j,s}^{\rm src}$.
The Beltrami composition formula gives
\begin{equation}\label{eq:DR-finite-continuity-cancellation}
 \mu_{F_{j,s}}=0.
\end{equation}
Hence $F_{j,s}$ is conformal, and we may define
\begin{equation}\label{eq:DR-finite-continuity-output}
 f_{j,s}^{\rm out}(z):=\frac{F_{j,s}(z)-F_{j,s}(0)}{F_{j,s}'(0)}.
\end{equation}
Use the subscript $\infty$ for these constructions with $f$ in place
of $f_j$.
The exterior map $\Psi_{j,s}\circ g_j$ is conformal as well.
Applying the same affine normalization to both maps preserves their
welding, so $f_{j,s}^{\rm out}$ is the normalized interior map of
$T_s^vh_j$.

Let $j_k\to\infty$ and $s_k\to s\in I$.
Choose a closed disk $K_1\Subset\mathbb D$ whose interior contains
$K$ and $\widehat\psi_q^v(K)$ for every $q\in I$.  Then
\begin{equation}\label{eq:DR-map-and-derivative-convergence}
 f_{j_k}\longrightarrow f,\qquad f_{j_k}'\longrightarrow f'
 \quad\text{uniformly on }K_1.
\end{equation}
For $w\in\operatorname{int}f(K_1)$, convergence of the univalent maps
and smooth dependence of the flow give
\[
 (\widehat\psi_{s_k}^v)^{-1}(f_{j_k}^{-1}(w))
 \longrightarrow(\widehat\psi_s^v)^{-1}(f^{-1}(w)).
\]
Substitution in \eqref{eq:DR-finite-continuity-betajs} yields
$\beta_{j_k,s_k}(w)\to\beta_{\infty,s}(w)$.
For $w\notin f(K_1)$, uniform convergence on a neighborhood of $K_1$
puts $w$ outside $f_{j_k}(K_1)$ eventually, and both coefficients
vanish.  The exceptional curve $f(\partial K_1)$ is analytic and has
planar area zero.  Therefore
\begin{equation}\label{eq:DR-beta-ae-stability}
 \begin{gathered}
 \beta_{j_k,s_k}\longrightarrow\beta_{\infty,s}
 \quad\text{almost everywhere with respect to }dA,\\[3pt]
 \sup_k\|\beta_{j_k,s_k}\|_\infty\leq q_0<1.
 \end{gathered}
\end{equation}
These are the hypotheses of \cite[Lemma~18]{AB60} for solutions fixing
$0,1,\infty$.  Consequently,
\begin{equation}\label{eq:DR-normalized-solution-stability}
 \Psi_{j_k,s_k}\longrightarrow\Psi_{\infty,s}
 \quad\text{locally uniformly on }\mathbb C.
\end{equation}
It follows that $F_{j_k,s_k}\to F_{\infty,s}$ locally uniformly on
$\mathbb D$.  These maps are conformal, so Cauchy's formula gives
$F_{j_k,s_k}'(0)\to F_{\infty,s}'(0)\ne0$.  Thus
\begin{equation}\label{eq:DR-normalized-interior-convergence}
 f_{j_k,s_k}^{\rm out}\longrightarrow f_{\infty,s}^{\rm out}
 \quad\text{locally uniformly on }\mathbb D.
\end{equation}
For any fixed $0<r<1$,
\begin{align}
 |a_m(T_{s_k}^vh_{j_k})-a_m(T_s^vh)|
 &\leq r^{-m-1}\sup_{|z|=r}
 |f_{j_k,s_k}^{\rm out}(z)-f_{\infty,s}^{\rm out}(z)|
 \longrightarrow0.
 \label{eq:DR-sequential-coefficient-stability}
\end{align}
The same argument with the constant sequence $f_j=f$ gives
$a_m(T_{s_k}^vh)\to a_m(T_s^vh)$.  Hence
\begin{align*}
 |a_m(T_{s_k}^vh_{j_k})-a_m(T_{s_k}^vh)|
 &\leq|a_m(T_{s_k}^vh_{j_k})-a_m(T_s^vh)|\\
 &\quad+|a_m(T_s^vh)-a_m(T_{s_k}^vh)|\longrightarrow0.
\end{align*}
Compactness of $I$ proves \eqref{eq:DR-flow-coefficient-continuity}.

For a convergent sequence of parameters
$(s_{1,k},\ldots,s_{\ell,k})\to(s_1,\ldots,s_\ell)$, apply
\eqref{eq:DR-sequential-coefficient-stability} successively.  At the
$p$th step, for every $m\geq1$,
\begin{equation}\label{eq:DR-composition-sequential-stability}
 a_m(h_{j_k}\circ\psi_{s_{1,k}}^{v_1}\circ\cdots\circ\psi_{s_{p,k}}^{v_p})
 \longrightarrow
 a_m(h\circ\psi_{s_1}^{v_1}\circ\cdots\circ\psi_{s_p}^{v_p}),
 \quad 1\leq p\leq\ell.
\end{equation}
Applying this also to the constant sequence $h_j=h$ and using compactness
of $I_1\times\cdots\times I_\ell$ proves
\eqref{eq:DR-composition-coefficient-continuity}.
\end{proof}

\begin{samepage}
For each $v\in H_{\mathrm{WP}}^{\mathrm{fin}}$, the score $\rho_v$ is a
polynomial in finitely many $a_m,\overline{a_m}$ by
Lemma~\ref{lem:DR-stress-polynomial}.  The bounds $|a_m|\leq m+1$
therefore give an integer $N_v$ and a constant $C_v<\infty$ such that
\begin{equation}\label{eq:DR-score-coefficient-continuity-bound}
 |\rho_v(h_1)-\rho_v(h_2)|
 \leq C_v\sum_{m=1}^{N_v}|a_m(h_1)-a_m(h_2)|,
 \qquad h_1,h_2\in\mathcal W.
\end{equation}
Thus both uniform convergence statements in
Lemma~\ref{lem:DR-finite-mode-continuity} hold with $a_m$ replaced by
$\rho_v$.
\par
\end{samepage}

\begin{samepage}
\subsection{Final proof of Theorem~\ref{thm:DR-radial-closure}}
\label{sec:proof-radial-action}

Fix $h\in\mathcal W$, a composition $\Phi$ as in
\eqref{eq:intro-finite-composition}, and a sequence $(\widehat h_N)$
specified before the theorem.  Let $I_j$ be the closed interval with
endpoints $0$ and $t_j$.
For each $N\geq1$ and $1\leq j\leq\ell$, the maps
$\widehat h_N\circ\Phi_{j-1}\circ\psi_s^{v_j}$ are smooth circle
diffeomorphisms in $\mathcal W$ for $s\in I_j$, by
Lemma~\ref{lem:B1-unique-welding-flow-stability}.  Applying
\eqref{eq:DR-smooth-general-variation} at these maps gives
\[
 \frac d{ds}H_\kappa(\widehat h_N\circ\Phi_{j-1}\circ\psi_s^{v_j})
 =\rho_{v_j}(\widehat h_N\circ\Phi_{j-1}\circ\psi_s^{v_j}).
\]
The fundamental theorem of calculus therefore gives
\begin{align}
 &H_\kappa(\widehat h_N\circ\Phi)-H_\kappa(\widehat h_N)\notag\\
 &\quad=\sum_{j=1}^{\ell}
 \{H_\kappa(\widehat h_N\circ\Phi_j)
       -H_\kappa(\widehat h_N\circ\Phi_{j-1})\}\notag\\
 &\quad=\sum_{j=1}^{\ell}\int_0^{t_j}
 \rho_{v_j}(\widehat h_N\circ\Phi_{j-1}\circ\psi_s^{v_j})\,ds.
 \label{eq:DR-radial-composition-telescope}
\end{align}
\end{samepage}
The assumed coefficient convergence satisfies the hypotheses of
Lemma~\ref{lem:DR-finite-mode-continuity}.  Together with
\eqref{eq:DR-score-coefficient-continuity-bound}, this gives, for every $j$,
\begin{align}
 &\sup_{s\in I_j}
 |\rho_{v_j}(\widehat h_N\circ\Phi_{j-1}\circ\psi_s^{v_j})
       -\rho_{v_j}(h\circ\Phi_{j-1}\circ\psi_s^{v_j})|\notag\\
 &\quad\leq C_{v_j}\sum_{m=1}^{N_{v_j}}\sup_{s\in I_j}
 |a_m(\widehat h_N\circ\Phi_{j-1}\circ\psi_s^{v_j})
       -a_m(h\circ\Phi_{j-1}\circ\psi_s^{v_j})|
 \longrightarrow0.
 \label{eq:DR-radial-intermediate-stress}
\end{align}
Consequently,
\begin{align}
 &|H_\kappa(\widehat h_N\circ\Phi)-H_\kappa(\widehat h_N)-B_\Phi(h)|
 \notag\\
 &\quad\leq\sum_{j=1}^{\ell}|t_j|\sup_{s\in I_j}
 |\rho_{v_j}(\widehat h_N\circ\Phi_{j-1}\circ\psi_s^{v_j})
       -\rho_{v_j}(h\circ\Phi_{j-1}\circ\psi_s^{v_j})|
 \longrightarrow0.
 \label{eq:DR-analytic-increment-error}
\end{align}
This proves \eqref{eq:intro-analytic-approximation}.
The radial weldings $h^{(r)}$ are real-analytic circle diffeomorphisms
in $\mathcal W$, and \eqref{eq:intro-radial-map} gives
\begin{equation}\label{eq:DR-radial-coefficient-limit}
 a_m(h^{(r)})=r^ma_m(h)\longrightarrow a_m(h),\qquad m\geq1.
\end{equation}
The same estimate, applied as $r\uparrow1$, proves
\begin{equation}\label{eq:DR-radial-composition-limit}
 B_\Phi^{\mathrm{rad}}(h)
 =\sum_{j=1}^{\ell}\int_0^{t_j}
 \rho_{v_j}(h\circ\Phi_{j-1}\circ\psi_s^{v_j})\,ds
 =B_\Phi(h).
\end{equation}

Suppose $\Phi$ and $\Psi$ are two finite compositions representing the
same circle diffeomorphism.  On the radial weldings,
\begin{align}
 B_\Phi(h^{(r)})
 &=H_\kappa(h^{(r)}\circ\Phi)-H_\kappa(h^{(r)})\notag\\
 &=H_\kappa(h^{(r)}\circ\Psi)-H_\kappa(h^{(r)})
 =B_\Psi(h^{(r)}).
 \label{eq:DR-composition-independence-before-limit}
\end{align}
Letting $r\uparrow1$ gives $B_\Phi(h)=B_\Psi(h)$.

Finally, if $(f,g)$ is the normalized pair of $h$, then
$(f,g\circ r_{-\alpha})$ is the normalized pair of $r_\alpha\circ h$.
Thus, for all $h\in\mathcal W$,
\begin{equation}\label{eq:DR-exterior-mark-coefficients}
 a_m(r_\alpha\circ h)=a_m(h),\qquad
 \rho_v(r_\alpha\circ h)=\rho_v(h).
\end{equation}
For smooth circle diffeomorphisms $h\in\mathcal W$, the definitions of
$S_1$ and $k$ also give
\begin{equation}\label{eq:DR-exterior-mark-invariance}
 H_\kappa(r_\alpha\circ h)=H_\kappa(h).
\end{equation}
Since left and right composition commute,
\begin{align}
 B_\Phi(r_\alpha\circ h)
 &=\sum_{j=1}^{\ell}\int_0^{t_j}
 \rho_{v_j}\bigl(r_\alpha\circ(h\circ\Phi_{j-1}\circ\psi_s^{v_j})\bigr)
 \,ds\notag\\
 &=B_\Phi(h).
 \label{eq:DR-source-action-mark-invariance}
\end{align}
Changing an exterior rotational mark therefore changes neither a smooth
increment nor its limit.  This completes the proof of
Theorem~\ref{thm:DR-radial-closure}.

\subsection{Radial likelihood and covariance}

Let $\Phi$ be as in \eqref{eq:intro-finite-composition}, and let $F$ be
a bounded measurable function on $\mathcal W$.

\begin{corollary}\label{cor:DR-radial-likelihood}
\begin{align}
 \int F(h\circ\Phi)e^{-B_\Phi^{\mathrm{rad}}(h)}
 \,d\widetilde\nu_\kappa(h)
 &=\int F\,d\widetilde\nu_\kappa,
 \label{eq:DR-radial-source-RN}\\
 \shortintertext{and}
 \frac{d(R_\Phi)_\#\widetilde\nu_\kappa}{d\widetilde\nu_\kappa}(y)
 &=\exp\{B_\Phi^{\mathrm{rad}}(y\circ\Phi^{-1})\}.
 \label{eq:DR-radial-target-RN}
\end{align}
\end{corollary}

\begin{proof}
Substitute \eqref{eq:intro-radial-equality} into
\eqref{eq:DR-composition-source-identity} and
\eqref{eq:DR-composition-target-density}.
\end{proof}

For $v,w\in H_{\mathrm{WP}}^{\mathrm{fin}}$, retain
$\Phi_{s,t}$ and $A_{s,t}^{\mathrm{rad}}$ from the introduction.

\begin{corollary}\label{cor:DR-radial-Kahler-Hessian}
The ordered mixed derivative exists in the uniform norm and satisfies
\begin{align}
 \left.\partial_t\left(
 \left.\partial_s A_{s,t}^{\mathrm{rad}}\right|_{s=0}
 \right)\right|_{t=0}
 &=D_w\rho_v,
 \label{eq:DR-radial-ordered-derivative}\\
 \shortintertext{and}
 \int\left.\partial_t\left(
 \left.\partial_s A_{s,t}^{\mathrm{rad}}\right|_{s=0}
 \right)\right|_{t=0}d\widetilde\nu_\kappa
 &=G_\kappa(v,w).
 \label{eq:DR-radial-Kahler-Hessian}
\end{align}
\end{corollary}

\begin{proof}
By Theorem~\ref{thm:DR-radial-closure},
\[
 A_{s,t}^{\mathrm{rad}}(y)
 =-B_{\Phi_{s,t}}(y\circ\Phi_{s,t}^{-1})=A_{s,t}(y).
\]
Both identities follow from Corollary~\ref{cor:DR-finite-Hessian}.
\end{proof}

\clearpage
\appendix
\part*{Appendix}
\addcontentsline{toc}{section}{Appendix}
\section{Proofs of the right response identities}
\label{app:coefficient-response}

The Cauchy kernels displayed in \cite[Lemma~2.3]{BJ24} and
\cite[Eq.~(5.5)]{BJ25} have opposite signs.
We fix the convention $\bar\partial C_\mu=\mu$ and derive the response
from the actual right composition $h\mapsto h\circ\psi_t^v$.

\subsection{Cauchy representatives and normalized right derivatives}
\label{app:BJ-convention}

Retain the complex responses $R_\mu,\bar R_\mu$ and the pairings
\eqref{eq:DR-holomorphic-pairing}--\eqref{eq:DR-rho-pairing} from the
smooth-variation setup in Section~\ref{sec:integrated-response}.
For $\mu\in L^\infty(\mathbb D)$ with
$\operatorname{supp}\mu\Subset\mathbb D\setminus\{0\}$, extended by zero
outside $\mathbb D$, define
\begin{align}
 C_\mu(\xi)&:=\frac1\pi\int_{\mathbb D}
 \frac{\mu(z)}{\xi-z}\,dA(z),
 \label{eq:DR-Cauchy-convention-below}\\
 C_\mu^{\rm nor}(\xi)&:=C_\mu(\xi)-C_\mu(0)-\xi C_\mu'(0).
 \label{eq:DR-normalized-Cauchy-vector}
\end{align}
Let $\delta_0$ denote the Dirac mass at zero.  The distributional
identity $\bar\partial(1/(\pi z))=\delta_0$ gives
$\bar\partial C_\mu=\mu$.
Choose radii $0<r_-<r_+<1$ with
\[
 \frac{r_+}{(1-r_+)^2}<\frac14,
\]
and a smooth radial cutoff $\chi_0$ equal to zero on $|z|\leq r_-$ and
one on $|z|\geq r_+$.  For $n\geq1$, put
\begin{equation}\label{eq:DR-Fourier-Cauchy-vector}
 V_n(z):=\frac i2(z^{1-n}-z),\qquad
 \mu_{n,c}:=\bar\partial(\chi_0 V_n),\qquad
 \mu_{n,s}:=i\mu_{n,c}.
\end{equation}
The product $\chi_0 V_n$ is smooth at zero and vanishes there with its
first derivative; $\mu_{n,c}$ is supported in the annulus
$r_-\leq|z|\leq r_+$.  Since
\[
 \bar\partial(\chi_0V_n-C_{\mu_{n,c}})=0,
 \qquad \chi_0V_n-C_{\mu_{n,c}}=O(1+|z|),
\]
the difference is affine.  Its value and derivative at zero give
\[
 C_{\mu_{n,c}}^{\rm nor}=\chi_0 V_n,
 \qquad C_{\mu_{n,c}}^{\rm nor}=V_n\quad\text{near }\mathbb S^1.
\]
The normalized real boundary fields are $\cos(n\theta)-1$ and
$\sin(n\theta)$, respectively.  On observables admitting these responses,
\begin{align}
 D_{n,c}&=R_{\mu_{n,c}}+\bar R_{\mu_{n,c}}+D_{\mathrm{rot}},
 \label{eq:DR-real-cos-response}\\
 \shortintertext{and}
 D_{n,s}&=R_{\mu_{n,s}}+\bar R_{\mu_{n,s}}.
 \label{eq:DR-real-sin-response}
\end{align}

The pairings in \eqref{eq:DR-sigma-pairing}--\eqref{eq:DR-rho-pairing}
depend only on the boundary trace
$C_\mu^{\rm nor}|_{\mathbb S^1}$.
Let $\delta\mu$ be the difference of two representatives with the same
normalized boundary trace.  Then $C_{\delta\mu}^{\rm nor}=0$ on
$\mathbb S^1$, so
\[
 C_{\delta\mu}(\xi)=C_{\delta\mu}(0)+\xi C_{\delta\mu}'(0)
\]
on the outer component.  Analytic continuation to infinity and
\(C_{\delta\mu}(\xi)=O(\xi^{-1})\) force all three terms to vanish there.
Thus $C_{\delta\mu}=O(z^2)$ near zero and vanishes on the outer
component.

For $p=\mathcal Sf$ and
$p=f'^2/f^2-z^{-2}$, Stokes' formula on the annulus between zero and a
circle enclosing $\operatorname{supp}\delta\mu$ gives
\[
 (p,\delta\mu)
 =-\operatorname*{Res}_{z=0}\{p(z)C_{\delta\mu}(z)\}=0.
\]
Here the first function $p$ is holomorphic at zero and the second
has at most a simple pole there.

For $n\geq2$, choose $r_+<r<1$.  The cutoff $\chi_0$ vanishes near zero
and equals one on $|z|=r$, so Stokes' formula gives
\begin{align}
 \frac1\pi\int_{\mathbb D}q_f\mu_{n,c}\,dA
 &=\frac1{2\pi i}\oint_{|z|=r}q_f(z)V_n(z)\,dz\notag\\
 &=\operatorname*{Res}_{z=0}\{q_f(z)V_n(z)\}
 =\frac i2t_n.
 \label{eq:DR-Fourier-Stokes}
\end{align}
Indeed, the term $-iz/2$ in $V_n$ has zero residue because $q_f$ has no
$z^{-2}$ coefficient.  Since $\rho_h(\mu)=-(q_f,\mu)$, we obtain
\begin{equation}\label{eq:DR-Fourier-A-values-proof}
 \rho_h(\mu_{n,c})=-\frac i2t_n,
 \qquad
 \rho_h(i\mu_{n,c})=\frac12t_n.
\end{equation}

For $v$ as in \eqref{eq:intro-real-finite-field}, put
\begin{equation}\label{eq:B1-normalized-vector-field}
 u_v(\theta):=v(\theta)-v(0),
\end{equation}
and choose the compactly supported representative
\begin{equation}\label{eq:DR-finite-Fourier-Beltrami}
 \mu_v:=\sum_{n=2}^N(x_n\mu_{n,c}+y_n\mu_{n,s}).
\end{equation}
\begin{lemma}\label{lem:B1-BJ-interface}
\begin{align}
 \mathcal C
 &\subset\operatorname{Dom}(R_{\mu_v})
  \cap\operatorname{Dom}(\bar R_{\mu_v})
  \cap\operatorname{Dom}(D_{\mathrm{rot}}),
 \label{eq:B1-core-in-BJ-domain}\\
 \shortintertext{and}
 D_vF&=(R_{\mu_v}+\bar R_{\mu_v})F+v(0)D_{\mathrm{rot}}F,
 \qquad F\in\mathcal C.
 \label{eq:B1-core-BJ-identification}
\end{align}
\end{lemma}

\begin{proof}
On $z=e^{i\theta}$, reflection across the circle gives
\begin{equation}\label{eq:B1-BJ-an-bn}
 V_n(z)-z^2\overline{V_n(z)}
 =iz\{\cos(n\theta)-1\},
\end{equation}
and
\begin{equation}\label{eq:B1-BJ-cos-sin}
 iV_n(z)-z^2\overline{iV_n(z)}=iz\sin(n\theta).
\end{equation}
Both fields vanish at $z=1$.  For the family
$\widetilde\Phi_t$ normalized to fix $0,1,\infty$, write
\[
 \widetilde\Phi_t(e^{i\theta})=e^{i\widetilde\Psi_t(\theta)}.
\]
The normalized Cauchy vector in
\eqref{eq:DR-normalized-Cauchy-vector} determines its complex angular
velocity:
\begin{equation}\label{eq:B1-normalized-Fourier-generators}
 \xi_v(\theta)
 :=\frac{C_{\mu_v}^{\rm nor}(e^{i\theta})}{ie^{i\theta}}
 =\frac12\sum_{n=2}^N(x_n+iy_n)(e^{-in\theta}-1).
\end{equation}
The support of $\mu_v$ is separated from $\mathbb S^1$.
For $t$ near zero, the reflected coefficients lie in a fixed
$L^\infty$ ball of radius below one.  The Ahlfors--Bers parameter
theorem \cite[Theorem~11]{AB60} and Cauchy estimates on a fixed annulus give
\begin{equation}\label{eq:B1-BJ-boundary-expansion}
 \widetilde\Psi_t(\theta)
 =\theta+t\xi_v(\theta)+\bar t\,\overline{\xi_v(\theta)}
  +O(|t|^2)
 \quad\text{in }C^1(\mathbb T).
\end{equation}
In particular,
\begin{equation}\label{eq:B1-BJ-real-infinitesimal}
 \xi_v+\overline{\xi_v}
 =\sum_{n=2}^N\{x_n(\cos(n\theta)-1)+y_n\sin(n\theta)\}
 =u_v.
\end{equation}
The change-of-variables calculation in Lemma~\ref{lem:B1-J-response}, now
using \eqref{eq:B1-BJ-boundary-expansion}, gives uniformly in $h$
\begin{align}
 &J_{a,b}(h\circ\widetilde\Phi_t)-J_{a,b}(h)\notag\\
 &\quad=-\frac{t}{2\pi}\int_0^{2\pi}
   (a\xi_v)'(\theta)b(h(e^{i\theta}))\,d\theta\notag\\
 &\qquad\phantom{=}
  -\frac{\bar t}{2\pi}\int_0^{2\pi}
   (a\overline{\xi_v})'(\theta)b(h(e^{i\theta}))\,d\theta
  +O_{a,b}(|t|^2).
 \label{eq:B1-BJ-J-complex-expansion}
\end{align}
The defining expansion of $R_\mu,\bar R_\mu$ therefore gives
\begin{align}
 R_{\mu_v}J_{a,b}
 &=-\frac1{2\pi}\int_0^{2\pi}
   (a\xi_v)'(\theta)b(h(e^{i\theta}))\,d\theta,
 \label{eq:B1-BJ-R-on-J}\\
 \shortintertext{and}
 \bar R_{\mu_v}J_{a,b}
 &=-\frac1{2\pi}\int_0^{2\pi}
   (a\overline{\xi_v})'(\theta)b(h(e^{i\theta}))\,d\theta.
 \label{eq:B1-BJ-Rbar-on-J}
\end{align}
Equations \eqref{eq:B1-BJ-real-infinitesimal} and
\eqref{eq:B1-J-response} show that
\begin{equation}\label{eq:B1-BJ-real-sum-on-J}
 (R_{\mu_v}+\bar R_{\mu_v})J_{a,b}=D_{u_v}J_{a,b}.
\end{equation}
Write $J_j=J_{a^{(j)},b^{(j)}}$ and $J=(J_1,\ldots,J_m)$ for the
observables defining $F=\phi(J)\in\mathcal C$.  Their values range in
a fixed compact set.  Put $\Delta J=J(h\circ\widetilde\Phi_t)-J(h)$.
Taylor's formula on a compact neighborhood gives
\[
 |\phi(J+\Delta J)-\phi(J)-d\phi(J)[\Delta J]|
 \leq\frac12\sup\|d^2\phi\|\,|\Delta J|^2
 =O_F(|t|^2).
\]
Together with \eqref{eq:B1-BJ-J-complex-expansion}, this yields
\begin{equation}\label{eq:B1-core-BJ-expansion}
 F(h\circ\widetilde\Phi_t)
 =F(h)+tR_{\mu_v}F(h)+\bar t\,\bar R_{\mu_v}F(h)
  +O_F(|t|^2),
\end{equation}
uniformly in $h$.  Thus
\[
 \mathcal C\subset\operatorname{Dom}(R_{\mu_v})
 \cap\operatorname{Dom}(\bar R_{\mu_v}),
 \qquad
 (R_{\mu_v}+\bar R_{\mu_v})F=D_{u_v}F.
\]
Lemma~\ref{lem:B1-J-response} and the cylinder chain rule also give
$\mathcal C\subset\operatorname{Dom}(D_{\mathrm{rot}})$.
Since $v=u_v+v(0)$, real linearity proves
\eqref{eq:B1-core-BJ-identification}.
\end{proof}

\subsection{Proof of Lemma~\ref{lem:DR-stress-polynomial}}

The identities
\begin{align}
 \frac{f'}f
 &=\frac1z+\frac d{dz}\log\left(1+\sum_{m\geq1}a_mz^m\right),
 \label{eq:DR-log-derivative}\\
 \mathcal Sf&=\frac{f'''}{f'}-\frac32\left(\frac{f''}{f'}\right)^2
 \label{eq:DR-Schwarzian-algebra}
\end{align}
show that $[z^{n-2}]q_f$ depends only on $a_1,\ldots,a_n$.
For the radial maps $f^{(r)}$ in \eqref{eq:intro-radial-map}, the chain
rules give
\begin{equation}\label{eq:DR-stress-scaling}
 q_{f^{(r)}}(z)=r^2q_f(rz),
 \qquad t_n(f^{(r)})=r^nt_n(f).
\end{equation}
Give $a_m$ weight $m$.  Equation \eqref{eq:DR-stress-scaling} shows
that every monomial of $t_n$ has weight $n$.  The only monomial of weight $n$ containing
$a_n$ is $a_n$.  At
$f_\varepsilon(z)=z+\varepsilon a_nz^{n+1}$,
\begin{align}
 \left.\partial_\varepsilon\right|_0
 \left(\frac{f_\varepsilon'^2}{f_\varepsilon^2}-\frac1{z^2}\right)
 &=2na_nz^{n-2},
 \label{eq:DR-weight-linear}\\
 \shortintertext{and}
 \left.\partial_\varepsilon\right|_0\mathcal Sf_\varepsilon
 &=(n^3-n)a_nz^{n-2}.
 \label{eq:DR-central-linear}
\end{align}
Therefore
\begin{equation}\label{eq:DR-tn-triangular}
 t_n-C_{\kappa,n}a_n\in\mathbb C[a_1,\ldots,a_{n-1}].
\end{equation}
Since $\partial_{a_{m+n}}t_n=0$ for $m\geq1$,
\[
 \mathcal L_nt_n=\partial_{a_n}t_n=C_{\kappa,n}.
\]
For fixed nonnegative integers $j_1,\ldots,j_n$, write
$P=\partial_{a_1}^{j_1}\cdots\partial_{a_n}^{j_n}t_n$.
De Branges' theorem applies to every map in the stated normalized
univalent class and gives $|a_m|\leq m+1$ \cite{deBranges85}.  Hence
\[
 \sup_f|P(a_1(f),\ldots,a_n(f))|
 \leq\max_{|z_m|\leq m+1,\,1\leq m\leq n}
       |P(z_1,\ldots,z_n)|<\infty.
\]
This proves \eqref{eq:DR-level-one} and
\eqref{eq:DR-stress-polynomial-uniform-bound}.

\subsection{Proof of Lemma~\ref{lem:DR-exact-change-frame}}

\subsubsection{Construction of the deformed uniformizers}

Fix $n\geq1$ and $h\in\mathcal W$ with normalized pair $(f,g)$, and set
\[
 \psi_{c,t}:=\psi_t^{\cos(n\cdot)},
 \qquad
 \psi_{s,t}:=\psi_t^{\sin(n\cdot)}.
\]
Extend the two circle fields to $\mathbb D$ by
\[
 \widehat v_{n,c}(z)=\chi_0(z)V_n(z)+\frac i2(z^{n+1}-z)+iz,
 \qquad
 \widehat v_{n,s}(z)=i\chi_0(z)V_n(z)+\frac12(z^{n+1}-z).
\]
On $z=e^{i\theta}$, these satisfy
\[
 \widehat v_{n,c}(z)=iz\cos(n\theta),\qquad
 \widehat v_{n,s}(z)=iz\sin(n\theta),
 \qquad
 \bar\partial\widehat v_{n,a}=\mu_{n,a},\quad a\in\{c,s\}.
\]
Let $\widehat\psi_{a,t}$ be their flows for $|t|\leq\epsilon$.
They fix zero, preserve $\mathbb D$, and restrict to $\psi_{a,t}$ on
$\mathbb S^1$.  The fields are holomorphic near zero and the boundary.
After decreasing $\epsilon$, there is one
$K\Subset\mathbb D\setminus\{0\}$ such that
\begin{equation}\label{eq:DR-actual-extension-C1}
 \operatorname{supp}\mu_{\widehat\psi_{a,t}}\subset K,
 \qquad
 t\longmapsto\mu_{\widehat\psi_{a,t}}
 \quad\text{is $C^1$ in $L^\infty$}.
\end{equation}
The coefficients are smooth in $(t,z)$ and vanish on a fixed neighborhood
of $\partial K$.  A dot denotes differentiation in $t$; in particular,
$\dot\mu_{\widehat\psi_{a,0}}=\mu_{n,a}$.

Set \(f_{a,t}^{\rm src}=f\circ\widehat\psi_{a,t}\), extend
\(\beta_{a,t}\) by zero after defining
\begin{equation}\label{eq:DR-physical-coefficient}
 \beta_{a,t}(f_{a,t}^{\rm src}(z))
 :=-\mu_{f_{a,t}^{\rm src}}(z)
 \frac{(f_{a,t}^{\rm src})_z(z)}
      {\overline{(f_{a,t}^{\rm src})_z(z)}},
\end{equation}
and let $\Psi_{a,t}$ be the principal solution with coefficient
$\beta_{a,t}$.  Choose the compact set $K$ in
\eqref{eq:DR-actual-extension-C1} large enough that
\[
 \widehat\psi_{a,t}
 (\operatorname{supp}\mu_{\widehat\psi_{a,t}})\subset K
 \qquad(a\in\{c,s\},\ |t|\leq\epsilon).
\]
On this set, $f$ is holomorphic and $\inf_K|f'|>0$.
The change of variables and the fixed zero collar therefore give
\[
 t\longmapsto\beta_{a,t}
 \quad\text{in }C^1\bigl(({-}\epsilon,\epsilon);L^\infty(\mathbb C)\bigr),
 \qquad
 \dot\beta_{a,0}(f(z))
 =-\dot\mu_{\widehat\psi_{a,0}}(z)
   \frac{f'(z)}{\overline{f'(z)}}.
\]
In particular,
\[
 \operatorname{supp}\beta_{a,t}\subset f(K),\qquad
 \sup_{a,\,|t|\leq\epsilon}\|\beta_{a,t}\|_\infty<1.
\]
For a map $H$ conformal near zero with $H'(0)\ne0$, define
\[
 \operatorname{Norm}(H):=\frac{H-H(0)}{H'(0)}.
\]
The normalized interior map of
the deformed welding is
\begin{equation}\label{eq:DR-actual-normalized-family}
 F_{a,t}:=\operatorname{Norm}
  (\Psi_{a,t}\circ f_{a,t}^{\rm src}).
\end{equation}

The exact composition formula gives
\begin{equation}\label{eq:DR-exact-physical-cancellation}
 \mu_{\Psi_{a,t}\circ f_{a,t}^{\rm src}}=0,
 \qquad
 (\Psi_{a,t}\circ g)^{-1}\circ
 (\Psi_{a,t}\circ f_{a,t}^{\rm src})
 =h\circ\psi_{a,t}.
\end{equation}
The Ahlfors--Bers parameter theorem gives $C^1$ dependence of
$\Psi_{a,t}$ on $t$ in the coefficient-free regions.
The normalizing derivative in \eqref{eq:DR-actual-normalized-family}
is nonzero, so $F_{a,t}$ is $C^1$ in $t$ near zero and on the outer
analytic collar.  Cauchy's integral formula extends this regularity
throughout $\mathbb D$.  For $0<r<1$, it gives
\[
 a_m(F_{a,t})=\frac1{2\pi i}\oint_{|z|=r}
                 F_{a,t}(z)z^{-m-2}\,dz,
 \qquad
 \frac d{dt}a_m(F_{a,t})
 =\frac1{2\pi i}\oint_{|z|=r}
                 \dot F_{a,t}(z)z^{-m-2}\,dz.
\]
\subsubsection{The positive modes}

The reflected normalized vector is
\begin{equation}\label{eq:DR-reflected-vector}
 V_n^\#(z)=\frac i2(z^{n+1}-z).
\end{equation}
Put $B:=\dot\Psi_{c,0}$; complex linearity of the principal
linearized solution gives $\dot\Psi_{s,0}=iB$.
Subtract the source rotation $iz$ from $\widehat v_{n,c}$, and denote
the resulting unnormalized interior variations by $\delta_cf$ and
$\delta_sf$.  On the analytic collar,
\begin{samepage}
\begin{align}
 \delta_cf&=B\circ f+f'(V_n+V_n^\#),
 \label{eq:DR-raw-cos}\\
 \shortintertext{and}
 \delta_sf&=iB\circ f+f'(iV_n-iV_n^\#).
 \label{eq:DR-raw-sin}
\end{align}
\end{samepage}
Thus
\begin{equation}\label{eq:DR-positive-raw-cancellation}
 \delta_cf+i\delta_sf=i(z^{n+1}-z)f'.
\end{equation}
Moreover,
\begin{equation}\label{eq:DR-normalization-differential}
 (d\operatorname{Norm})_f[\delta H]
 =\delta H-\delta H(0)-f\delta H'(0).
\end{equation}
The right-rotation response is
\begin{equation}\label{eq:DR-A0-f}
 D_{\mathrm{rot}}f=i(zf'-f).
\end{equation}
The right side of \eqref{eq:DR-positive-raw-cancellation} is holomorphic
on $\mathbb D$, vanishes at zero, and has derivative $-i$ there.
Applying \eqref{eq:DR-normalization-differential} and restoring the
rotation gives
\begin{align}
 D_n^+f
 &=(d\operatorname{Norm})_f[i(z^{n+1}-z)f']+D_{\mathrm{rot}}f\notag\\
 &=i(z^{n+1}-z)f'+if+i(zf'-f)
 =iz^{n+1}f'.
 \label{eq:DR-positive-on-f}
\end{align}

\subsubsection{The negative modes}

On the analytic collar, the explicit source fields satisfy
\[
 \widehat v_{n,c}(w)-i\widehat v_{n,s}(w)
 =2V_n(w)+iw=iw^{1-n}.
\]
In the target coordinate $\xi=f(w)$, write this velocity as
\[
 u(f(w))=if'(w)w^{1-n}.
\]
Fix \(r<1\) inside the analytic collar so that
\(K\Subset\{|w|<r\}\), and orient \(|w|=r\) positively.
For $|z|<r$, define the normalized interior Cauchy projection by
\begin{align}
 Y_{\rm in}(f(z))
 &=\frac{1}{2\pi i}\oint_{|w|=r}
 u(f(w))f'(w)
 \left\{\frac1{f(w)-f(z)}-\frac1{f(w)}
 -\frac{f(z)}{f(w)^2}\right\}\,dw\notag\\
 &=\frac{i f(z)^2}{2\pi i}\oint_{|w|=r}
 \frac{f'(w)^2w^{1-n}}
 {f(w)^2\{f(w)-f(z)\}}\,dw.
 \label{eq:DR-negative-interior-projection}
\end{align}
Here we used
\[
 \frac1{\xi-x}-\frac1\xi-\frac{x}{\xi^2}
 =\frac{x^2}{\xi^2(\xi-x)}.
\]
On the analytic contour \(f(|w|=r)\), set
\(Y_{\rm out}:=Y_{\rm in}-u\).  The Sokhotski--Plemelj formula gives
holomorphic interior and exterior extensions with jump $u$.
For inversion and its action on vector fields, write
\[
 \iota(\xi)=\xi^{-1},
 \qquad
 (\iota_*U)(v):=-v^2U(v^{-1}).
\]
The three normalizations are
\begin{equation}\label{eq:DR-negative-projection-gauge}
 Y_{\rm in}(0)=Y_{\rm in}'(0)=0,
 \qquad (\iota_*Y_{\rm out})(0)=0.
\end{equation}
For two pairs with jump $u$ and normalizations
\eqref{eq:DR-negative-projection-gauge}, the difference glues across
the contour to a holomorphic vector field $A+B\xi+C\xi^2$ on the
Riemann sphere.  The normalization conditions give
\[
 A=0,\qquad B=0,\qquad
 (\iota_*(A+B\xi+C\xi^2))(0)=-C=0.
\]
For a fixed extension, differentiation of
\eqref{eq:DR-actual-normalized-family} gives
\begin{equation}\label{eq:DR-linearized-actual-family}
 \dot F_{a,0}
 =(d\operatorname{Norm})_f
 \left[\dot\Psi_{a,0}\circ f
       +f'\dot{\widehat\psi}_{a,0}\right],
\qquad
 \bar\partial\dot\Psi_{a,0}=\dot\beta_{a,0},
\end{equation}
where \(\dot\Psi_{a,0}\) is the unique principal linearized
Ahlfors--Bers solution.  Set
\[
 A_{a,t}(\xi):=
 \frac{\xi-(\Psi_{a,t}\circ f_{a,t}^{\rm src})(0)}
 {(\Psi_{a,t}\circ f_{a,t}^{\rm src})'(0)}.
\]
The accompanying exterior map is
$A_{a,t}\circ\Psi_{a,t}\circ g$.  Put
\[
 P_a(x):=\left.\partial_t\right|_{t=0}
          (A_{a,t}\circ\Psi_{a,t})(x).
\]
The support condition
$\operatorname{supp}\dot\beta_{a,0}\subset f(K)$ makes $P_a$ holomorphic
on $\widehat{\mathbb C}\setminus f(K)$, in particular across the
welding curve.  The two actual velocities satisfy
\begin{align*}
 \dot F_{a,0}(w)
 &=P_a(f(w))+f'(w)\widehat v_{n,a}(w),\\
 \left.\partial_t\right|_0(A_{a,t}\circ\Psi_{a,t}\circ g)(z)
 &=P_a(g(z)).
\end{align*}
Their $c-is$ combination has jump
\begin{align*}
 &(\dot F_{c,0}-i\dot F_{s,0})(w)
 -(P_c-iP_s)(f(w))\\
 &\qquad=f'(w)(\widehat v_{n,c}-i\widehat v_{n,s})(w)
 =u(f(w)).
\end{align*}
The affine normalization and the principal normalization at infinity give
\[
 \dot F_{a,0}(0)=\dot F_{a,0}'(0)=0,
 \qquad (\iota_*P_a)(0)=0.
\]
Thus the actual velocities have the same jump and the same three
normalizations as $Y_{\rm in},Y_{\rm out}$.  The uniqueness proved above
identifies their interior branch with
\eqref{eq:DR-negative-interior-projection}.  Consequently,
\begin{equation}\label{eq:DR-negative-on-f}
 (D_{n,c}-iD_{n,s})f=i\mathcal L_{-n}f.
\end{equation}
Two extensions in \eqref{eq:DR-actual-extension-C1} give normalized
pairs for the same $h\circ\psi_{a,t}\in\mathcal W$, by
Lemma~\ref{lem:B1-unique-welding-flow-stability}.
Conformal removability relates them by a common M\"obius map $M$ satisfying
\[
 M(0)=0,\qquad M'(0)=1,\qquad M(\infty)=\infty,
 \qquad\text{hence }M=\id.
\]

The coefficient expansion \eqref{eq:DR-negative-mode-polynomial}
shows that $\mathcal L_{-n}$ preserves $\mathcal P_{\mathrm{hol}}$.
Coefficient comparison in \eqref{eq:DR-positive-on-f} and
\eqref{eq:DR-negative-on-f}, followed by the Leibniz rule, proves
\eqref{eq:DR-positive-change-frame}--
\eqref{eq:DR-negative-change-frame}; adding and subtracting gives
\eqref{eq:DR-real-cos-Kirillov}--
\eqref{eq:DR-real-sin-Kirillov}.

\subsection{Proof of Lemma~\ref{lem:DR-flow-core-uniform-differentiability}}
\label{app:uniform-response}

It suffices to treat a generator $F=F_0\circ R_\Phi$ of
\eqref{eq:DR-flow-saturated-core}.  Put
\begin{equation}\label{eq:DR-flow-composition-family}
 \Xi_t:=\psi_t^v\circ\Phi.
\end{equation}
\subsubsection{Averaged observables}

Write $\Phi(e^{i\theta})=e^{i\phi(\theta)}$ and define
\begin{equation}\label{eq:DR-averaged-composition-weight}
 a_\Phi(u):=a(\phi^{-1}(u))(\phi^{-1})'(u).
\end{equation}
Changing variables $u=\phi(\theta)$ gives
\begin{equation}\label{eq:DR-averaged-composition-change-variable}
 J_{a,b}(h\circ\Phi)=J_{a_\Phi,b}(h).
\end{equation}
The same formula applies to $\Xi_t$.  Since
$t\mapsto a_{\Xi_t}$ is $C^1$ in every $C^r$ norm,
\[
 \begin{aligned}
 &\sup_h\left|
 \frac{J_{a,b}(h\circ\Xi_t)-J_{a,b}(h\circ\Xi_0)}{t}
 -\frac1{2\pi}\int_0^{2\pi}
       \left.\partial_ta_{\Xi_t}\right|_0(u)
       b(h(e^{iu}))\,du\right|\\
 &\qquad\leq\|b\|_\infty
 \left\|\frac{a_{\Xi_t}-a_{\Xi_0}}t
       -\left.\partial_ta_{\Xi_t}\right|_0\right\|_\infty
 \longrightarrow0.
 \end{aligned}
\]
The cylinder chain rule proves
\eqref{eq:DR-flow-core-uniform-differentiability} for the averaged
observables.

\subsubsection{Coefficient observables}

Extend each field in $\Xi_t$ by linear combinations of the disk fields
above and $iz$.  Let $\widehat\psi_t^v$ be the extended flow of $v$ and
$\widehat\Phi$ the corresponding composition.  Set
\[
 \widehat\Xi_t=\widehat\psi_t^v\circ\widehat\Phi.
\]
These maps fix zero and are conformal on fixed neighborhoods of zero
and the boundary.  Choose $K\Subset\mathbb D\setminus\{0\}$ so that
their Beltrami coefficients vanish outside $K$ and on a fixed neighborhood
of $\partial K$.
Smooth dependence of the disk flows and the Beltrami composition formula
give, after extension by zero,
$t\mapsto\mu_{\widehat\Xi_t}\in
 C^2([-\epsilon,\epsilon];W^{2,\infty}(\mathbb C))$ and
\begin{equation}\label{eq:DR-composition-extension-uniform-ball}
 \sup_{|t|\leq\epsilon}
 \sum_{j=0}^2
 \|\partial_t^j\mu_{\widehat\Xi_t}\|_{W^{2,\infty}}
 <\infty,
 \qquad
 \sup_{|t|\leq\epsilon}
 \|\mu_{\widehat\Xi_t}\|_\infty\leq q_0<1.
\end{equation}
For a normalized univalent map $f$, set
\begin{align}
 f_t^{\rm src}&:=f\circ\widehat\Xi_t,
 \label{eq:DR-composition-source-map}\\
 \beta_{f,t}(f_t^{\rm src}(z))
 &:=-\mu_{f_t^{\rm src}}(z)
   \frac{(f_t^{\rm src})_z(z)}{\overline{(f_t^{\rm src})_z(z)}},
 \qquad
 \beta_{f,t}=0\quad\text{off }f_t^{\rm src}(\mathbb D).
 \label{eq:DR-composition-physical-coefficient}
\end{align}
Choose a compact neighborhood $K_1\Subset\mathbb D\setminus\{0\}$ of
$\bigcup_{|t|\leq\epsilon}\widehat\Xi_t(K)$.
Koebe distortion and Cauchy estimates give constants $c,C>0$ such that
\[
 \inf_f\inf_{z\in K_1}|f'(z)|\geq c,
 \qquad
 \sup_f\sup_{z\in K_1}\sum_{j=1}^3|f^{(j)}(z)|\leq C.
\]
Together with \eqref{eq:DR-composition-extension-uniform-ball},
these bounds give $0<\delta<R<\infty$ with
\[
 \operatorname{supp}\beta_{f,t}\subset\{\delta\leq|w|\leq R\},
 \qquad \|\beta_{f,t}\|_\infty\leq q_0<1
 \quad\text{for all }f,t.
\]

To differentiate at a fixed physical point, let $\nabla$ denote the real
Jacobian, applied componentwise to complex functions, and put
\[
 b_{f,t}:=\beta_{f,t}\circ f_t^{\rm src},\qquad
 u_{f,t}:=(\nabla f_t^{\rm src})^{-1}\partial_t f_t^{\rm src}.
\]
The chain rule gives
\[
 (\partial_t\beta_{f,t})\circ f_t^{\rm src}
 =\partial_t b_{f,t}-u_{f,t}\cdot\nabla b_{f,t}.
\]
The fixed zero collar makes this formula valid after extension by zero.
A second differentiation, with all terms evaluated at $(f,t)$, yields
\begin{align*}
 (\partial_t^2\beta_{f,t})\circ f_t^{\rm src}
 &=(\partial_t-u_{f,t}\cdot\nabla)^2b_{f,t}\\
 &=\partial_t^2b_{f,t}
   -2u_{f,t}\cdot\nabla\partial_t b_{f,t}
   +\sum_{p,q=1}^2(u_{f,t})_p(u_{f,t})_q
       \partial_p\partial_q b_{f,t}\\
 &\quad+
 \bigl((u_{f,t}\cdot\nabla)u_{f,t}-\partial_tu_{f,t}\bigr)
       \cdot\nabla b_{f,t}.
\end{align*}
The distortion bounds above and the smooth source flows bound every term.
They also give joint continuity in $f$ and $t$, with $f$ in the
compact-open topology.  Thus
\begin{equation}\label{eq:DR-composition-beta-C1}
 \begin{gathered}
 \beta_{f,\cdot}\in C^2([-\epsilon,\epsilon];L^\infty(\mathbb C)),\\
 \sup_{f,\,|t|\leq\epsilon}
 \sum_{j=0}^2\|\partial_t^j\beta_{f,t}\|_\infty<\infty.
 \end{gathered}
\end{equation}

Let $\Psi_{f,t}$ be the principal solution with coefficient $\beta_{f,t}$
and put
\begin{equation}\label{eq:DR-composition-normalized-output}
 \begin{aligned}
 H_{f,t}&:=\Psi_{f,t}\circ f_t^{\rm src},\\
 f_{\Xi_t}^{\mathrm{out}}
 &:=\frac{H_{f,t}-H_{f,t}(0)}{H_{f,t}'(0)},\\
 g_{\Xi_t}^{\mathrm{out}}
 &:=\frac{\Psi_{f,t}\circ g-H_{f,t}(0)}{H_{f,t}'(0)}.
 \end{aligned}
\end{equation}
The Beltrami composition formula gives
\[
 \mu_{H_{f,t}}=0,
 \qquad
 (g_{\Xi_t}^{\mathrm{out}})^{-1}\circ f_{\Xi_t}^{\mathrm{out}}
 =h\circ\Xi_t.
\]
By Lemma~\ref{lem:B1-unique-welding-flow-stability} and uniqueness of
the normalized pair, these are the maps of $h\circ\Xi_t$.

Smooth dependence of the source maps on $t$ and the Koebe growth
bound give a radius $r_*\in(0,1)$ such that
\[
 \sup_{f,\,|t|\leq\epsilon}\sup_{|z|\leq r_*}
 |f_t^{\rm src}(z)|<\frac\delta2.
\]
On the coefficient-free disk $B(0,\delta)$, the Ahlfors--Bers
parameter theorem, \eqref{eq:DR-composition-beta-C1}, and Cauchy's formula
give joint continuity of the first two time derivatives of every interior jet.
Since the normalized univalent class is compact and $H_{f,t}'(0)\ne0$,
\[
 \inf_{f,\,|t|\leq\epsilon}|H_{f,t}'(0)|>0,
 \qquad
 \sup_{f,\,|t|\leq\epsilon}
 \sum_{j=0}^2|\partial_t^jH_{f,t}'(0)|<\infty.
\]
For each $m\geq1$, define
\begin{equation}\label{eq:DR-composition-output-coefficient}
 a_m^{\mathrm{out}}(f,t):=a_m(f_{\Xi_t}^{\mathrm{out}}).
\end{equation}
The quotient rule and Cauchy's formula now give
\[
 \sup_{f,\,|t|\leq\epsilon}
 \sum_{j=0}^2|\partial_t^ja_m^{\mathrm{out}}(f,t)|<\infty.
\]
Taylor's formula therefore yields
\begin{align}
 &\sup_f\left|
 \frac{a_m^{\mathrm{out}}(f,t)-a_m^{\mathrm{out}}(f,0)}{t}
 -\partial_ta_m^{\mathrm{out}}(f,0)
 \right|\notag\\
 &\qquad\leq\frac{|t|}{2}
 \sup_{f,\,|s|\leq|t|}
 |\partial_t^2a_m^{\mathrm{out}}(f,s)|
 \longrightarrow0.
 \label{eq:DR-composition-coefficient-uniform-difference}
\end{align}
\subsubsection{Exterior rotation characters}

For $m\in\mathbb Z$, retain the character $\chi_m$ from
\eqref{eq:DR-exterior-rotation-characters}.  The principal solution satisfies \(\Psi_{f,t}'(\infty)=1\).  Hence the common affine
normalization in \eqref{eq:DR-composition-normalized-output} changes the exterior
derivative to \(g'(\infty)/H_{f,t}'(0)\), and gives
\begin{equation}\label{eq:DR-composition-exterior-character}
 \chi_m(h\circ\Xi_t)
 =\chi_m(h)
  \left(\frac{|H_{f,t}'(0)|}{H_{f,t}'(0)}\right)^m.
\end{equation}
The lower bound and the first two derivative bounds for $H_{f,t}'(0)$
proved above imply
\[
 \sup_{h,\,|t|\leq\epsilon}
 \sum_{j=1}^2
 |\partial_t^j\chi_m(h\circ\Xi_t)|<\infty.
\]
Taylor's formula gives the required uniform expansion for each character.

For bounded $F,G$ with uniform responses,
\[
 \frac{(FG)\circ T_t^v-FG}{t}
 =(F\circ T_t^v)\frac{G\circ T_t^v-G}{t}
   +G\frac{F\circ T_t^v-F}{t}
 \longrightarrow F D_vG+G D_vF
\]
in the uniform norm.  Finite sums and complex conjugation preserve this
convergence as well, proving the assertion for $\mathcal A_{\mathrm{fl}}$.

\subsection{Sign of the right divergence}

Fix $n\geq2$ and use the divergence convention
\eqref{eq:DR-divergence-convention} for $\widetilde\nu_\kappa$.
Suppose that integration by parts on the algebra generated by
$a_m,\overline{a_m}$ gives, for one $\varepsilon\in\{-1,1\}$,
\begin{align}
 \operatorname{Div}_{\widetilde\nu_\kappa}D_{n,c}
 &=\varepsilon\operatorname{Im}t_n,
 \label{eq:DR-sign-cos-divergence}\\
 \operatorname{Div}_{\widetilde\nu_\kappa}D_{n,s}
 &=\varepsilon\operatorname{Re}t_n.
 \label{eq:DR-sign-sin-divergence}
\end{align}

\begin{proposition}\label{prop:DR-sign-forcing}
\begin{equation}\label{eq:DR-sign-fixed}
 \varepsilon=1.
\end{equation}
\end{proposition}

\begin{proof}
Lemmas~\ref{lem:DR-stress-polynomial} and
\ref{lem:DR-exact-change-frame} give
\[
 D_n^+t_n=i\mathcal L_nt_n=iC_{\kappa,n}.
\]
The assumed divergences imply
\begin{equation}\label{eq:DR-sign-complex-divergence}
 \operatorname{Div}_{\widetilde\nu_\kappa}D_n^+
 =\varepsilon(\operatorname{Im}t_n+i\operatorname{Re}t_n)
 =\varepsilon i\overline{t_n}.
\end{equation}
Since $t_n$ is a bounded coefficient polynomial, integration gives
\begin{equation}\label{eq:DR-sign-positive-forcing}
 0<C_{\kappa,n}
 =\frac1i\int D_n^+t_n\,d\widetilde\nu_\kappa
 =\varepsilon\int|t_n|^2\,d\widetilde\nu_\kappa.
\end{equation}
The last integral is nonnegative, so $\varepsilon=1$.
\end{proof}

\section{Physical Ward identities and frame measures}
\label{app:physical-ward}

Put
\begin{equation}\label{eq:DR-central-charges}
 c_{\mathrm m}:=26-c_{\mathrm L}
 =1-6\left(\frac{2}{\sqrt\kappa}-\frac{\sqrt\kappa}{2}\right)^2.
\end{equation}
For a measure $\nu$ and a derivation $X$ with divergence as in
\eqref{eq:DR-divergence-convention}, let $a$ be a multiplier such that
$aF$ remains an admissible test.  The product rule gives
\[
 \int aXF\,d\nu
 =\int X(aF)\,d\nu-\int F Xa\,d\nu
 =\int F\{a\operatorname{Div}_\nu X-Xa\}\,d\nu.
\]
Thus
\begin{equation}\label{eq:DR-multiplier-divergence}
 \operatorname{Div}_\nu(aX)
 =a\operatorname{Div}_\nu X-Xa.
\end{equation}

\subsection{The physical trace and capacity response}

Fix $h\in\mathcal W$ with normalized pair $(f,g)$.
Let $\mu\in L^\infty(\mathbb D)$ have
$\operatorname{supp}\mu\Subset\mathbb D\setminus\{0\}$.  Suppose that
radii $0<r_1<r_2<r_3$ satisfy
\[
 f(\operatorname{supp}\mu)\subset B(0,r_1),
 \qquad
 B(0,r_3)\subset f(\mathbb D).
\]
For the representatives \eqref{eq:DR-Fourier-Cauchy-vector}, Koebe's
distortion and one-quarter theorems allow
\(r_+/(1-r_+)^2<r_1<r_2<r_3<1/4\).

Write $\psi=f^{-1}$ and define
\begin{equation}\label{eq:DR-Neretin-P}
 w^2\mathcal S\psi(w)=\sum_{\ell\geq0}\pi_\ell(f)w^\ell,
 \qquad \pi_0=\pi_1=0.
\end{equation}
For $a\in\operatorname{Aut}(\mathbb D)$, the Schwarzian chain rule gives
$\mathcal S(a^{-1}\circ\psi)=\mathcal S\psi$.  Thus $\pi_\ell$ depends only
on the physical loop and its chosen interior component.
For $j\geq0$ and an unnormalized univalent map $\widehat f$, define
\begin{align}
 \widehat m_j(\widehat f)&:=\frac1\pi\int_{\mathbb D}
  \mu(z)\widehat f'(z)^2\widehat f(z)^j\,dA(z),
 \label{eq:DR-mj-raw-extension}\\
 m_j(f)&:=\widehat m_j(f).
 \label{eq:DR-mj}
\end{align}
Include the argument $\mu$ as $m_j(f;\mu)$ or
$\widehat m_j(\widehat f;\mu)$ when needed.  For $p\geq2$, set
\begin{equation}\label{eq:DR-two-real-physical-fields}
 u_{-p}^{(1)}(\xi):=-\xi^{1-p},
 \qquad
 u_{-p}^{(2)}(\xi):=-i\xi^{1-p}.
\end{equation}
Let \(\phi_{p,t}^{(a)}\) be the real local flow of \(u_{-p}^{(a)}\),
$a\in\{1,2\}$.  On differentiable loop observables define
\[
 L_{-p}^{(a)}F(\gamma):=\left.\frac d{dt}\right|_{t=0}
 F(\phi_{p,t}^{(a)}\gamma),\qquad
 L_{-p}^C:=\frac12\{L_{-p}^{(1)}-iL_{-p}^{(2)}\}.
\]

To fix the lift of $L_{-j-2}^C$ to unnormalized conformal maps, define
\begin{equation}\label{eq:DR-Schiffer-kernel}
 K_\psi(x,\zeta)
 :=\frac{\psi'(\zeta)^2}
 {\psi'(x)\{\psi(x)-\psi(\zeta)\}}-\frac1{x-\zeta}.
\end{equation}
Its interior Schiffer response is
\begin{equation}\label{eq:DR-Schiffer-variation}
 (L_{-j-2}^Cf)(z)
 =\frac1{2\pi i}\oint_{|\zeta|=r_2}
 K_\psi(f(z),\zeta)\zeta^{-j-1}\,d\zeta.
\end{equation}
The contour is counterclockwise; the formula extends analytically from
$f^{-1}(B(0,r_2))$ to $\mathbb D$.  The exterior response is the Cauchy
projection after inversion, with translation and dilation terms retained.
The conformal pushforward of $\mu$ and pullback of a physical
Beltrami coefficient $\beta$ are
\[
 (f_*\mu)(f(z)):=\mu(z)\frac{f'(z)}{\overline{f'(z)}},
 \qquad
 (f^*\beta)(z):=\beta(f(z))\frac{\overline{f'(z)}}{f'(z)}.
\]
Extend $f_*\mu$ by zero outside $f(\mathbb D)$.
For $t\in\mathbb C$ near zero, let $\Psi_t$ be the principal physical compensator for the source
deformation $\widetilde\Phi_t$, so that
$\Psi_t\circ f\circ\widetilde\Phi_t$ and $\Psi_t\circ g$ are conformal.
Write $\beta=\left.\partial_t\mu_{\Psi_t}\right|_{t=0}$ and define
\[
 V_{\mu,f}(\xi):=\frac1\pi\int_{\mathbb C}
             \frac{\beta(w)}{\xi-w}\,dA(w).
\]
Let $R_\mu^{\rm raw}$ be its Schiffer lift under
\eqref{eq:DR-Schiffer-variation}; series of such lifts are interpreted by
$C^2$ convergence of the physical vectors on compact subsets of
$\{|\xi|>r_1\}$.

\begin{lemma}\label{lem:DR-physical-trace}
\begin{equation}\label{eq:DR-physical-pushforward}
 \beta=-f_*\mu.
\end{equation}
For $|\xi|\geq r_2$, with uniform convergence,
\begin{align}
 V_{\mu,f}(\xi)
 &=-\frac1\pi\int_{\mathbb D}
  \frac{\mu(z)f'(z)^2}{\xi-f(z)}\,dA(z)
 =-\sum_{j\geq0}m_j(f)\xi^{-j-1},
 \label{eq:DR-physical-Laurent}
\end{align}
Consequently,
\begin{equation}\label{eq:DR-Rmu-variable-coefficients}
 R_\mu^{\rm raw}=\sum_{j\geq0}m_jL_{-j-2}^C.
\end{equation}
The contractions satisfy
\begin{align}
 \sum_{j\geq0}m_j\pi_{j+2}
 &=\sigma_h(\mu),
 \label{eq:DR-matter-contraction}\\
 \shortintertext{and}
 \left.\sum_{j\geq0}L_{-j-2}^C\widehat m_j
 \right|_{\widehat f=f}
 &=\frac{26}{12}(\mathcal Sf,\mu),
 \label{eq:DR-ghost-contraction}\\
 \shortintertext{hence}
 -\frac{c_{\mathrm m}}{12}\sum_{j\geq0}m_j\pi_{j+2}
 -\left.\sum_{j\geq0}L_{-j-2}^C\widehat m_j
 \right|_{\widehat f=f}
 &=\frac{c_{\mathrm L}}{12}\sigma_h(\mu).
 \label{eq:DR-algebraic-central-balance}
\end{align}
Finally, the complex capacity response is
\begin{equation}\label{eq:DR-capacity-complex-variation}
 2R_\mu k=\tau_h(\mu).
\end{equation}
\end{lemma}

\begin{proof}
Let $K=\operatorname{supp}\mu$ and put
$M_{f,\mu}=\pi^{-1}\|\mu\|_\infty\int_K|f'|^2\,dA$.  Then
\[
 |m_j(f)|\leq M_{f,\mu}r_1^j.
\]
For $r=0,1,2$, termwise differentiation gives
\[
 \sup_{|\xi|\geq r_2}
 \left|\partial_\xi^r\sum_{j>N}m_j(f)\xi^{-j-1}\right|
 \leq M_{f,\mu}r_2^{-r-1}
       \sum_{j>N}(j+1)\cdots(j+r)(r_1/r_2)^j
 \longrightarrow0,
\]
with empty product $1$ when $r=0$.
For the fixed finite-mode representatives, Koebe distortion bounds
$M_{f,\mu}$ uniformly over normalized univalent maps.  The Beltrami composition formula gives
\[
 \left.\partial_t\right|_{0}
 \mu_{\Psi_t\circ f\circ\widetilde\Phi_t}
 =\mu+(\beta\circ f)\frac{\overline{f'}}{f'}=0.
\]
Thus $\beta=-f_*\mu$.  Expanding its Cauchy vector in the physical
fields $-\xi^{-j-1}\partial_\xi$ gives
\eqref{eq:DR-Rmu-variable-coefficients}.

The inverse-Schwarzian chain rule gives
\begin{align}
 \sum_{j\geq0}m_j\pi_{j+2}
 &=\frac1\pi\int_{\mathbb D}
   \mu f'^2(\mathcal S\psi)\circ f\,dA\notag\\
 &=-\frac1\pi\int_{\mathbb D}\mu\mathcal Sf\,dA
 =\sigma_h(\mu).
 \label{eq:DR-matter-contraction-proof}
\end{align}
For the trace, differentiate the raw moments
\eqref{eq:DR-mj-raw-extension}, substitute
\eqref{eq:DR-Schiffer-variation}, and use
\begin{align}
 \sum_{j\geq0}f^j\zeta^{-j-1}&=(\zeta-f)^{-1},
 \label{eq:DR-geometric-one}\\
 \sum_{j\geq1}jf^{j-1}\zeta^{-j-1}&=(\zeta-f)^{-2}.
 \label{eq:DR-geometric-two}
\end{align}
The separation $|f(z)|\leq r_1<r_2$ on $\operatorname{supp}\mu$
justifies termwise integration, giving
\begin{align*}
 \left.\sum_{j\geq0}L_{-j-2}^C\widehat m_j
 \right|_{\widehat f=f}
 &=\frac1\pi\int_{\mathbb D}\mu
 \sum_{j\geq0}\left\{
 2f'f^j(L_{-j-2}^Cf)'
 +jf'^2f^{j-1}L_{-j-2}^Cf\right\}\,dA\\
 &=\frac1\pi\int_{\mathbb D}\mu f'^2
 \frac1{2\pi i}\oint_{|\zeta|=r_2}
 \left\{
 \frac{2\partial_xK_\psi(f(z),\zeta)}{\zeta-f(z)}
 +\frac{K_\psi(f(z),\zeta)}{(\zeta-f(z))^2}
 \right\}\,d\zeta\,dA(z).
\end{align*}
The term multiplied by $j$ is zero when $j=0$.  The apparent diagonal
singularity of $K_\psi$ is removable.  Cauchy's formula in $\zeta$ therefore gives
\begin{equation}\label{eq:DR-variable-coefficient-trace}
 \left.\sum_{j\geq0}L_{-j-2}^C\widehat m_j
 \right|_{\widehat f=f}
 =\frac1\pi\int_{\mathbb D}\mu f'^2
 \left.(2\partial_x+\partial_\zeta)K_\psi(x,\zeta)
 \right|_{x=\zeta=f(z)}dA(z).
\end{equation}
To compute the diagonal, put $A(\zeta)=\psi''(\zeta)/\psi'(\zeta)$ and
$B(\zeta)=\psi^{(3)}(\zeta)/\psi'(\zeta)$.  Taylor expansion at $x=\zeta$
gives
\[
 K_\psi(\zeta+\varepsilon,\zeta)
 =-\frac32A(\zeta)
   +\left(\frac74A(\zeta)^2-\frac23B(\zeta)\right)\varepsilon
   +O(\varepsilon^2).
\]
Thus, differentiating in each variable separately,
\begin{align*}
 \left.\partial_xK_\psi\right|_{x=\zeta}
 &=\frac74A^2-\frac23B,\\
 \left.\partial_\zeta K_\psi\right|_{x=\zeta}
 &=-\frac32(B-A^2)-\left(\frac74A^2-\frac23B\right)
   =-\frac14A^2-\frac56B.
\end{align*}
Since $\mathcal S\psi=B-\tfrac32A^2$, this proves the diagonal identity
of \cite[Lemma~5.5]{BJ25} in our convention:
\begin{equation}
 \left.(2\partial_x+\partial_\zeta)K_\psi\right|_{x=\zeta}
 =-\frac{13}{6}\mathcal S\psi.
 \label{eq:DR-kernel-diagonal-identity}
\end{equation}
Using $(\mathcal S\psi)\circ f\,f'^2=-\mathcal Sf$ in
\eqref{eq:DR-variable-coefficient-trace} proves
\eqref{eq:DR-ghost-contraction}.  Equation
\eqref{eq:DR-algebraic-central-balance} follows from
$c_{\mathrm L}=26-c_{\mathrm m}$.
For the capacity response, differentiation at the interior base point gives
\begin{align}
 C_\mu'(0)&=-\frac1\pi\int_{\mathbb D}\frac{\mu(z)}{z^2}\,dA(z),
 \label{eq:DR-source-derivative-zero}\\
 \shortintertext{and}
 V_{\mu,f}'(0)&=\frac1\pi\int_{\mathbb D}
  \mu(z)\frac{f'(z)^2}{f(z)^2}\,dA(z).
 \label{eq:DR-physical-derivative-zero}
\end{align}
Write $[t]$ and $[\bar t]$ for the two Wirtinger coefficients at zero.
The capacity $k$ is unchanged by right rotation of the welding.  We may
therefore choose a circle-preserving source family $\Phi_t^0$ fixing
$0$ and $\infty$, with first variation
\[
 \Phi_t^0(z)
 =z+t\{C_\mu(z)-C_\mu(0)\}
 -\bar t z^2\overline{C_\mu(1/\bar z)-C_\mu(0)}+o(|t|).
\]
The reflected term is $O(z^2)$ at zero because
$C_\mu(\xi)=O(\xi^{-1})$ at infinity.  Consequently,
\[
 [t]\log(\Phi_t^0)'(0)=C_\mu'(0),
 \qquad [\bar t]\log(\Phi_t^0)'(0)=0.
\]
Using the principal compensator for this source family, set
$F_t=\Psi_t\circ f\circ\Phi_t^0$ and $G_t=\Psi_t\circ g$.
Since $\Psi_t=\mathrm{id}+tV_{\mu,f}+o(|t|)$,
$[\bar t]\log F_t'(0)=0$ and the chain rule gives
\begin{equation}\label{eq:DR-interior-radius-variation}
 [t]\log F_t'(0)
 =V_{\mu,f}'(0)+C_\mu'(0)=-\tau_h(\mu).
\end{equation}
In particular,
\[
 \log|F_t'(0)|
 =\operatorname{Re}\{t(V_{\mu,f}'(0)+C_\mu'(0))\}+o(|t|).
\]
Taking the Wirtinger coefficient gives
\begin{equation}\label{eq:DR-log-modulus-holomorphic-coefficient}
 [t]\log|F_t'(0)|
 =\frac12[t]\log F_t'(0)
 =-\frac12\tau_h(\mu).
\end{equation}
On the exterior component, $V_{\mu,f}(\xi)=O(\xi^{-1})$ gives
\begin{equation}\label{eq:DR-exterior-radius-zero-variation}
 [t]\log|G_t'(\infty)|=0.
\end{equation}
The ratio in \eqref{eq:DR-capacity-potential-definition} is unchanged by
the common affine normalization of \(F_t\) and \(G_t\).  Therefore
\begin{align}
 R_\mu k
 &=[t]\log\left|\frac{G_t'(\infty)}{F_t'(0)}\right|
 =\frac12\tau_h(\mu).
 \label{eq:DR-capacity-complex-variation-expanded}
\end{align}
\end{proof}

The compensation and capacity identities require only
$\operatorname{supp}\mu\Subset\mathbb D\setminus\{0\}$.
The contraction calculations also apply to unnormalized univalent maps
and $\mu$ compactly supported in $\mathbb D$, whenever the three-radius
separation holds.

The trace in \eqref{eq:DR-ghost-contraction} differentiates the raw
moments $\widehat m_j$ before restriction to $f(0)=0$, $f'(0)=1$.
For a raw variation $U$, its normalized projection is
\[
 U^{\rm nor}f=Uf-(Uf)(0)-f\,(Uf)'(0).
\]
At a normalized map, differentiating the moments in
\eqref{eq:DR-mj-raw-extension} gives
\begin{align*}
 U\widehat m_j-U^{\rm nor}m_j
 &=\frac1\pi\int_{\mathbb D}\mu
 \left\{2(Uf)'(0)f'^2f^j
 +j f'^2f^{j-1}\bigl((Uf)(0)+f(Uf)'(0)\bigr)\right\}\,dA\\
 &=(j+2)(Uf)'(0)m_j+j(Uf)(0)m_{j-1},\qquad j\geq1,\\
 U\widehat m_0-U^{\rm nor}m_0&=2(Uf)'(0)m_0.
\end{align*}
\subsection{Frame spaces and the normalized measure}
\label{sec:DR-frame-spaces}

\subsubsection{Measurable frame coordinates}

For a Jordan curve $\gamma\subset\widehat{\mathbb C}$, let $U_+,U_-$
be its two complementary components in a specified order.  Set
\begin{equation}\label{eq:DR-ordered-spherical-loops}
 \mathcal J^{\rm or}
 :=\{(\gamma,U_+,U_-):
       \widehat{\mathbb C}\setminus\gamma=U_+\sqcup U_-\}.
\end{equation}
Let $d_{\rm sph}$ denote spherical distance.  Give compact sets the
corresponding Hausdorff topology, record each component by its closed
complement, and give conformal maps the topology of locally uniform
spherical convergence.  Equip these spaces with the Borel sigma-algebras
of the stated topologies.

If
$(\gamma,U_+,U_-)\in\mathcal J^{\rm or}$, let
$\operatorname{Fr}^{+}(\gamma,U_+)$ and
$\operatorname{Fr}^{-}(\gamma,U_-)$ be the spaces of conformal bijections
\begin{equation}\label{eq:DR-frame-torsors}
 \widehat f:\mathbb D\longrightarrow U_+,
 \qquad
 \widehat g:\mathbb D^*\longrightarrow U_-.
\end{equation}

Precomposition by the respective automorphism groups acts freely and
transitively on these frame spaces.
The full frame space is
\begin{equation}\label{eq:DR-full-frame-bundle}
 \mathcal F_{\rm fr}
 :=\{(\gamma,U_+,U_-,\widehat f,\widehat g):
      (\gamma,U_+,U_-)\in\mathcal J^{\rm or},\ 
      \widehat f\in\operatorname{Fr}^+(\gamma,U_+),\ 
      \widehat g\in\operatorname{Fr}^-(\gamma,U_-)\}.
\end{equation}

\begin{lemma}
\label{lem:DR-standard-Borel-frame-model}
The spaces $\mathcal J^{\rm or}$ and $\mathcal F_{\rm fr}$, with the
measurable structures specified above, are standard Borel spaces.
There are measurable frame selections
\[
 (\gamma,U_+,U_-)\longmapsto
 (f_\gamma^\circ,g_\gamma^\circ)
 \in\operatorname{Fr}^+(\gamma,U_+)\times
    \operatorname{Fr}^-(\gamma,U_-).
\]
Finite-jet and boundary evaluation, frame composition, and simultaneous
M\"obius postcomposition are measurable maps.
\end{lemma}

\begin{proof}
The hyperspace of nonempty compact subsets of
$\widehat{\mathbb C}$ is compact metrizable.  Simple closed curves form an
$F_{\sigma\delta}$ subset of the hyperspace of continua in $\mathbb R^3$
\cite[Section~1]{KO24}; restricting to the sphere shows that Jordan
compacta form a Borel subset here.

Fix a countable dense set $Q=\{q_1,q_2,\ldots\}$ in the sphere.
Two points of $Q$ lie in the same component exactly when a finite
concatenation of short geodesics with vertices in $Q$ joins them without
meeting $\gamma$.  This is a countable union of positive-distance
conditions, so the two component orders give measurable
closed-complement coordinates.
Choose the first point of $Q$ in each component.  These selections are
measurable because their level sets are
\[
 \{q_j\in U_\pm\}\cap\bigcap_{i<j}\{q_i\notin U_\pm\}.
\]

For the selected points $u\in U_+$ and $v\in U_-$, define
\[
 A_{u,v}(z)=
 \begin{cases}
 (z-u)/(z-v),&u,v\in\mathbb C,\\
 z-u,&u\in\mathbb C,\ v=\infty,\\
 1/(z-v),&u=\infty.
 \end{cases}
\]
It sends \((u,v)\) to \((0,\infty)\).  Put
\(U:=A_{u,v}(U_+)\).  In the Polish compact-open space of holomorphic maps,
consider the graph of the univalent map onto \(U\) that fixes
\(0\) and has positive derivative there.  Put
\[
 K_m:=\{|z|\leq1-m^{-1}\},
 \qquad
 U_n:=\{w\in U:d_{\rm sph}(w,\partial U)\geq n^{-1}\}.
\]
The range condition \(f(\mathbb D)=U\) is equivalent to
\[
 f(K_m)\subset U\quad(m\geq1),
 \qquad
 \forall n\ \exists m:\ U_n\subset f(K_m).
\]
Together with univalence and the normalization, these countable
conditions define a Borel graph with singleton fibers.  The projection
onto the domain coordinate is a measurable bijection between standard
Borel spaces; Lusin--Souslin gives its measurable inverse $U\mapsto f_U$.
Put $V:=\iota(A_{u,v}(U_-))$ and apply the same construction to $V$.
The two frame selections are
\[
 f_\gamma^\circ=A_{u,v}^{-1}\circ f_U,
 \qquad
 g_\gamma^\circ=A_{u,v}^{-1}\circ\iota\circ f_V\circ\iota.
\]

Cauchy's formula in local target coordinates gives measurable finite jets.
Carath\'eodory's theorem extends each frame continuously to the boundary, with
\[
 f(\zeta)=\lim_{m\to\infty}f((1-m^{-1})\zeta),
 \qquad \zeta\in\mathbb S^1.
\]
Thus boundary evaluation is jointly measurable in the frame and boundary
point; inversion gives the exterior assertion.  The same countable range
conditions used above show that $\mathcal F_{\rm fr}$ is a Borel subset
of the product of the loop and map spaces.  In the seed coordinates,
\[
 (\gamma,U_+,U_-,a,b)\longmapsto
 (\gamma,U_+,U_-,f_\gamma^\circ\circ a,g_\gamma^\circ\circ b)
\]
is a measurable bijection from
$\mathcal J^{\rm or}\times\operatorname{Aut}(\mathbb D)
\times\operatorname{Aut}(\mathbb D^*)$ onto $\mathcal F_{\rm fr}$.
Lusin--Souslin gives a measurable inverse.  Composition and simultaneous
M\"obius postcomposition are measurable in these coordinates.
\end{proof}

Let $G=\operatorname{PSL}_2(\mathbb C)$ act by
\begin{equation}\label{eq:DR-ordered-Mobius-action}
 M\cdot(\gamma,U_+,U_-):=(M\gamma,MU_+,MU_-),\qquad M\in G,
\end{equation}
and by simultaneous postcomposition on $\mathcal F_{\rm fr}$.
Let $\nu_\kappa^{\rm sph,or}$ be the symmetric two-sided lift to
$\mathcal J^{\rm or}$ of the M\"obius-invariant, $\sigma$-finite spherical
SLE loop measure \cite{Zhan21}, with multiplicative constant to be fixed below.

Compute derivatives and area in the chosen local coordinate, using
$\iota(z)=z^{-1}$ at infinity.  For a Jordan domain $U$ and $z\in U$,
choose a conformal bijection $f_z:\mathbb D\to U$ with $f_z(0)=z$ and define
\begin{equation}\label{eq:DR-conformal-area-measure}
 \operatorname{crad}_U(z):=|f_z'(0)|,
 \qquad \lambda_U(dz):=\frac{dA(z)}{\operatorname{crad}_U(z)^2}.
\end{equation}
The radius is independent of the rotation of $f_z$.  For a conformal bijection
$\Phi:U\to U'$, the chain rule gives
\begin{equation}\label{eq:DR-conformal-area-covariance}
 \operatorname{crad}_{U'}(\Phi(z))
 =|\Phi'(z)|\operatorname{crad}_U(z),
 \qquad dA(\Phi(z))=|\Phi'(z)|^2dA(z),
\end{equation}
so $\Phi_*\lambda_U=\lambda_{U'}$.

Fix $c_+,c_->0$.  On $\operatorname{Aut}(\mathbb D)$, take the Haar measure
\begin{equation}\label{eq:DR-disc-automorphism-Haar}
 dH_{\mathbb D}(a)
 =c_+\frac{dA(a(0))}{(1-|a(0)|^2)^2}
      \frac{d\arg a'(0)}{2\pi}.
\end{equation}
Let $H_{\mathbb D^*}$ be its image under $a\mapsto\iota\circ a\circ\iota$,
multiplied by $c_-/c_+$.  For the measurable seed $f_\gamma^\circ$,
\begin{equation}\label{eq:DR-crad-frame-coordinate}
 \operatorname{crad}_{U_+}(f_\gamma^\circ(a(0)))
 =|(f_\gamma^\circ)'(a(0))|(1-|a(0)|^2).
\end{equation}
Thus the pushforwards under composition with the seeds have the following
densities.  Put $z_+=\widehat f(0)$, $z_-=\widehat g(\infty)$, and let
$\alpha_\pm$ be the tangent arguments in the chosen local coordinates:
\begin{align}
 dH_\gamma^+(\widehat f)
 &=c_+\lambda_{U_+}(dz_+)\frac{d\alpha_+}{2\pi},
 \label{eq:DR-interior-frame-Haar}\\
 \shortintertext{and}
 dH_\gamma^-(\widehat g)
 &=c_-\lambda_{U_-}(dz_-)\frac{d\alpha_-}{2\pi}.
 \label{eq:DR-exterior-frame-Haar}
\end{align}
For every nonnegative measurable function $F_+$ on the space of interior frames,
\[
 \int F_+(\widehat f)\,dH_\gamma^+(\widehat f)
 =\int_{\operatorname{Aut}(\mathbb D)}
   F_+(f_\gamma^\circ\circ a)\,dH_{\mathbb D}(a).
\]
The right side is measurable in the ordered loop; the exterior formula
is identical with $g_\gamma^\circ$ and $H_{\mathbb D^*}$.
Changing a seed left-translates the group coordinate, so
\begin{equation}\label{eq:DR-frame-Haar-kernels}
 (\gamma,U_+,U_-)\longmapsto H_\gamma^\pm
\end{equation}
are measurable measure kernels independent of the seeds.

Define
\begin{equation}\label{eq:DR-doubly-marked-loop-measure}
 dM_\kappa^{\rm fr}(\gamma,U_+,U_-,\widehat f,\widehat g)
 :=d\nu_\kappa^{\rm sph,or}(\gamma,U_+,U_-)\,
    dH_\gamma^+(\widehat f)\,dH_\gamma^-(\widehat g).
\end{equation}
Let $E_m$ exhaust $\mathcal J^{\rm or}$ with finite
$\nu_\kappa^{\rm sph,or}$-measure, and let $K_n^\pm$ be compact
exhaustions of the two automorphism groups.  In the seed coordinates,
\[
 M_\kappa^{\rm fr}(E_m\times K_n^+\times K_\ell^-)
 =\nu_\kappa^{\rm sph,or}(E_m)
   H_{\mathbb D}(K_n^+)H_{\mathbb D^*}(K_\ell^-)<\infty.
\]
These sets cover $\mathcal F_{\rm fr}$, proving $\sigma$-finiteness.

\subsubsection{The normalized slice}

Define the normalized slice
\begin{equation}\label{eq:DR-Weil-slice}
 \Sigma:=\{(\gamma,U_+,U_-,f,g)\in\mathcal F_{\rm fr}:
 f(0)=0,\ f'(0)=1,\ g(\infty)=\infty\}.
\end{equation}
Let \(\Sigma_{\rm r}\subset\Sigma\) be the image of \(\mathcal W\) under
the normalized-frame map
\[
 h\longmapsto
 \bigl(f_h(\mathbb S^1),f_h(\mathbb D),g_h(\mathbb D^*),f_h,g_h\bigr).
\]
Use on \(\Sigma_{\rm r}\) the trace measurable structure inherited from the
ambient slice.  The map
\begin{equation}\label{eq:DR-ambient-welding-map}
 \operatorname{weld}:\Sigma\longrightarrow
 \operatorname{Homeo}_+(\mathbb S^1),
 \qquad
 \operatorname{weld}(\gamma,U_+,U_-,f,g):=g^{-1}\circ f,
\end{equation}
is measurable by Lemma~\ref{lem:DR-standard-Borel-frame-model}.
Conformal removability and the three normalizations make its restriction
to $\Sigma_{\rm r}$ a bijection onto $\mathcal W$.

Realize the normalized SLE frame and its independent uniform exterior
rotational mark on a standard Borel probability space, and denote their
law by $\nu_\kappa^{\rm nf}$.  Almost-sure removability
supplies a Borel conull event in the SLE sample space.  Its measurable image in
$\Sigma_{\rm r}$ is analytic and contains a Borel conull set $\Sigma_0$.
Since $\Sigma_0$ and $\operatorname{Homeo}_+(\mathbb S^1)$ are standard
Borel spaces, Lusin--Souslin applied to the injective restriction of
$\operatorname{weld}$ gives a Borel image and a measurable inverse:
\[
 H_0:=\operatorname{weld}(\Sigma_0),\qquad
 \operatorname{weld}^{-1}:H_0\longrightarrow\Sigma_0.
\]
By the definition of the welding law,
\[
 (\operatorname{weld})_\#\nu_\kappa^{\rm nf}=\widetilde\nu_\kappa,
 \qquad \widetilde\nu_\kappa(H_0)=1.
\]
Consequently, for measurable $F$ on $H_0$ and $1\leq p\leq\infty$,
\[
 \|F\circ\operatorname{weld}\|_{L^p(\nu_\kappa^{\rm nf})}
 =\|F\|_{L^p(\widetilde\nu_\kappa)}.
\]
We therefore identify these completed $L^p$ spaces and write
$\widetilde\nu_\kappa$ also for the normalized-frame law on $\Sigma$.
The defining formulas for $t_n$, $\sigma_h(\mu)$, $\tau_h(\mu)$, and
$\rho_h(\mu)$, evaluated on the normalized interior frame, give measurable
extensions from $\Sigma_{\rm r}$ to $\Sigma$.

\subsubsection{Disintegration of the frame measure}

Fix a Haar measure $H_G$ on $G$ and define
\[
 \Theta:G\times\Sigma\longrightarrow\mathcal F_{\rm fr},
 \qquad \Theta(M,h)=M\cdot h.
\]
Define the measure $\widehat\nu_\kappa$ on $\Sigma$ by
\begin{equation}\label{eq:DR-raw-frame-measure}
 d\widehat\nu_\kappa(h):=e^{2k(h)}d\widetilde\nu_\kappa(h).
\end{equation}
It is $\sigma$-finite, since $k$ is finite and measurable and
\[
 \widehat\nu_\kappa(\{h:|k(h)|\leq n\})\leq e^{2n},\qquad n\geq1.
\]
Fix the multiplicative constant of $\nu_\kappa^{\rm sph,or}$ by the
following product identity.

\begin{lemma}\label{lem:DR-doubly-marked-slice}
The map $\Theta$ is a measurable bijection with measurable inverse, and
\begin{equation}\label{eq:DR-full-frame-Weil-disintegration}
 M_*M_\kappa^{\rm fr}=M_\kappa^{\rm fr}\quad(M\in G),
 \qquad \Theta_*(H_G\otimes\widehat\nu_\kappa)=M_\kappa^{\rm fr}.
\end{equation}
\end{lemma}

\begin{proof}
M\"obius invariance of $\nu_\kappa^{\rm sph,or}$ and
\eqref{eq:DR-conformal-area-covariance} give the first identity.
For $x=(\gamma,U_+,U_-,\widehat f,\widehat g)$, use the marked points
$z_+,z_-$ above.
The normalizing map is
\begin{equation}\label{eq:DR-explicit-slice-Mobius}
 N_x(z):=
 \begin{cases}
 \displaystyle
 \frac{(z_+-z_-)(z-z_+)}{\widehat f'(0)(z-z_-)},
   &z_+,z_-\ne\infty,\\[7pt]
 \displaystyle\frac{z-z_+}{\widehat f'(0)},
   &z_+\ne\infty,\ z_-=\infty,\\[7pt]
 \displaystyle\frac1{(\iota\circ\widehat f)'(0)(z-z_-)},
   &z_+=\infty.
 \end{cases}
\end{equation}
It satisfies
\begin{equation}\label{eq:DR-unique-slice-Mobius}
 N_x(\widehat f(0))=0,
 \qquad N_x(\widehat g(\infty))=\infty,
 \qquad (N_x\circ\widehat f)'(0)=1.
\end{equation}
These conditions characterize $N_x$ uniquely.  By
Lemma~\ref{lem:DR-standard-Borel-frame-model}, it depends measurably on $x$,
and the inverse of $\Theta$ is
\begin{equation}\label{eq:DR-Weil-product-chart}
 s(x):=N_x\cdot x,\qquad M(x):=N_x^{-1},
 \qquad \Theta^{-1}(x)=(M(x),s(x)).
\end{equation}
Put
\[
 M_\kappa^{\rm prod}:=(\Theta^{-1})_*M_\kappa^{\rm fr}.
\]
Restrict the loop measure to the prescribed marks $0$ and $\infty$:
\begin{equation}\label{eq:DR-zero-infinity-restriction}
 d\nu_{\kappa,0,\infty}^{\rm sph}(\gamma)
 :=\mathbf 1_{\{0\in U_+,\,\infty\in U_-\}}
   d\nu_\kappa^{\rm sph,or}(\gamma,U_+,U_-).
\end{equation}
Let $\rho:=\log\operatorname{crad}_{U_+}(0)$ and let $\nu_\kappa^\#$
be the SLE shape law with unit interior conformal radius.  With a
normalization constant $c_\kappa^{\rm sc}>0$, the scale disintegration
of \cite[Remark~2.14]{GQW25} reads
\begin{equation}\label{eq:DR-loop-scale-disintegration}
 d\nu_{\kappa,0,\infty}^{\rm sph}
 =c_\kappa^{\rm sc}\,d\rho\,d\nu_\kappa^\#.
\end{equation}
Let $\alpha$ be the independent uniform relative exterior rotation.
In the normalization of \cite[Definition~1.1]{BJ25},
\begin{equation}\label{eq:DR-welding-law-shape-rotation}
 d\widetilde\nu_\kappa
 =d\nu_\kappa^\#\frac{d\alpha}{2\pi}.
\end{equation}
For a unit-radius loop $\eta$, let $U_+(\eta)$ and $U_-(\eta)$ contain
$0$ and $\infty$, respectively.  Choose Riemann maps $f,g_0$ with
$f(0)=0$, $f'(0)=1$, $g_0(\infty)=\infty$, and $g_0'(\infty)>0$, and set
\begin{equation}\label{eq:DR-relative-rotation-section}
 h_{\eta,\alpha}
 :=\bigl(\eta,U_+(\eta),U_-(\eta),f,g_0\circ r_\alpha\bigr).
\end{equation}
Its welding is $r_{-\alpha}\circ g_0^{-1}\circ f$, realizing
\eqref{eq:DR-welding-law-shape-rotation}.
Writing $g=g_0\circ r_\alpha$ and $k=k(h_{\eta,\alpha})$, scale the pair by
\begin{equation}\label{eq:DR-scaled-uniformizers}
 f_\rho=e^\rho f,
 \qquad
 g_\rho=e^\rho g,
 \qquad
 f'(0)=1,
 \qquad
 |g'(\infty)|=e^k.
\end{equation}
Then
\begin{align}
 \operatorname{crad}_{U_+}(0)
 &=f_\rho'(0)=e^\rho,
 \label{eq:DR-interior-crad-rho}\\
 \operatorname{crad}_{U_-}(\infty)
 &=|g_\rho'(\infty)|^{-1}=e^{-\rho-k}.
 \label{eq:DR-exterior-crad-rho}
\end{align}
The product of the two mark densities at $(0,\infty)$ is therefore
\begin{equation}\label{eq:DR-two-mark-density}
 e^{-2\rho}\,e^{2\rho+2k}=e^{2k}.
\end{equation}
For general marks, use the group chart
\[
 u=M(0)\in\mathbb C,\qquad v=\iota(M(\infty))\in\mathbb C,
 \qquad q=M'(0)\ne0,
 \qquad 1-uv\ne0.
\]
Set $M_{u,v}(z):=(u+z)/(1+vz)$.  The map with the prescribed derivative is
\[
 M(z)=M_{u,v}\left(\frac{qz}{1-uv}\right).
\]
With $\rho=\log|q|$ and $\theta_G=\arg q$, the group coordinates are
\begin{equation}\label{eq:DR-six-Haar-variables}
 M\longmapsto(u,v,\rho,\theta_G).
\end{equation}
Direct differentiation at the two marks yields
\[
 M'(0)=q,
 \qquad (\iota\circ M\circ\iota)'(0)=\frac{(1-uv)^2}{q}.
\]
For the frame $M\cdot h_{\eta,\alpha}$, conformal covariance therefore gives
\begin{align*}
 \operatorname{crad}_{MU_+(\eta)}(u)&=|q|,\\
 \operatorname{crad}_{MU_-(\eta)}(M(\infty))
 &=\frac{|1-uv|^2}{|q|}e^{-k(h_{\eta,\alpha})}.
\end{align*}
The second radius is computed in the coordinate $\iota$ at the exterior
mark.  Thus the product of the two frame densities relative to
$dA(u)dA(v)$ is
\begin{equation}\label{eq:DR-two-mark-general-density}
 \operatorname{crad}_{MU_+}(u)^{-2}
 \operatorname{crad}_{MU_-}(M(\infty))^{-2}
 =\frac{e^{2k(h_{\eta,\alpha})}}{|1-uv|^4}.
\end{equation}
For fixed $u,v$, pull the loop back by $M_{u,v}^{-1}$.  Its scale is
$\rho-\log|1-uv|$, so the scale differential remains $d\rho$.
Rotation invariance of the loop measure and uniqueness of
\eqref{eq:DR-loop-scale-disintegration} give
$(r_\theta)_*\nu_\kappa^\#=\nu_\kappa^\#$.  Thus the rotation
$\arg(q/(1-uv))$ is absorbed in the shape variable.

For fixed $u,v$, the tangent angles in these two target coordinates satisfy
\[
 \alpha_+=\theta_G,\qquad
 \alpha_-=2\arg(1-uv)-\theta_G-\alpha,\qquad
 \left|\det\frac{\partial(\alpha_+,\alpha_-)}
 {\partial(\theta_G,\alpha)}\right|=1.
\]
Let $C_\kappa>0$ collect the fixed measure normalizations.  Combining
\eqref{eq:DR-two-mark-general-density} with the scale and angle changes
gives, by Tonelli's theorem,
\begin{equation}\label{eq:DR-full-coordinate-density}
 dM_\kappa^{\rm prod}
 =C_\kappa\frac{dA(u)dA(v)d\rho d\theta_G}{|1-uv|^4}
 e^{2k(h_{\eta,\alpha})}
 d\nu_\kappa^\#(\eta)\frac{d\alpha}{2\pi}.
\end{equation}
In finite coordinates $z_+,z_-$, the invariant measure on ordered
distinct spherical points is
\[
 \frac{dA(z_+)dA(z_-)}{|z_+-z_-|^4}.
\]
For a M\"obius map $N$, the identity
\[
 |N(z_+)-N(z_-)|^2
 =|N'(z_+)N'(z_-)|\,|z_+-z_-|^2
\]
cancels the two area Jacobians.  Substitution of $z_+=u$, $z_-=1/v$ gives
\[
 \frac{dA(z_+)dA(z_-)}{|z_+-z_-|^4}
 =\frac{dA(u)|v|^{-4}dA(v)}{|u-v^{-1}|^4}
 =\frac{dA(u)dA(v)}{|1-uv|^4}.
\]
The stabilizer of the ordered pair is $\mathbb C^*$, with Haar measure
$d\rho d\theta_G$.  Hence, for a normalization constant $C_G>0$,
\begin{equation}\label{eq:DR-Haar-coordinate-density}
 dH_G=C_G\frac{dA(u)dA(v)d\rho d\theta_G}{|1-uv|^4}.
\end{equation}
Combining \eqref{eq:DR-full-coordinate-density},
\eqref{eq:DR-Haar-coordinate-density}, and
\eqref{eq:DR-welding-law-shape-rotation} gives
\begin{equation}\label{eq:DR-slice-density-computation}
 dM_\kappa^{\rm prod}
 =\frac{C_\kappa}{C_G}\,dH_G\,e^{2k(h_{\eta,\alpha})}
   d\nu_\kappa^\#(\eta)\frac{d\alpha}{2\pi}
 =\frac{C_\kappa}{C_G}\,dH_G\,d\widehat\nu_\kappa.
\end{equation}
Conformal covariance gives the same identity in every spherical mark
chart; a countable partition by such charts extends it to the full frame
space.  Choosing the normalization of $\nu_\kappa^{\rm sph,or}$ so that
$C_\kappa=C_G$ yields
\begin{equation}\label{eq:DR-Weil-product-measure}
 M_\kappa^{\rm prod}=H_G\otimes\widehat\nu_\kappa,
\end{equation}
which proves the second identity.
\end{proof}

Choose $\chi\in C_c(G)$ with
\begin{equation}\label{eq:DR-slice-cutoff-normalization}
 \chi\geq0,\qquad\int_G\chi\,dH_G=1.
\end{equation}
The product formula gives, for every nonnegative measurable $F$ on $\Sigma$,
\begin{equation}\label{eq:DR-slice-measure-cutoff-definition}
 \int_\Sigma F\,d\widehat\nu_\kappa
 =\int_{\mathcal F_{\rm fr}}\chi(M(x))F(s(x))\,dM_\kappa^{\rm fr}(x).
\end{equation}

\subsection{Changes of frame and their divergences}

\subsubsection{Changes of source frame}

Let $f:\mathbb D\to\mathbb C$ be an unnormalized univalent map,
with the contours chosen as in Lemma~\ref{lem:DR-physical-trace}.
For $j\geq0$, use the raw Schiffer response $L_{-j-2}^Cf$ from
\eqref{eq:DR-Schiffer-variation}, with $\psi=f^{-1}$.
Let $a\in\operatorname{Aut}(\mathbb D)$, put $f_a=f\circ a$, and choose
a matrix representative
\begin{equation}\label{eq:DR-source-Mobius-matrix}
 a^{-1}(u)=\frac{\alpha u+\beta}{cu+d},
 \qquad \alpha d-\beta c=1.
\end{equation}

For the contractions, let $\mu\in L^\infty(\mathbb D)$ have compact
support and satisfy the three-radius condition of
Lemma~\ref{lem:DR-physical-trace} for $(f_a,\mu)$.
Define
\begin{equation}\label{eq:DR-Beltrami-frame-pullback}
 \mu^a(w)
 :=\mu(a^{-1}(w))
   \frac{a'(a^{-1}(w))}{\overline{a'(a^{-1}(w))}}.
\end{equation}
The same three-radius condition holds for $(f,\mu^a)$, and
$\operatorname{supp}\mu^a=a(\operatorname{supp}\mu)\Subset\mathbb D$.
Define the lift equivariant under source precomposition by
\begin{equation}\label{eq:DR-equivariant-Schiffer-seed-extension}
 (L_{-j-2}^{\rm eq}f_a)(z):=(L_{-j-2}^Cf)(a(z)).
\end{equation}

The difference between the two lifts is described by
\begin{align}
 Q_{a,j}(w)
 &:=\frac{c}{2\pi i}\oint_{|\zeta|=r_2}\psi'(\zeta)^2
 \left\{
 \frac{(cw+d)^2}{(c\psi(\zeta)+d)^3}
 +\frac{cw+d}{(c\psi(\zeta)+d)^2}
 +\frac1{c\psi(\zeta)+d}
 \right\}\zeta^{-j-1}\,d\zeta,
 \label{eq:DR-source-frame-defect-polynomial}\\
 V_{f,j}(a)(z)&:=\frac{Q_{a,j}(a(z))}{a'(z)},
 \label{eq:DR-source-frame-vertical-field}\\
 V_{j,f}^{\rm src}(f_a)&:=L_{-j-2}^Cf_a-(L_{-j-2}^Cf)\circ a.
 \label{eq:DR-source-frame-anomaly-field}
\end{align}
For uniform estimates, let $a$ range over a fixed compact subset of
$\operatorname{Aut}(\mathbb D)$ on which the same three radii apply.
Set $q=r_1/r_2<1$; the constant $C$ below depends only on this
localization and $\mu$.

\begin{lemma}\label{lem:DR-frame-covariant-contracted-trace}
The source-frame defect is vertical, with
\begin{equation}\label{eq:DR-source-frame-anomaly-zero-divergence}
 V_{j,f}^{\rm src}(f_a)=f_a'V_{f,j}(a),
 \qquad \operatorname{Div}_{H_{\mathbb D}}V_{j,f}^{\rm src}=0.
\end{equation}
The exterior defect obtained by inversion has the same properties.

The contractions satisfy
\begin{align}
 \sum_{j\geq0}\widehat m_j(f_a;\mu)\pi_{j+2}(f_a)
 &=-(\mathcal S f_a,\mu),
 \label{eq:DR-frame-covariant-matter-contraction}\\
 \sum_{j\geq0}L_{-j-2}^{\rm eq}
     \widehat m_j(f_a;\mu)
 &=\frac{26}{12}(\mathcal S f_a,\mu).
 \label{eq:DR-frame-covariant-ghost-contraction}
\end{align}
For $N\geq0$,
\begin{equation}\label{eq:DR-frame-covariant-truncation-rate}
 \begin{aligned}
 \left|\sum_{j=0}^N \widehat m_j(f_a;\mu)\pi_{j+2}(f_a)
       +(\mathcal S f_a,\mu)\right|
 &\leq C(N+2)^2q^{N+1},\\
 \left|\sum_{j=0}^N L_{-j-2}^{\rm eq}\widehat m_j(f_a;\mu)
       -\frac{26}{12}(\mathcal S f_a,\mu)\right|
 &\leq C(N+2)^2q^{N+1}.
 \end{aligned}
\end{equation}
\end{lemma}

\begin{proof}
Put $w=\psi(x)$, $v=\psi(\zeta)$, and $A=a^{-1}$.  Then
\begin{equation}\label{eq:DR-source-Mobius-difference-identity}
 \frac{A'(v)^2}{A'(w)\{A(w)-A(v)\}}
 =\left(\frac{cw+d}{cv+d}\right)^3\frac1{w-v}.
\end{equation}
Substitution into \eqref{eq:DR-Schiffer-kernel} gives
\begin{align}
 K_{a^{-1}\circ\psi}(x,\zeta)
 &=\left(\frac{cw+d}{cv+d}\right)^3
   \left\{K_\psi(x,\zeta)+\frac1{x-\zeta}\right\}
   -\frac1{x-\zeta},
 \label{eq:DR-Schiffer-kernel-frame-change}\\
 \shortintertext{so}
 K_{a^{-1}\circ\psi}(x,\zeta)-K_\psi(x,\zeta)
 &=\frac{\psi'(\zeta)^2}{\psi'(x)(w-v)}
   \left\{
   \left(\frac{cw+d}{cv+d}\right)^3-1
   \right\}.
 \label{eq:DR-Schiffer-kernel-frame-defect}
\end{align}
Since
\begin{equation}\label{eq:DR-source-cubic-factorization}
 \frac1{w-v}
 \left\{\left(\frac{cw+d}{cv+d}\right)^3-1\right\}
 =c\left\{
 \frac{(cw+d)^2}{(cv+d)^3}
 +\frac{cw+d}{(cv+d)^2}
 +\frac1{cv+d}
 \right\},
\end{equation}
putting $x=f(w)$ and integrating against
$\zeta^{-j-1}d\zeta/(2\pi i)$ gives
\begin{align}
 (L_{-j-2}^Cf_a)(z)-(L_{-j-2}^Cf)(a(z))
 &=f'(a(z))Q_{a,j}(a(z))
  =f_a'(z)V_{f,j}(a)(z).
 \label{eq:DR-Schiffer-response-frame-defect}
\end{align}
Since $\deg Q_{a,j}\leq2$ and M\"obius pullback preserves quadratic
vector fields, $V_{f,j}(a)\partial_z$ lies in the complexified Lie algebra
of $\operatorname{Aut}(\mathbb D)$.  This proves verticality.
For $b\in\operatorname{Aut}(\mathbb D)$ and a source field $X$, write
\begin{equation}\label{eq:DR-source-frame-pullback-vector}
 b^*X:=\frac{X\circ b}{b'}.
\end{equation}
Adding and subtracting $(L_{-j-2}^C(f\circ a))\circ b$ in the defect for
$a\circ b$ gives
\begin{equation}\label{eq:DR-source-frame-anomaly-cocycle}
 V_{f,j}(a\circ b)
 =V_{f\circ a,j}(b)+b^*V_{f,j}(a).
\end{equation}
Put
\begin{equation}\label{eq:DR-source-frame-Jj}
 \ell_j(f):=\frac1{2\pi i}\oint_{|\zeta|=r_2}
 \psi'(\zeta)^2\zeta^{-j-1}\,d\zeta.
\end{equation}
For $X(z)=x_0+x_1z+x_2z^2$ in the real source Lie algebra, let
$a_t=\exp(tX)$.  Thus
\begin{equation}\label{eq:DR-source-frame-infinitesimal-a}
 a_t(z)=z+t(x_0+x_1z+x_2z^2)+O(t^2).
\end{equation}
The denominator coefficient of $a_t^{-1}$ is $c=tx_2+O(t^2)$, so
\eqref{eq:DR-source-frame-defect-polynomial} gives
\begin{equation}\label{eq:DR-source-frame-infinitesimal-defect}
 T_{f,j}X:=\left.\partial_t\right|_{t=0}V_{f,j}(a_t)
 =3x_2\ell_j(f).
\end{equation}
In the real basis
\[
 E_1=1-z^2,\qquad E_2=i(1+z^2),\qquad E_3=iz,
\]
the complexified map has matrix and trace
\begin{equation}\label{eq:DR-source-frame-defect-trace-zero}
 [T_{f,j}]=\frac{3\ell_j(f)}2
 \begin{pmatrix}-1&i&0\\i&1&0\\0&0&0\end{pmatrix},
 \qquad \operatorname{tr}T_{f,j}=0.
\end{equation}
Use the left-invariant frame
\begin{equation}\label{eq:DR-source-left-invariant-frame}
 X_a^L:=\left.\frac d{dt}\right|_{t=0}
 a\circ\exp(tX).
\end{equation}
Since
\begin{equation}\label{eq:DR-source-pullback-bracket}
 \left.\partial_t\right|_{t=0}(\exp(tX))^*V=[X,V],
\end{equation}
differentiating \eqref{eq:DR-source-frame-anomaly-cocycle} gives
\begin{equation}\label{eq:DR-source-frame-anomaly-derivative}
 X_a^L[V_{f,j}]
 =T_{f\circ a,j}(X)+[X,V_{f,j}(a)].
\end{equation}
Write $V_{f,j}(a)=\sum_{r=1}^3v_r(a)E_r$ in the displayed real
basis, with complex coefficients.  Unimodularity and
\eqref{eq:DR-source-frame-anomaly-derivative} give
\[
 \operatorname{Div}_{H_{\mathbb D}}V_{j,f}^{\rm src}
 =-\sum_{r=1}^3E_r^Lv_r
 =-\operatorname{tr}T_{f\circ a,j}
   +\operatorname{tr}\operatorname{ad}_{V_{f,j}(a)}=0.
\]
Thus, for $F_+\in C_c^1(\operatorname{Fr}^+(\gamma,U_+))$,
\begin{equation}\label{eq:DR-source-frame-anomaly-Haar-zero}
 \int_{\operatorname{Aut}(\mathbb D)}
 V_{j,f}^{\rm src}F_+\,dH_{\mathbb D}
 =\int F_+\operatorname{Div}_{H_{\mathbb D}}V_{j,f}^{\rm src}
       \,dH_{\mathbb D}=0.
\end{equation}
Inversion gives the exterior assertion.

The two contraction identities in Lemma~\ref{lem:DR-physical-trace}
also hold when $\operatorname{supp}\mu^a$ meets $0$.  Indeed,
\[
 \|\mu^a-\mathbf1_{\{|w|>\varepsilon\}}\mu^a\|_{L^1(\mathbb D)}
 \leq\pi\varepsilon^2\|\mu\|_\infty\longrightarrow0,
\]
and the integral kernels in
\eqref{eq:DR-matter-contraction-proof} and
\eqref{eq:DR-variable-coefficient-trace} are bounded on the fixed
compact support.
The change of variables $w=a(z)$ gives
\begin{align}
 \widehat m_j(f_a;\mu)
 &=\frac1\pi\int_{\mathbb D}\mu^a(w)f'(w)^2f(w)^j\,dA(w)
 =\widehat m_j(f;\mu^a),
 \label{eq:DR-moment-frame-covariance}\\
 (\mathcal S f_a,\mu)&=(\mathcal S f,\mu^a),
 \label{eq:DR-Schwarzian-frame-covariance}
\end{align}
because $\mathcal S a=0$.  Also $\pi_{j+2}(f_a)=\pi_{j+2}(f)$, so
Lemma~\ref{lem:DR-physical-trace} gives
\[
 \sum_{j\geq0}\widehat m_j(f_a;\mu)\pi_{j+2}(f_a)
 =\sum_{j\geq0}\widehat m_j(f;\mu^a)\pi_{j+2}(f)
 =-(\mathcal S f,\mu^a)=-(\mathcal S f_a,\mu).
\]
Differentiating \eqref{eq:DR-moment-frame-covariance} while holding $a$
fixed gives the trace identity from the same lemma:

\begin{align}
 \sum_{j\geq0}L_{-j-2}^{\rm eq}\widehat m_j(f_a;\mu)
 &=\sum_{j\geq0}L_{-j-2}^{C}\widehat m_j(f;\mu^a)\notag\\
 &=\frac{26}{12}(\mathcal S f,\mu^a)
 =\frac{26}{12}(\mathcal S f_a,\mu),
 \label{eq:DR-frame-covariant-ghost-proof}
\end{align}
On the fixed collar, Cauchy's estimates for $\mathcal S\psi$ and
$K_\psi$, together with \eqref{eq:DR-mj-raw-extension}, give
\begin{align*}
 |\widehat m_j(f_a;\mu)|&\leq Cr_1^j,\\
 |\pi_{j+2}(f_a)|+
 \sup_{z\in a(\operatorname{supp}\mu)}
 \bigl(|(L_{-j-2}^Cf)(z)|+|(L_{-j-2}^Cf)'(z)|\bigr)&\leq Cr_2^{-j}.
\end{align*}
Differentiating the moment integral therefore yields
\[
 |L_{-j-2}^{\rm eq}\widehat m_j(f_a;\mu)|
 \leq C(j+1)q^j.
\]
For $\ell\in\{0,1,2\}$, the geometric tails satisfy
\begin{align}
 \sum_{j=N+1}^{\infty}(j+1)^{\ell}q^j
 &=q^{N+1}\sum_{m=0}^{\infty}(N+m+2)^{\ell}q^m\notag\\
 &\leq(N+2)^{\ell}q^{N+1}\sum_{m=0}^{\infty}(m+1)^{\ell}q^m\notag\\
 &\leq\frac{2(N+2)^{\ell}q^{N+1}}{(1-q)^3}.
 \label{eq:DR-geometric-tail-estimate}
\end{align}
Summing the two coefficient bounds with
\eqref{eq:DR-geometric-tail-estimate} proves
\eqref{eq:DR-frame-covariant-truncation-rate}.
\end{proof}

\subsubsection{Changes of target frame}

Write
\begin{equation}\label{eq:DR-Mobius-matrix}
 M(\xi)=\frac{a\xi+b}{c\xi+d},
 \qquad ad-bc=1.
\end{equation}
Set $\operatorname{pole}(M)=-d/c$ when $c\ne0$ and
$\operatorname{pole}(M)=\infty$ when $c=0$.  Define
\begin{equation}\label{eq:DR-Mobius-pole-free-chart}
 G_f
 :=\{M\in G:\operatorname{pole}(M)
       \notin\overline{f(\mathbb D)}\}.
\end{equation}
Use the physical Cauchy vector \eqref{eq:DR-physical-Laurent} also for
unnormalized univalent maps.  Fix $M\in G_f$ and put
\begin{equation}\label{eq:DR-Mobius-anomaly-integrals}
 I_r(M,f):=\frac1\pi\int_{\mathbb D}
 \frac{\mu(z)f'(z)^2}{\{cf(z)+d\}^{r}}\,dA(z),
 \qquad r=1,2,3,
\end{equation}
and
\begin{equation}\label{eq:DR-Mobius-anomaly-polynomial}
 p_\mu(M,f;y)
 :=-c\left\{I_3(M,f)+(a-cy)I_2(M,f)
 +(a-cy)^2I_1(M,f)\right\}.
\end{equation}
The right-hand side is unchanged when the matrix representative in
\eqref{eq:DR-Mobius-matrix} is multiplied by $-1$ and is therefore
well-defined on $G$.

For $h\in\Sigma$ and $x=M\cdot h$, write $\widehat f=M\circ f$ and
$\widehat g=M\circ g$.  The condition $M\in G_f$ makes $\widehat f$
holomorphic on $\mathbb D$.  Let
$\widehat R_\mu^{\mathrm C}$ be the raw Schiffer variation induced by the
physical vector $V_{\mu,\widehat f}$.  It acts on the unnormalized
maps in $\mathcal F_{\rm fr}$ and restricts on the slice to
$R_\mu^{\rm raw}$ from \eqref{eq:DR-Rmu-variable-coefficients}.
To specify its interior component, choose a positively oriented
analytic contour $\Gamma_0\Subset\widehat f(\mathbb D)$ enclosing
$\widehat f(\operatorname{supp}\mu)$, and put
$\widehat\psi=\widehat f^{-1}$.  The Cauchy projection used in
\eqref{eq:DR-Schiffer-variation} gives
\begin{equation}\label{eq:DR-Cauchy-field-interior-definition}
 (\widehat R_\mu^{\mathrm C}\widehat f)(z)
 =-\frac1{2\pi i}\oint_{\Gamma_0}
 K_{\widehat\psi}(\widehat f(z),\zeta)
 V_{\mu,\widehat f}(\zeta)\,d\zeta.
\end{equation}
The formula extends analytically from the inside of $\Gamma_0$ to
$\mathbb D$.  The exterior component is the Schiffer projection in the
coordinate $\iota$, with no affine normalization subtracted.

Define its $G$-equivariant extension from the slice by
\[
 \widehat R_\mu^{\,\natural}(M\cdot h)
 :=d(M\cdot)_h\bigl[\widehat R_\mu^{\mathrm C}(h)\bigr],
 \qquad (M,h)\in G\times\Sigma,
\]
where $d(M\cdot)_h$ differentiates simultaneous postcomposition of the two
frames.  Define the difference $V_\mu^R$ by
\begin{equation}\label{eq:DR-Cauchy-natural-anomaly-decomposition}
 \widehat R_\mu^{\mathrm C}
 =\widehat R_\mu^{\,\natural}+V_\mu^R.
\end{equation}

In the following identities, $h\in\Sigma$ is fixed,
$w\notin f(\operatorname{supp}\mu)$, $cw+d\ne0$, and
$\chi\in C_c^1(G_f)$.

\begin{lemma}\label{lem:DR-Mobius-normalization-anomaly}
\begin{equation}\label{eq:DR-exact-Mobius-anomaly}
 V_{\mu,M\circ f}(M(w))
 =M'(w)V_{\mu,f}(w)+p_\mu(M,f;M(w)).
\end{equation}
The polynomial has degree at most two, and the vertical field satisfies
\[
 V_\mu^R\widehat f=p_\mu(M,f;\widehat f),
 \qquad V_\mu^R\widehat g=p_\mu(M,f;\widehat g).
\]
On $G_f\times\{h\}$,
\begin{align}
 \operatorname{Div}_{H_G}V_\mu^R(\,\cdot\,,h)&=0,
 \label{eq:DR-Mobius-anomaly-zero-divergence}\\
 \int_{G_f}V_\mu^R\chi\,dH_G&=0.
 \label{eq:DR-Mobius-anomaly-Haar-zero}
\end{align}
\end{lemma}

\begin{proof}
Write
\begin{equation}\label{eq:DR-Cauchy-vector-short}
 j_{\mu,f}(z):=\frac1\pi\mu(z)f'(z)^2,
 \qquad
 V_{\mu,f}(w)
 =-\int_{\mathbb D}\frac{j_{\mu,f}(z)}{w-f(z)}\,dA(z).
\end{equation}
Since
\begin{align}
 M'(x)&=(cx+d)^{-2},
 \label{eq:DR-Mobius-derivative}\\
 M(w)-M(x)&=\frac{w-x}{(cw+d)(cx+d)},
 \label{eq:DR-Mobius-difference}
\end{align}
direct substitution gives
\begin{align}
 V_{\mu,M\circ f}(M(w))
 &=-\int_{\mathbb D}j_{\mu,f}(z)
 \frac{cw+d}{\{cf(z)+d\}^3\{w-f(z)\}}\,dA(z),
 \label{eq:DR-Mobius-transformed-Cauchy}\\
 M'(w)V_{\mu,f}(w)
 &=-\int_{\mathbb D}j_{\mu,f}(z)
 \frac{1}{(cw+d)^2\{w-f(z)\}}\,dA(z).
 \label{eq:DR-Mobius-pushed-Cauchy}
\end{align}
Put $s_w=cw+d$ and $s_z=cf(z)+d$.  The algebraic identity
\begin{samepage}
\begin{equation}\label{eq:DR-Mobius-anomaly-algebra}
 \frac1{w-f(z)}\left\{\frac{s_w}{s_z^3}-\frac1{s_w^2}\right\}
 =c\left\{\frac1{s_z^3}+\frac1{s_ws_z^2}+\frac1{s_w^2s_z}\right\}
\end{equation}
and
\begin{equation}\label{eq:DR-Mobius-a-cy}
 a-cM(w)=\frac1{cw+d}=\frac1{s_w}
\end{equation}
\end{samepage}
give \eqref{eq:DR-exact-Mobius-anomaly}, with $p_\mu$ as in
\eqref{eq:DR-Mobius-anomaly-polynomial}.
To compare the lifts, substitute the Cauchy vector into
\eqref{eq:DR-Cauchy-field-interior-definition} and use Cauchy's formula:
\begin{align*}
 (\widehat R_\mu^{\mathrm C}\widehat f)(z)
 &=\frac1\pi\int_{\mathbb D}\mu(u)\widehat f'(u)^2
   K_{\widehat\psi}(\widehat f(z),\widehat f(u))\,dA(u)\\
 &=\widehat f'(z)C_\mu(z)
   +V_{\mu,\widehat f}(\widehat f(z)).
\end{align*}
Equivariance gives
\[
 (\widehat R_\mu^{\,\natural}\widehat f)(z)
 =M'(f(z))\{f'(z)C_\mu(z)+V_{\mu,f}(f(z))\}.
\]
Since $\widehat f'=(M'\circ f)f'$, subtraction and
\eqref{eq:DR-exact-Mobius-anomaly} yield
\begin{align*}
 V_\mu^R\widehat f
 &=V_{\mu,M\circ f}\circ M\circ f-(M'\circ f)V_{\mu,f}\circ f
   =p_\mu(M,f;\widehat f),\\
 V_\mu^R\widehat g
 &=V_{\mu,M\circ f}\circ M\circ g-(M'\circ g)V_{\mu,f}\circ g
   =p_\mu(M,f;\widehat g).
\end{align*}
The exterior identity uses the Schiffer projection after inversion.
Thus the difference is simultaneous M\"obius postcomposition.

For $X(y)=(x_0+x_1y+x_2y^2)\partial_y\in\mathfrak{sl}_2(\mathbb C)$,
the matrix of $M_t=\exp(tX)$ is
\begin{equation}\label{eq:DR-infinitesimal-Mobius-matrix}
 M_t=
 \begin{pmatrix}
  1+tx_1/2&tx_0\\
  -tx_2&1-tx_1/2
 \end{pmatrix}+O(t^2).
\end{equation}
Then
\begin{equation}\label{eq:DR-infinitesimal-Mobius-action}
 M_t(y)=y+t(x_0+x_1y+x_2y^2)+O(t^2).
\end{equation}
At $t=0$, all three integrals in
\eqref{eq:DR-Mobius-anomaly-integrals} equal
$m_0(f)$.  Hence
\begin{equation}\label{eq:DR-infinitesimal-anomaly-map}
 p_\mu(M_t,f;y)=3tx_2m_0(f)+O(t^2),
 \qquad T_{\mu,\widehat f}X:=3x_2\widehat m_0(\widehat f).
\end{equation}
In the ordered complex basis $(1,y,y^2)$,
\[
 [T_{\mu,f}]=
 \begin{pmatrix}0&0&3m_0(f)\\0&0&0\\0&0&0\end{pmatrix}.
\]
Therefore
\begin{equation}\label{eq:DR-anomaly-trace-zero}
 \operatorname{tr}_{\mathbb C}T_{\mu,f}=0,
 \qquad
 \operatorname{tr}_{\mathbb R}T_{\mu,f}
 =2\operatorname{Re}\operatorname{tr}_{\mathbb C}T_{\mu,f}=0.
\end{equation}

Whenever $M\in G_f$, $N\in G_{M\circ f}$, and
$N\circ M\in G_f$, the exact defect satisfies the cocycle identity
\begin{equation}\label{eq:DR-Mobius-anomaly-cocycle}
 p_\mu(N\circ M,f)=p_\mu(N,M\circ f)+N_*p_\mu(M,f).
\end{equation}
This follows by inserting $N_*V_{\mu,M\circ f}$ between the two
terms in the defect for $N\circ M$.  Use the holomorphic right-invariant
frame, with $t\in\mathbb C$:
\begin{equation}\label{eq:DR-right-invariant-frame-definition}
 X_M^R:=\left.\partial_t\right|_{t=0}\exp(tX)\circ M.
\end{equation}
The coefficient of $V_\mu^R$ in this frame is $p_\mu(M,f)$.
Differentiating \eqref{eq:DR-Mobius-anomaly-cocycle} gives
\begin{equation}\label{eq:DR-Mobius-anomaly-left-derivative}
 X_M^R[p_\mu(\,\cdot\,,f)]
 =T_{\mu,M\circ f}(X)+\operatorname{ad}_{p_\mu(M,f)}X.
\end{equation}
Here the second term follows from
$\left.\partial_t\right|_0(N_t)_*P=[P,X]$ for the flow of $X$.
Unimodularity gives zero Haar divergence for the real invariant fields
and their complex combinations.  Since $p_\mu(M,f)$ is holomorphic in
$M\in G_f$, the multiplier rule and
\eqref{eq:DR-Mobius-anomaly-left-derivative} give
\begin{equation}\label{eq:DR-Mobius-anomaly-divergence-proof}
 \operatorname{Div}_{H_G}V_\mu^R
 =-\operatorname{tr}_{\mathbb C}T_{\mu,M\circ f}
  -\operatorname{tr}_{\mathbb C}\operatorname{ad}_{p_\mu(M,f)}
 =0.
\end{equation}
This proves \eqref{eq:DR-Mobius-anomaly-zero-divergence}; integration by
parts against compactly supported $\chi$ proves
\eqref{eq:DR-Mobius-anomaly-Haar-zero}.
\end{proof}

\subsection{The localized physical Ward identity}

\subsubsection{The finite restriction formula}

Fix $p\geq2$ and use the physical flows $\phi_{p,t}^{(a)}$ from
\eqref{eq:DR-two-real-physical-fields}.
For $0<r<1$, put
\[
 \operatorname{Loop}^r
 :=\{(\gamma,U_+,U_-)\in\mathcal J^{\rm or}:
       0\in U_+,\ \infty\in U_-,\
       \gamma\subset\{\xi:r<|\xi|<r^{-1}\}\}.
\]
The SLE loop measure satisfies
$\nu_\kappa^{\rm sph,or}(\operatorname{Loop}^r)<\infty$
\cite[Section~2.2]{GQW25}.

Fix $0<\delta<R<\infty$ and a family of ordered loops satisfying
\begin{equation}\label{eq:DR-physical-collar-localization}
 B(0,\delta)\subset U_+,\qquad
 \infty\in U_-,\qquad
 \gamma\subset B(0,R).
\end{equation}
The flows are defined on a common annular neighborhood of these loops.
Choose $\epsilon>0$ so that their images for $|t|<\epsilon$ satisfy the
same containments with $\delta/2$ and $2R$.  We call this localization a
physical collar.  For $\gamma_t=\phi_{p,t}^{(a)}\gamma$, transport the
component order and write $U_{\pm,t}^{(a)}$ for the resulting components.

Let $G_0$ be a bounded measurable observable whose support and short
preimages lie in this collar.
All supremum norms below are pointwise norms on the target collar.  Set
\[
 s_p^{(1)}:=-\frac{c_{\mathrm m}}6\operatorname{Re}\pi_p,
 \qquad s_p^{(2)}:=\frac{c_{\mathrm m}}6\operatorname{Im}\pi_p.
\]

\begin{proposition}
\label{prop:DR-external-localized-Ward}
There are positive measurable densities $J_{p,t}^{(a)}$ such that
\begin{equation}\label{eq:DR-physical-finite-restriction-RN}
 \begin{aligned}
 &\int G_0(\phi_{p,t}^{(a)}\gamma,
           U_{+,t}^{(a)},U_{-,t}^{(a)})\,
      d\nu_\kappa^{\rm sph,or}(\gamma,U_+,U_-)\\
 &\qquad=\int G_0(\eta,V_+,V_-)J_{p,t}^{(a)}(\eta)\,
      d\nu_\kappa^{\rm sph,or}(\eta,V_+,V_-).
 \end{aligned}
\end{equation}
Their logarithmic derivatives at zero are
\begin{align}
 \left.\partial_t\right|_{t=0}\log J_{p,t}^{(1)}
 &=s_p^{(1)},
 \label{eq:DR-first-real-physical-score}\\
 \left.\partial_t\right|_{t=0}\log J_{p,t}^{(2)}
 &=s_p^{(2)},
 \label{eq:DR-second-real-physical-score}\\
 \shortintertext{hence}
 \frac12\{s_p^{(1)}-is_p^{(2)}\}
 &=-\frac{c_{\mathrm m}}{12}\pi_p.
 \label{eq:DR-complex-physical-score}
\end{align}
After shrinking the time interval, if necessary,
\begin{equation}\label{eq:DR-local-likelihood-dominator}
 \sup_{a\in\{1,2\}}\sup_{0<|t|<\epsilon}
 \left\|\frac{J_{p,t}^{(a)}-1}{t}\right\|_{\sup}<\infty,
 \qquad
 \frac{J_{p,t}^{(a)}-1}{t}\longrightarrow s_p^{(a)}
 \quad\text{pointwise}.
\end{equation}
\end{proposition}

\begin{proof}
Use the simply connected spherical domains
\[
 D=\{\xi:|\xi|>\delta/4\}\cup\{\infty\},
 \qquad E=\{\xi:|\xi|>\delta/3\}\cup\{\infty\}.
\]
Both physical fields are holomorphic on $D$, and $\overline E\Subset D$.
Shrink $\epsilon$ so that $\phi_{p,q}^{(a)}$ is conformal near
$\overline E$ for $|q|\leq2\epsilon$ and all tested loops and their
short preimages remain in $E$.
Let $\Lambda^*(K_1,K_2)$ denote the renormalized Brownian loop mass
of loops meeting both disjoint compact sets, as in
\cite[Section~2.1]{GQW25}.  The conformal bijection
$\phi_{p,q}^{(a)}:E\to\phi_{p,q}^{(a)}(E)$ and the preceding
containments satisfy the hypotheses of \cite[Lemma~2.2]{GQW25}, giving
\eqref{eq:DR-physical-finite-restriction-RN} with
\begin{align}
 E_q^{(a)}&:=\phi_{p,q}^{(a)}(E),
 \label{eq:DR-transported-restriction-domain}\\
 J_{p,q;E}^{(a)}(\eta)
 &:=\exp\!\left\{\frac{c_{\mathrm m}}2\left[
   \Lambda^*(\eta,\widehat{\mathbb C}\setminus E_q^{(a)})
   -\Lambda^*(\phi_{p,-q}^{(a)}\eta,
              \widehat{\mathbb C}\setminus E)
 \right]\right\}.
 \label{eq:DR-domain-indexed-restriction-density}
\end{align}
For $t,r,t+r$ in the time interval, addition of the two exponents gives
\begin{equation}\label{eq:DR-domain-indexed-density-cocycle}
 J_{p,t+r;E}^{(a)}(\eta)
 =J_{p,r;E}^{(a)}\!\left(\phi_{p,-t}^{(a)}\eta\right)
  J_{p,t;E_r^{(a)}}^{(a)}(\eta).
\end{equation}
To compare the densities on different domains, choose a simply connected
spherical domain $F$ containing all tested loops and their shifted images,
and reduce the time interval so that
\[
 \overline F\Subset E_q^{(a)},
 \qquad |q|\leq2\epsilon,\quad a\in\{1,2\}.
\]
Put $F_t^{(a)}:=\phi_{p,t}^{(a)}(F)$ and write
$\operatorname{Wer}(A,B)$ for the Werner-loop mass of loops meeting both
$A$ and $B$.  Conformal invariance of that measure inside $E$, and
identically inside each $E_r^{(a)}$, gives
\begin{align}
 &\operatorname{Wer}(\eta,\widehat{\mathbb C}\setminus F_t^{(a)})
  -\operatorname{Wer}(\eta,\widehat{\mathbb C}\setminus E_t^{(a)})
  \notag\\
 &\qquad=
  \operatorname{Wer}(\phi_{p,-t}^{(a)}\eta,
                     \widehat{\mathbb C}\setminus F)
  -\operatorname{Wer}(\phi_{p,-t}^{(a)}\eta,
                     \widehat{\mathbb C}\setminus E).
 \label{eq:DR-Werner-localization-independence}
\end{align}
The differences count loops contained in the larger domain and meeting
the complement of the smaller one; $\phi_{p,t}^{(a)}$ maps one such family
onto the other.  Apply \cite[Theorem~2.3]{GQW25} to the conformal pairs
$(E,E_t^{(a)})$ and $(F,F_t^{(a)})$, whose domains are simply connected
and contain the tested curves.  It replaces each conformal restriction
difference for $\Lambda^*$ by that for $\operatorname{Wer}$, so
\[
 J_{p,t;E}^{(a)}=J_{p,t;F}^{(a)}=J_{p,t;E_r^{(a)}}^{(a)}.
\]
Denote this common positive measurable density by $J_{p,t}^{(a)}$.

Using $E_{-t}^{(a)}$ as the source domain keeps the target domain fixed:
\[
 \log J_{p,t}^{(a)}(\eta)
 =\frac{c_{\mathrm m}}2\left\{
   \Lambda^*(\eta,\widehat{\mathbb C}\setminus E)
   -\Lambda^*(\phi_{p,-t}^{(a)}\eta,
              \widehat{\mathbb C}\setminus E_{-t}^{(a)})
 \right\}.
\]
Here $\widehat{\mathbb C}\setminus E=\overline{B(0,\delta/3)}$ lies
strictly inside every tested loop.  Thus \cite[Corollary~3.9]{GQW25}
applies with index $-p$.  For an interior uniformizer $f$, the defining
counterclockwise contour integral is
\[
 \frac1{2\pi i}\oint
 (-w^{1-p})\mathcal S(f^{-1})(w)\,dw=-\pi_p.
\]
The two real logarithmic derivatives are consequently
\[
 -\frac{c_{\mathrm m}}6\operatorname{Re}\pi_p=s_p^{(1)},
 \qquad \frac{c_{\mathrm m}}6\operatorname{Im}\pi_p=s_p^{(2)}.
\]
The complex combination gives \eqref{eq:DR-complex-physical-score}, and
\eqref{eq:DR-domain-indexed-density-cocycle} becomes
\begin{equation}\label{eq:DR-shifted-target-density-cocycle}
 J_{p,t+r}^{(a)}(\eta)
 =J_{p,r}^{(a)}\!\left(\phi_{p,-t}^{(a)}\eta\right)
  J_{p,t}^{(a)}(\eta).
\end{equation}
Taking logarithms, subtracting $\log J_{p,t}^{(a)}(\eta)$, and dividing
by $r$ gives
\[
 \frac{\log J_{p,t+r}^{(a)}(\eta)-\log J_{p,t}^{(a)}(\eta)}r
 =\frac{\log J_{p,r}^{(a)}(\phi_{p,-t}^{(a)}\eta)}r.
\]
Letting $r\to0$ yields
\begin{equation}\label{eq:DR-shifted-target-score}
 \partial_t\log J_{p,t}^{(a)}(\eta)
 =s_p^{(a)}\!\left(\phi_{p,-t}^{(a)}\eta\right).
\end{equation}
Since $J_{p,0}^{(a)}=1$, integration gives
\begin{equation}\label{eq:DR-shifted-target-score-integral}
 J_{p,t}^{(a)}(\eta)
 =\exp\!\left\{\int_0^t
 s_p^{(a)}\!\left(\phi_{p,-q}^{(a)}\eta\right)\,dq\right\}.
\end{equation}
Choose
\[
 0<r_0<\min\{\delta/2,(2R)^{-1}\}.
\]
The collar assumptions place all shifted loops in
\(\operatorname{Loop}^{r_0}\).  The Neretin-polynomial formula and
\cite[Lemma~3.7]{GQW25} therefore give one \(C<\infty\) such that
\begin{equation}\label{eq:DR-shifted-target-score-bound}
 \left|
 s_p^{(a)}\!\left(\phi_{p,-r}^{(a)}\eta\right)
 \right|\leq C,
 \qquad |r|\leq\epsilon,\qquad a\in\{1,2\}.
\end{equation}
Consequently,
\begin{align}
 e^{-C|t|}\leq J_{p,t}^{(a)}(\eta)&\leq e^{C|t|},
 \qquad
 \left|\frac{J_{p,t}^{(a)}(\eta)-1}{t}\right|
 \leq Ce^{C\epsilon}.
 \label{eq:DR-shifted-target-density-bound}
\end{align}
The Neretin polynomial is continuous along the flow by the same lemma.
Thus \eqref{eq:DR-shifted-target-score-integral} also gives the pointwise
limit in \eqref{eq:DR-local-likelihood-dominator}.
\end{proof}

\subsubsection{Lifting the physical deformation}

On this collar, use the canonical frames
\begin{equation}\label{eq:DR-canonical-physical-frame-seeds}
 f_\gamma^0(0)=0,\quad (f_\gamma^0)'(0)>0,
 \qquad
 g_\gamma^0(\infty)=\infty,\quad (g_\gamma^0)'(\infty)>0.
\end{equation}
For \(\gamma_t=\phi_{p,t}^{(a)}\gamma\), let
\(\delta^{\rm raw}f_\gamma^0\) and
\(\delta^{\rm raw}g_\gamma^0\) be the two raw real Schiffer tangents, and
set
\begin{align}
 W_{p,\gamma}^{(a),+}
 &:=\frac{\delta^{\rm raw}f_\gamma^0-
   \left.\partial_t\right|_{t=0}f_{\gamma_t}^0}
  {(f_\gamma^0)'},\label{eq:DR-explicit-interior-frame-field}\\
 W_{p,\gamma}^{(a),-}
 &:=\frac{\delta^{\rm raw}g_\gamma^0-
   \left.\partial_t\right|_{t=0}g_{\gamma_t}^0}
  {(g_\gamma^0)'}. \label{eq:DR-explicit-exterior-frame-field}
\end{align}
For the unique source coordinate $b$ of each frame, define
\begin{align}
 T_{p,t,\gamma}^{(a),+}(f_\gamma^0\circ b)
 &:=f_{\gamma_t}^0\circ
 \exp(tW_{p,\gamma}^{(a),+})\circ b,\label{eq:DR-explicit-interior-frame-lift}\\
 T_{p,t,\gamma}^{(a),-}(g_\gamma^0\circ b)
 &:=g_{\gamma_t}^0\circ
 \exp(tW_{p,\gamma}^{(a),-})\circ b.
 \label{eq:DR-explicit-exterior-frame-lift}
\end{align}

Let $\widehat T_{p,t}^{(a),\rm eq}$ denote the transformation of the
full frame bundle obtained by combining the loop flow with
\eqref{eq:DR-explicit-interior-frame-lift}--
\eqref{eq:DR-explicit-exterior-frame-lift}.  Tests $F_{\rm fr}$ below
are nonnegative measurable functions supported in its target localization.

\begin{proposition}
\label{prop:DR-measurable-frame-lift}
For sufficiently small $|t|$, $\widehat T_{p,t}^{(a),\rm eq}$ is a
bijection between the source and target localizations.  The map and its
inverse are jointly measurable in time and frame.  The source fields satisfy
\[
 W_{p,\gamma}^{(a),+}\in\operatorname{Lie}\operatorname{Aut}(\mathbb D),
 \qquad W_{p,\gamma}^{(a),-}\in\operatorname{Lie}\operatorname{Aut}(\mathbb D^*),
\]
and
\begin{align}
 (T_{p,t,\gamma}^{(a),\pm})_*H_\gamma^\pm
 &=H_{\gamma_t}^\pm,\label{eq:DR-frame-Haar-pushforward}\\
 \int_{\mathcal F_{\rm fr}}
 F_{\rm fr}\circ\widehat T_{p,t}^{(a),\rm eq}\,dM_\kappa^{\rm fr}
 &=\int_{\mathcal F_{\rm fr}}F_{\rm fr}J_{p,t}^{(a)}
   \,dM_\kappa^{\rm fr}.
 \label{eq:DR-full-frame-finite-RN}
\end{align}
\end{proposition}

\begin{proof}
Choose a smooth real radial cutoff $c_D$ equal to one on
$\{\delta/2\leq|\xi|\leq2R\}$ and supported in
$\{\delta/3<|\xi|<3R\}$.  Let $\widehat\phi_{p,t}^{(a)}$ be the
spherical flow of
\begin{equation}\label{eq:DR-global-physical-extension-field}
 X_p^{(a)}
 :=c_Du_{-p}^{(a)}\partial_\xi
   +\overline{c_Du_{-p}^{(a)}}\partial_{\bar\xi}.
\end{equation}
It agrees with $\phi_{p,t}^{(a)}$ near the loop collar for small $|t|$
and fixes neighborhoods of $0$ and $\infty$.  Put
$\beta_t^{(a)}:=\mu_{\widehat\phi_{p,t}^{(a)}}$.
The flow equation gives
\begin{equation}\label{eq:DR-global-physical-extension-support}
 \begin{aligned}
 \operatorname{supp}\beta_t^{(a)}
 &\subset\bigcup_{|s|\leq\epsilon}
 (\widehat\phi_{p,s}^{(a)})^{-1}
 \operatorname{supp}\bar\partial(c_Du_{-p}^{(a)})\\
 &\Subset\{\delta/4<|\xi|<3\delta/4\}
 \cup\{3R/2<|\xi|<4R\}.
 \end{aligned}
\end{equation}
Smooth dependence on time on this fixed compact set and
$\beta_0^{(a)}=0$ allow $\epsilon$ to be chosen so that
\begin{equation}\label{eq:DR-global-physical-extension-control}
 \sup_{a\in\{1,2\},\,|t|\leq\epsilon}
 \left(\|\beta_t^{(a)}\|_\infty
       +\|\partial_t\beta_t^{(a)}\|_\infty\right)<\infty,
 \qquad
 \sup_{a,t}\|\beta_t^{(a)}\|_\infty<1.
\end{equation}
The interior pullback is
\[
 \beta_{t,\gamma}^{(a),+}
 :=(\beta_t^{(a)}\circ f_\gamma^0)
   \frac{\overline{(f_\gamma^0)'}}{(f_\gamma^0)'},
\]
with the same norm bounds.  Let $\beta_{t,\gamma}^{(a),-}$ be the
exterior pullback expressed on $\mathbb D$ by inversion.
Only the inner annulus in \eqref{eq:DR-global-physical-extension-support}
meets $U_+$; Schwarz' lemma applied to $(f_\gamma^0)^{-1}$ on
$B(0,\delta)$ gives
\[
 |(f_\gamma^0)^{-1}(w)|\leq |w|/\delta\leq3/4.
\]
After inversion, the exterior bound is $2/3$.  The Schwarz lemma also
gives $\delta\leq(f_\gamma^0)'(0),(g_\gamma^0)'(\infty)\leq R$.
Koebe's growth theorem therefore yields
\[
 |f_\gamma^0(z)|\leq\frac{R|z|}{(1-|z|)^2},
 \qquad
 |(\iota\circ g_\gamma^0\circ\iota)(z)|
 \leq\frac{|z|}{\delta(1-|z|)^2}.
\]
On either pullback support,
\[
 \frac{|z|}{(1-|z|)^2}\geq\frac{\delta}{4R},
 \qquad\text{hence}\qquad |z|\geq\frac{\delta}{16R}.
\]
Combining these bounds gives
\[
 \operatorname{supp}\beta_{t,\gamma}^{(a),\pm}
 \subset\{z:\delta/(16R)\leq|z|\leq3/4\}
 \qquad(a=1,2,\ |t|\leq\epsilon).
\]

Reflect these coefficients across $\mathbb S^1$ to obtain
circle-preserving solutions.  The reflected coefficients remain in a
fixed $L^\infty$ ball of radius below one.
The Ahlfors--Bers parameter theorem \cite[Theorems~10--11]{AB60}
gives analytic dependence
of the normalized solutions on the coefficient and hence $C^1$
dependence along these time paths.  Cauchy's formula on the disks where
the coefficients vanish gives the same dependence for finite jets.
The normalizing derivatives satisfy
\[
 \frac\delta2\leq(f_{\gamma_t}^0)'(0)\leq2R,
 \qquad
 \frac\delta2\leq(g_{\gamma_t}^0)'(\infty)\leq2R,
\]
by the Schwarz lemma, applied also after inversion.  The denominators in the normalizations are therefore bounded away from
zero, and their finite jets and time derivatives are jointly measurable
and uniformly bounded on a smaller collar.

Before imposing the canonical frame normalizations, this construction
gives the raw Schiffer uniformizers $f_t^{\rm raw},g_t^{\rm raw}$.
Their source normalizations are
\[
 b_t^+=(f_{\gamma_t}^0)^{-1}\circ f_t^{\rm raw}
       \in\operatorname{Aut}(\mathbb D),
 \qquad
 b_t^-=(g_{\gamma_t}^0)^{-1}\circ g_t^{\rm raw}
       \in\operatorname{Aut}(\mathbb D^*),
\]
with $b_0^\pm=\mathrm{id}$.  Differentiating
$f_t^{\rm raw}=f_{\gamma_t}^0\circ b_t^+$ and its exterior analogue gives
\[
 W_{p,\gamma}^{(a),\pm}
 =\left.\partial_t\right|_0 b_t^\pm.
\]
This proves membership in the two real Lie algebras.
An interior source field has the form
\[
 W(z)=A-\overline A z^2+iBz,
 \qquad A=W(0)\in\mathbb C,\quad B=-iW'(0)\in\mathbb R.
\]
For $W=W_{p,\gamma}^{(a),+}$, the coefficients are explicitly
\begin{align*}
 W(0)&=\frac{\delta^{\rm raw}f_\gamma^0(0)}{(f_\gamma^0)'(0)},\\
 W'(0)&=\frac{(\delta^{\rm raw}f_\gamma^0)'(0)
       -\left.\partial_t\right|_0(f_{\gamma_t}^0)'(0)
       -(f_\gamma^0)''(0)W(0)}{(f_\gamma^0)'(0)}.
\end{align*}
The finite-jet bounds give measurability in $\gamma$ and, for a constant
$C$ depending only on the collar and $p$,
\[
 |A|+|B|=|W(0)|+|W'(0)|\leq C.
\]
The same argument applies to the inverted exterior field
$(\iota_*W)(z)=-z^2W(1/z)$.

Smoothness of the exponential gives joint measurability of the lifts.
For $\eta=\gamma_t$ and target source coordinate $b'$, their inverse is
\[
 b=\exp\{-tW_{p,\phi_{p,-t}^{(a)}\eta}^{(a),\pm}\}\circ b'.
\]
This inverse is jointly measurable.  The map
\[
 b\longmapsto\exp(tW_{p,\gamma}^{(a),\pm})\circ b
\]
is left translation, so Haar invariance proves
\eqref{eq:DR-frame-Haar-pushforward}.

Integrate over the target frame fibers:
\[
 K_{F_{\rm fr}}(\eta)
 :=\int F_{\rm fr}(\eta,\widehat f,\widehat g)
         \,dH_\eta^+(\widehat f)\,dH_\eta^-(\widehat g).
\]
The frame kernels make this function measurable.  Extend
\eqref{eq:DR-physical-finite-restriction-RN} to nonnegative tests by
monotone convergence and apply Tonelli's theorem:
\begin{align*}
 \int F_{\rm fr}\circ\widehat T_{p,t}^{(a),\rm eq}\,dM_\kappa^{\rm fr}
 &=\int K_{F_{\rm fr}}(\phi_{p,t}^{(a)}\gamma)\,
      d\nu_\kappa^{\rm sph,or}(\gamma)\\
 &=\int K_{F_{\rm fr}}(\eta)J_{p,t}^{(a)}(\eta)\,
      d\nu_\kappa^{\rm sph,or}(\eta).
\end{align*}
Tonelli's theorem identifies the last line with the right side of
\eqref{eq:DR-full-frame-finite-RN}.  The derivative at zero is
\begin{align*}
 \left.\partial_t\right|_0
 T_{p,t,\gamma}^{(a),+}(f_\gamma^0\circ b)
 &=\left\{\left.\partial_t\right|_0f_{\gamma_t}^0
       +(f_\gamma^0)'W_{p,\gamma}^{(a),+}\right\}\circ b
   =\delta^{\rm raw}f_\gamma^0\circ b,\\
 \left.\partial_t\right|_0
 T_{p,t,\gamma}^{(a),-}(g_\gamma^0\circ b)
 &=\left\{\left.\partial_t\right|_0g_{\gamma_t}^0
       +(g_\gamma^0)'W_{p,\gamma}^{(a),-}\right\}\circ b
   =\delta^{\rm raw}g_\gamma^0\circ b.
\end{align*}
These are the equivariant extensions of the raw seed tangents.
\end{proof}

\subsubsection{Integration by parts on the frame space}

Let $\widehat L_{-p}^C$ be the Cauchy-normalized raw lift to
$\mathcal F_{\rm fr}$ of $-\xi^{1-p}\partial_\xi$.
Use bounded measurable tests $F_{\rm fr}$ supported on a collar
\eqref{eq:DR-physical-collar-localization}, with compact support in the
group coordinate and the two canonical source coordinates.
Extend each test and its short pullbacks by zero outside their
localizations; all supremum norms below are over the full frame space.
For the two lifted families of Proposition~\ref{prop:DR-measurable-frame-lift},
write
\[
 \widehat L_{-p}^{(a),\rm eq}F_{\rm fr}
 :=\left.\partial_t\right|_0
 (F_{\rm fr}\circ\widehat T_{p,t}^{(a),\rm eq}),
 \qquad
 \widehat L_{-p}^{\rm eq}
 :=\tfrac12(\widehat L_{-p}^{(1),\rm eq}
          -i\widehat L_{-p}^{(2),\rm eq}).
\]
The test class consists of functions whose actual difference quotients satisfy
\[
 \lim_{t\to0}\max_{a\in\{1,2\}}
 \left\|
 \frac{F_{\rm fr}\circ\widehat T_{p,t}^{(a),\rm eq}-F_{\rm fr}}t
 -\widehat L_{-p}^{(a),\rm eq}F_{\rm fr}
 \right\|_{\sup}=0.
\]
The tests are also $C^1$ along the real and imaginary parts of the two
source-frame anomaly fields, with bounded derivatives on the localization.
For compactly supported $\mu\in L^\infty(\mathbb D)$ and $N\geq0$, put
\begin{equation}\label{eq:DR-raw-variable-IBP}
 \widehat R_{\mu,N}
 :=\sum_{j=0}^N\widehat m_j\widehat L_{-j-2}^C.
\end{equation}

Set
\begin{align}
 \delta_{\mu,N}
 &:=-\frac{c_{\mathrm m}}{12}
   \sum_{j=0}^N\widehat m_j\pi_{j+2}
   -\sum_{j=0}^N
     \widehat L_{-j-2}^C\widehat m_j.
 \label{eq:DR-raw-partial-divergence}
\end{align}

For the finite-sum identity, require the same boundedness, support, and
differentiability conditions for $\widehat m_jF_{\rm fr}$ along the mode
$p=j+2$, for $0\leq j\leq N$.

\begin{lemma}\label{lem:DR-raw-physical-IBP}
\begin{equation}\label{eq:DR-raw-mode-IBP}
 \int_{\mathcal F_{\rm fr}}\widehat L_{-p}^CF_{\rm fr}\,dM_\kappa^{\rm fr}
 =-\frac{c_{\mathrm m}}{12}
    \int_{\mathcal F_{\rm fr}}F_{\rm fr}\pi_p\,dM_\kappa^{\rm fr},
\end{equation}
and
\begin{align}
 \int_{\mathcal F_{\rm fr}}\widehat R_{\mu,N}F_{\rm fr}\,dM_\kappa^{\rm fr}
 &=\int_{\mathcal F_{\rm fr}}F_{\rm fr}\delta_{\mu,N}\,dM_\kappa^{\rm fr}.
 \label{eq:DR-raw-partial-IBP}
\end{align}
\end{lemma}

\begin{proof}
The collar conditions place the loops in $\operatorname{Loop}^{r_0}$,
where $r_0:=\tfrac12\min\{\delta,R^{-1}\}\in(0,1)$.  If $K_+$ and $K_-$
are the compact source-coordinate supports, then
\[
 M_\kappa^{\rm fr}(\operatorname{supp}F_{\rm fr})
 \leq\nu_\kappa^{\rm sph,or}(\operatorname{Loop}^{r_0})
      H_{\mathbb D}(K_+)H_{\mathbb D^*}(K_-)<\infty.
\]
Proposition~\ref{prop:DR-measurable-frame-lift} places all short pullbacks
of the support in a common larger localization $E_F$, with the same
type of finite-measure bound.

Subtract the $t=0$ instance of \eqref{eq:DR-full-frame-finite-RN} and
divide by $t$:
\[
 \int\frac{F_{\rm fr}\circ\widehat T_{p,t}^{(a),\rm eq}-F_{\rm fr}}t
       \,dM_\kappa^{\rm fr}
 =\int F_{\rm fr}\frac{J_{p,t}^{(a)}-1}t\,dM_\kappa^{\rm fr}.
\]
On the left, the error is at most
\[
 M_\kappa^{\rm fr}(E_F)
 \left\|
 \frac{F_{\rm fr}\circ\widehat T_{p,t}^{(a),\rm eq}-F_{\rm fr}}t
 -\widehat L_{-p}^{(a),\rm eq}F_{\rm fr}
 \right\|_{\sup}\longrightarrow0.
\]
On the right, \eqref{eq:DR-shifted-target-density-bound} gives the
integrable majorant
\[
 \left|F_{\rm fr}\frac{J_{p,t}^{(a)}-1}t\right|
 \leq Ce^{C\epsilon}\|F_{\rm fr}\|_{\sup}
       \mathbf1_{\operatorname{supp}F_{\rm fr}}.
\]
Dominated convergence therefore yields
\[
 \int\widehat L_{-p}^{(a),\rm eq}F_{\rm fr}\,dM_\kappa^{\rm fr}
 =\lim_{t\to0}\int F_{\rm fr}
       \frac{J_{p,t}^{(a)}-1}{t}\,dM_\kappa^{\rm fr}
 =\int F_{\rm fr}s_p^{(a)}\,dM_\kappa^{\rm fr}.
\]
Taking the complex combination
$\tfrac12(\widehat L_{-p}^{(1),\rm eq}
-i\widehat L_{-p}^{(2),\rm eq})$ gives
\begin{equation}\label{eq:DR-equivariant-mode-IBP}
 \int_{\mathcal F_{\rm fr}}\widehat L_{-p}^{\rm eq}F_{\rm fr}\,dM_\kappa^{\rm fr}
 =-\frac{c_{\mathrm m}}{12}
   \int_{\mathcal F_{\rm fr}}F_{\rm fr}\pi_p\,dM_\kappa^{\rm fr}.
\end{equation}

The difference $\widehat L_{-p}^C-\widehat L_{-p}^{\rm eq}$ is the sum
of the interior source-frame field
\eqref{eq:DR-source-frame-anomaly-field} and its exterior analogue.
For a real or imaginary part $Y$ of either anomaly field, its local
flow $\Phi_t^Y$ preserves the corresponding Haar measure $H$ by
\eqref{eq:DR-source-frame-anomaly-zero-divergence}.  On each fiber,
bounded directional derivatives and compact support give
\[
 \int YF_{\rm fr}\,dH
 =\lim_{t\to0}\int\frac{F_{\rm fr}\circ\Phi_t^Y-F_{\rm fr}}t\,dH=0.
\]
The derivatives are integrable on $E_F$, so Fubini's theorem gives
\begin{equation}\label{eq:DR-source-anomaly-full-frame-zero}
 \int_{\mathcal F_{\rm fr}}
 (\widehat L_{-p}^C-\widehat L_{-p}^{\rm eq})F_{\rm fr}\,
 dM_\kappa^{\rm fr}=0.
\end{equation}
Combining this with \eqref{eq:DR-equivariant-mode-IBP} proves
\eqref{eq:DR-raw-mode-IBP}.

Finally, apply \eqref{eq:DR-raw-mode-IBP} to $\widehat m_jF_{\rm fr}$ and use the
Leibniz rule:
\begin{align}
 \int_{\mathcal F_{\rm fr}}\widehat m_j
   \widehat L_{-j-2}^CF_{\rm fr}\,dM_\kappa^{\rm fr}
 &=-\frac{c_{\mathrm m}}{12}
   \int_{\mathcal F_{\rm fr}}\widehat m_jF_{\rm fr}\pi_{j+2}\,dM_\kappa^{\rm fr}
   -\int_{\mathcal F_{\rm fr}}F_{\rm fr}
     \widehat L_{-j-2}^C\widehat m_j\,dM_\kappa^{\rm fr}.
 \label{eq:DR-full-frame-multiplier-step}
\end{align}
Summing \eqref{eq:DR-full-frame-multiplier-step} for
$0\leq j\leq N$ proves \eqref{eq:DR-raw-partial-IBP}.
\end{proof}

\section{Proof of the right integration-by-parts formula}
\label{app:right-IBP}

\subsection{Normalized frames}

Fix an algebraic expression for $F\in\mathcal A_{\mathrm{fl}}$ in
\eqref{eq:DR-flow-saturated-core}.  On a normalized frame $h\in\Sigma$,
evaluate the generators of $\mathcal A_0$ on its conformal maps.
For each composition $\Phi$, use the fixed extension $\widehat\Phi$ and
\eqref{eq:DR-composition-source-map}--\eqref{eq:DR-composition-normalized-output}.
This defines the bounded measurable extension and its invariant lift
\begin{equation}\label{eq:DR-ambient-observable-extension}
 F^\Sigma:\Sigma\longrightarrow\mathbb C,
 \qquad
 \widehat F:=F^\Sigma\circ s.
\end{equation}
For $A>0$, put
\begin{equation}\label{eq:DR-K-localized-slice}
 \Sigma_A:=\{h\in\Sigma:|k(h)|\leq A\}.
\end{equation}
For $\zeta\in C_c^1((-A,A))$, define
\begin{equation}\label{eq:DR-intrinsic-and-ambient-cutoff-tests}
 G_0:=\zeta(k)F\quad\text{on }\mathcal W,
 \qquad
 G_0^\Sigma:=\zeta(k)F^\Sigma\quad\text{on }\Sigma,
\end{equation}
and
\begin{equation}\label{eq:DR-ambient-cutoff-extension}
 \widehat G_0:=G_0^\Sigma\circ s.
\end{equation}
We use the product coordinates $(M,h)=\Theta^{-1}(x)$ on
$\mathcal F_{\rm fr}$ and write $k$ also for its invariant lift $k\circ s$.
The estimates in Appendix~\ref{app:uniform-response} apply to these
extensions.  On $\Sigma_{\rm r}$, conformal removability and
Lemma~\ref{lem:B1-unique-welding-flow-stability} identify the constructed
frames with the intrinsic welding frames:
\begin{equation}\label{eq:DR-ambient-extension-restriction}
 F^\Sigma(h)=F(\operatorname{weld}(h)),
 \qquad h\in\Sigma_{\rm r}.
\end{equation}

Fix \(0<\delta<R<\infty\), \(0<\lambda_0<1\), and
\(0<c<C<\infty\).  Let \(\mathcal K\) be a family of framed loops
\[
 x=(\gamma,U_+,U_-,\widehat f,\widehat g).
\]
Use the marks $z_+,z_-$ from Appendix~\ref{sec:DR-frame-spaces}, put
$\lambda_+:=\widehat f'(0)$, and assume
\begin{align}
 B(0,\delta)&\subset U_+,
 &
 \widehat{\mathbb C}\setminus\overline{B(0,R)}&\subset U_-,
 \label{eq:DR-quantitative-collar}\\
 |z_+|&\leq\lambda_0\delta,
 &
 |\iota(z_-)|&\leq\frac{\lambda_0}{R},
 &
 c\leq|\lambda_+|&\leq C,
 \label{eq:DR-quantitative-frame-data}
\end{align}
for every $x\in\mathcal K$.
Let \(f_\gamma^0,g_\gamma^0\) be the canonical seeds from
\eqref{eq:DR-canonical-physical-frame-seeds}, and set
\[
 a_x:=(f_\gamma^0)^{-1}\circ\widehat f,
 \qquad
 b_x:=(g_\gamma^0)^{-1}\circ\widehat g.
\]

Put $u=z_+$ and $v=\iota(z_-)$, and retain the group coordinate
$M(x)=N_x^{-1}$ from \eqref{eq:DR-Weil-product-chart}.  Formula
\eqref{eq:DR-explicit-slice-Mobius} becomes
\begin{equation}\label{eq:DR-normalizer-local-coordinates}
 N_x(\xi)=\frac{(1-uv)(\xi-u)}{\lambda_+(1-v\xi)}.
\end{equation}
For a $C^1$ curve $s\mapsto x_s$ through $x$, a dot denotes its derivative
at $s=0$.  Use the left-invariant trivialization to define tangent norms
on $G$, and let $X$ be a real tangent vector at $x$.
The constant $C_0$ below depends only on $\delta,R,\lambda_0,c,C$.

\begin{lemma}\label{lem:DR-quantitative-normalization}
The frame coordinates satisfy
\begin{equation}\label{eq:DR-quantitative-transition-bounds}
 |a_x(0)|\leq\lambda_0,
 \qquad
 |(\iota\circ b_x\circ\iota)(0)|\leq\lambda_0.
\end{equation}
Consequently,
\begin{equation}\label{eq:DR-quantitative-transition-compactness}
 \{a_x:x\in\mathcal K\}\Subset\operatorname{Aut}(\mathbb D),
 \qquad
 \{b_x:x\in\mathcal K\}\Subset\operatorname{Aut}(\mathbb D^*).
\end{equation}

For the group coordinate,
\begin{equation}\label{eq:DR-group-projection-finite-jet-bound}
 \|dM(X)\|
 \leq C_0\{|du(X)|+|dv(X)|+|d\lambda_+(X)|\}.
\end{equation}
\end{lemma}

\begin{proof}
Let \(\psi_+=(f_\gamma^0)^{-1}\).  The map
\(z\mapsto\psi_+(\delta z)\) maps \(\mathbb D\) into itself and fixes
zero.  Schwarz' lemma and \eqref{eq:DR-quantitative-frame-data} give
\[
 |a_x(0)|=|\psi_+(z_+)|
 \leq\frac{|z_+|}{\delta}\leq\lambda_0.
\]
Likewise,
\[
 \widetilde g_\gamma^0
 :=\iota\circ g_\gamma^0\circ\iota
\]
maps \(\mathbb D\) onto \(\iota(U_-)\), which contains
\(B(0,R^{-1})\).  Applying Schwarz' lemma to
\((\widetilde g_\gamma^0)^{-1}\) gives
\[
 |(\iota\circ b_x\circ\iota)(0)|
 \leq R|\iota(z_-)|\leq\lambda_0.
\]
For $w\in\mathbb D$ and $\theta\in\T$, every disk automorphism
with value $w$ at zero has the form
\[
 z\longmapsto\frac{w+e^{i\theta}z}{1+\overline w e^{i\theta}z}.
\]
The parameter set $\{|w|\leq\lambda_0\}\times\T$ is compact,
which proves \eqref{eq:DR-quantitative-transition-compactness}.

Since
\[
 |uv|\leq\lambda_0^2\frac{\delta}{R}<\lambda_0^2,
\]
the quantities \(\lambda_+^{-1}\), \((1-uv)^{-1}\), and
\((1-uv)^{-2}\) are uniformly bounded.  The parameter derivatives of
\eqref{eq:DR-normalizer-local-coordinates} satisfy
\[
 \begin{aligned}
 (\partial_uN_x)\circ N_x^{-1}(y)
   &=-\lambda_+^{-1}-\frac{2v}{1-uv}y,\\
 (\partial_vN_x)\circ N_x^{-1}(y)
   &=\frac{\lambda_+}{(1-uv)^2}y^2,\\
 (\partial_{\lambda_+}N_x)\circ N_x^{-1}(y)
   &=-\lambda_+^{-1}y.
 \end{aligned}
\]
The chain rule therefore gives
\begin{align}
 &\left.\partial_sN_{x_s}\right|_{s=0}\circ N_x^{-1}(y)
 \notag\\
 &\qquad=-\frac{\dot u}{\lambda_+}
 -\left(
    \frac{\dot\lambda_+}{\lambda_+}+\frac{2v\dot u}{1-uv}
  \right)y
 +\frac{\lambda_+\dot v}{(1-uv)^2}y^2.
 \label{eq:DR-explicit-normalizer-differential}
\end{align}
The denominator bounds imply
\begin{equation}\label{eq:DR-normalizer-differential-bound}
 \left\|
  \left.\partial_sN_{x_s}\right|_{s=0}\circ N_x^{-1}
 \right\|_{\mathfrak{sl}_2}
 \leq C_0\bigl(|\dot u|+|\dot v|+|\dot\lambda_+|\bigr).
\end{equation}
For left translation
$L_M(Q):=M\circ Q$, differentiation of $M=N^{-1}$ gives
\[
 d(L_{M^{-1}})_M\dot M=-\dot N\circ N^{-1}.
\]
Taking $\dot x=X$ proves
\eqref{eq:DR-group-projection-finite-jet-bound}.
\end{proof}

\subsection{Laurent approximation and test functions}

Let $\mu$ be a finite complex linear combination of the representatives
$\mu_{n,c}$, $n\geq2$, from \eqref{eq:DR-Fourier-Cauchy-vector}.
Using the support radius $r_+$ from that construction, choose
\begin{equation}\label{eq:DR-Laurent-three-radii}
 r_1:=\frac{r_+}{(1-r_+)^2}<r_2<r_3<\frac14.
\end{equation}
Koebe's growth and one-quarter theorems give
\begin{equation}\label{eq:DR-Laurent-Koebe-separation}
 |f(z)|\leq r_1\quad(z\in\operatorname{supp}\mu),
 \qquad B(0,1/4)\subset f(\mathbb D).
\end{equation}
Applied to $z\mapsto1/g(1/z)$, whose derivative at zero has modulus
$e^{-k}$, they also give
\begin{equation}\label{eq:DR-K-localizes-loop}
 \gamma\subset\{1/4\leq|\xi|\leq4e^k\}.
\end{equation}

Choose an identity neighborhood $U_A\Subset G$ and constants
\begin{equation}\label{eq:DR-Laurent-uniform-group-radii}
 0<r_{1,A}<r_{2,A}<r_{3,A}<R_A^0<\infty,
 \qquad 0<\lambda_A<1,\qquad r_{1,A}\leq\lambda_A r_{3,A},
\end{equation}
so that, for $(M,h)\in\overline U_A\times\Sigma_A$,
\begin{align}
 \{M(0)\}\cup M(f(\operatorname{supp}\mu))
 &\subset B(0,r_{1,A}),
 \label{eq:DR-Laurent-group-inner-separation}\\
 B(0,r_{3,A})&\subset MU_+,
 \label{eq:DR-Laurent-uniform-interior-ball}\\
 M(\gamma)&\subset\{\xi:r_{3,A}\leq|\xi|\leq R_A^0\},
 \label{eq:DR-Laurent-group-boundary-separation}\\
 |\iota(M(\infty))|&\leq\lambda_A/R_A^0,
 \label{eq:DR-Laurent-uniform-exterior-mark}\\
 0<c_A\leq|M'(0)|&\leq C_A<\infty.
 \label{eq:DR-Laurent-uniform-interior-tangent}
\end{align}
To obtain these bounds, start with radii strictly between $r_1$ and
$1/4$, and take $R_A^0>4e^A$.  Near the identity, $M$ converges
uniformly to the identity on $\overline{B(0,8e^A)}$, while
$M(0)$, $\iota(M(\infty))$, and $M'(0)$ vary continuously.
Shrink $U_A$ so that its poles stay outside this disk.
The displayed inner and outer disks are disjoint from $M\gamma$ and
contain the respective marks; connectedness places them in $MU_+$ and
$MU_-$.  Since $f(\mathbb D)\subset B(0,4e^A)$, this also makes
$M\circ f$ holomorphic on $\mathbb D$.

Relative to the canonical frames in
\eqref{eq:DR-canonical-physical-frame-seeds}, write
\begin{equation}\label{eq:DR-local-seed-transition-automorphisms}
 M\circ f=f_{M\gamma}^0\circ a_{M,h},
 \qquad M\circ g=g_{M\gamma}^0\circ b_{M,h}.
\end{equation}
Lemma~\ref{lem:DR-quantitative-normalization}, with
$\delta=r_{3,A}$, $R=R_A^0$, and $\lambda_0=\lambda_A$, gives
\begin{equation}\label{eq:DR-local-seed-transition-compactness}
 \begin{aligned}
 \{a_{M,h}:(M,h)\in\overline U_A\times\Sigma_A\}
   &\Subset\operatorname{Aut}(\mathbb D),\\
 \{b_{M,h}:(M,h)\in\overline U_A\times\Sigma_A\}
   &\Subset\operatorname{Aut}(\mathbb D^*).
 \end{aligned}
\end{equation}
In particular,
\begin{equation}\label{eq:DR-transition-support-compactness}
 \bigcup_{(M,h)\in\overline U_A\times\Sigma_A}
 a_{M,h}(\operatorname{supp}\mu)\Subset\mathbb D.
\end{equation}

Use $\widehat R_{\mu,N}$ and $\delta_{\mu,N}$ from
\eqref{eq:DR-raw-variable-IBP}--\eqref{eq:DR-raw-partial-divergence}, and set
\[
 \Delta_N:=\widehat R_{\mu,N}-\widehat R_\mu^{\mathrm C},
 \qquad q:=r_1/r_2<1,\qquad q_A:=r_{1,A}/r_{2,A}<1.
\]
Fix $\chi\in C_c^1(U_A)$ and regard it as $\chi\circ M$ on the
frame space.  Constants below may depend on the fixed
representative, test, and localization, but are independent of $N$.

\begin{lemma}\label{lem:DR-localized-Laurent-estimates}
For $N\geq0$,
\begin{equation}\label{eq:DR-Laurent-geometric-rate}
 \left\|(\Delta_N\widehat F)|_\Sigma\right\|_{\sup}
 +\left\|\delta_{\mu,N}|_\Sigma
       -\frac{c_{\mathrm L}}{12}\sigma_h(\mu)\right\|_{\sup}
 \leq C_F(N+2)^2q^{N+1}.
\end{equation}
On the localized frame space,
\begin{align}
 \sup_{\overline U_A\times\Sigma_A}|\Delta_N\widehat F|
 &\leq C_{F,A}(N+2)^2q_A^{N+1},
 \label{eq:DR-Laurent-full-frame-response-rate}\\
 \sup_{\overline U_A\times\Sigma_A}
 \left|\delta_{\mu,N}-\frac{c_{\mathrm L}}{12}\sigma_h(\mu)\right|
 &\leq C_A(N+2)^2q_A^{N+1},
 \label{eq:DR-Laurent-full-frame-divergence-rate}\\
 \shortintertext{and}
 \sup_{\overline U_A\times\Sigma_A}
 |\Delta_N(\chi\widehat G_0)|
 &\leq C_{\chi,F,\zeta,A}(N+2)^2q_A^{N+1}.
 \label{eq:DR-Laurent-cutoff-full-frame-response-rate}
\end{align}
\end{lemma}

\begin{proof}
With $\widehat f=M\circ f$, define
\begin{align}
 V_{\mu,\widehat f}^{(N)}(\xi)
 &:=-\sum_{j=0}^N\widehat m_j(\widehat f)\xi^{-j-1},
 \label{eq:DR-off-slice-partial-vector}\\
 E_N^M(\xi)&:=V_{\mu,\widehat f}(\xi)-V_{\mu,\widehat f}^{(N)}(\xi).
 \label{eq:DR-off-slice-vector-tail}
\end{align}
The moment integral and Koebe distortion on $\operatorname{supp}\mu$
give
\begin{equation}\label{eq:DR-off-slice-moment-bound}
 \sup_{\overline U_A\times\Sigma_A}|\widehat m_j(M\circ f)|
 \leq C_A r_{1,A}^{\,j}.
\end{equation}
For $r=0,1,2$, termwise differentiation therefore yields
\begin{align}
 \sup_{\substack{(M,h)\in\overline U_A\times\Sigma_A\\
                  |\xi|\geq r_{2,A}}}
 |\partial_\xi^rE_N^M(\xi)|
 &\leq C_A r_{2,A}^{-r-1}
   \sum_{j>N}(j+1)\cdots(j+r)q_A^j\notag\\
 &\leq C_{r,A}(N+2)^r q_A^{N+1}.
 \label{eq:DR-off-slice-C2-tail}
\end{align}
The product is $1$ for $r=0$; the last inequality is
\eqref{eq:DR-geometric-tail-estimate}.
At $M=\mathrm{id}$ write $E_N:=E_N^{\mathrm{id}}$.
The same calculation, now using the normalized bounds
\eqref{eq:DR-Laurent-Koebe-separation}, gives
\begin{equation}\label{eq:DR-Laurent-C2-tail}
 \sup_{\substack{h\in\Sigma\\|\xi|\geq r_2}}
 |\partial_\xi^rE_N(\xi)|
 \leq C_r(N+2)^r q^{N+1},\qquad r=0,1,2.
\end{equation}

Choose smooth radial cutoffs $\eta_A,\eta$, equal to zero up to
$r_{2,A},r_2$ and to one before $r_{3,A},r_3$, respectively.
Their derivatives have compact support in the intervening annuli.
Put
\[
 E_N^{M,\eta_A}:=\eta_AE_N^M,\qquad E_N^\eta:=\eta E_N.
\]
Both agree with the corresponding tail near the physical curve, and
$\bar\partial E_N^{M,\eta_A}=E_N^M\bar\partial\eta_A$.
In the chart $w=\xi^{-1}$ a vector field with coefficient $E(\xi)$
has coefficient $-w^2E(w^{-1})$.  Thus the preceding estimates give
\begin{align}
 \sup_{\overline U_A\times\Sigma_A}
 \left(\|\bar\partial E_N^{M,\eta_A}\|_\infty
       +\|E_N^{M,\eta_A}\|_{C^1(\widehat{\mathbb C})}\right)
 &\leq C_A(N+2)^2q_A^{N+1},
 \label{eq:DR-full-frame-cutoff-tail}\\
 \sup_{h\in\Sigma}
 \|\bar\partial E_N^\eta\|_\infty
 &\leq C(N+2)^2q^{N+1}.
 \label{eq:DR-Laurent-dbar-tail}
\end{align}

These physical coefficients pull back to fixed compact subsets of
the source disks.  In the normalized case, Schwarz' lemma gives
$|f^{-1}(\xi)|\leq4|\xi|$ for $|\xi|<1/4$, while Koebe growth gives
$|f(z)|\leq|z|/(1-|z|)^2$.  Consequently,
\[
 \operatorname{supp}f^*(\bar\partial E_N^\eta)
 \subset\{z:r_2/4\leq|z|\leq s\},
 \qquad s:=4r_3<1.
\]
For the localized case put $S_A:=\operatorname{supp}\bar\partial\eta_A$,
$u=M(0)$, and choose $r_{*,A}<r_{3,A}$ with
$S_A\subset\{r_{2,A}<|\xi|\leq r_{*,A}\}$.
Schwarz--Pick for $\widehat f^{-1}:B(0,r_{3,A})\to\mathbb D$ gives
\[
 |\widehat f^{-1}(\xi)|
 \leq\left|\frac{r_{3,A}(\xi-u)}
                   {r_{3,A}^2-\overline u\,\xi}\right|
 \leq\frac{r_{3,A}(r_{*,A}+r_{1,A})}
             {r_{3,A}^2+r_{*,A}r_{1,A}}
 =:s_{+,A}<1,\qquad \xi\in S_A.
\]
For $\xi=\widehat f(z)\in S_A$, Koebe growth also gives
\[
 r_{2,A}-r_{1,A}
 \leq|\widehat f(z)-u|
 \leq C_A\frac{|z|}{(1-|z|)^2}.
\]
Choose $s_{-,A}>0$ so that
$C_As_{-,A}/(1-s_{-,A})^2<r_{2,A}-r_{1,A}$.
Then
\begin{equation}\label{eq:DR-Laurent-pulled-support-annulus}
 \operatorname{supp}\widehat f^*(\bar\partial E_N^{M,\eta_A})
 \subset\{z:s_{-,A}\leq|z|\leq s_{+,A}\}.
\end{equation}
Conformal pullback preserves the $L^\infty$ norm of a Beltrami
coefficient, so the bounds in \eqref{eq:DR-full-frame-cutoff-tail}--
\eqref{eq:DR-Laurent-dbar-tail} remain valid after pullback.

Let $w_N,w$ be the complex angular coefficients induced by
$V_{\mu,f}^{(N)},V_{\mu,f}$ with the same Cauchy choice of frames.
Extend $D_w$ complex linearly in $w$.  On the source support just
obtained,
\[
 |e^{i\theta}-z|\geq1-s,\qquad |z|\geq r_2/4.
\]
These inequalities bound the Cauchy kernels, their first boundary
derivatives, and the normalization terms at zero.  The reflected
kernels obey the same bounds, hence
\begin{equation}\label{eq:DR-Laurent-boundary-field-tail}
 \|w_N-w\|_{C^1(\mathbb S^1)}
 \leq C\|\bar\partial E_N^\eta\|_\infty
 \leq C'(N+2)^2q^{N+1}.
\end{equation}
For averaged generators, \eqref{eq:B1-J-response} gives
\begin{align}
 (D_{w_N}-D_w)J_{a,b}(h)
 &=-\frac1{2\pi}\int_0^{2\pi}
       \{a(w_N-w)\}'(\theta)b(h(e^{i\theta}))\,d\theta,
 \label{eq:DR-Laurent-averaged-differential}\\
 |(D_{w_N}-D_w)J_{a,b}(h)|
 &\leq\|b\|_\infty
       (\|a'\|_\infty+\|a\|_\infty)\|w_N-w\|_{C^1}.
 \label{eq:DR-Laurent-averaged-tail}
\end{align}
For coefficient and character generators, put
$\beta_N:=f^*(\bar\partial E_N^\eta)$; this coefficient has the fixed
support above.
The first Ahlfors--Bers variation, Cauchy's formula, and
\eqref{eq:DR-composition-exterior-character} give
\[
 \sup_\Sigma\bigl(|\Delta_N\widehat{a_m}|
                 +|\Delta_N\widehat{\chi_j}|\bigr)
 \leq C_{m,j}\sup_{h\in\Sigma}\|\beta_N\|_\infty
 \leq C_{m,j}'(N+2)^2q^{N+1}.
\]
The reflected variation gives the same bound for $\overline{a_m}$.
This bound for $\chi_j$ uses the interior normalizing derivative in
\eqref{eq:DR-composition-exterior-character} and is independent of $k$.

For a fixed composition $\Phi(e^{i\theta})=e^{i\phi(\theta)}$,
the transformed angular coefficient satisfies
\[
 (\operatorname{Ad}_{\Phi^{-1}}w)(\theta)
 =\frac{w(\phi(\theta))}{\phi'(\theta)},\qquad
 \|\operatorname{Ad}_{\Phi^{-1}}w\|_{C^1}
 \leq C_\Phi\|w\|_{C^1}.
\]
The fixed disk extension $\widehat\Phi$ has bounded differential on the
source compact sets, so the coefficient and character estimates also
persist under composition.  For a product of generators,
\[
 \Delta_N\!\left(\prod_{\ell=1}^m\widehat F_\ell\right)
 =\sum_{\ell=1}^m(\Delta_N\widehat F_\ell)
                      \prod_{r\ne\ell}\widehat F_r,
\]
with bounded undifferentiated factors.  The cylinder chain rule and this
identity prove the response bound in
\eqref{eq:DR-Laurent-geometric-rate}.
The same calculation using \eqref{eq:DR-Laurent-pulled-support-annulus}
and the compact transition maps
\eqref{eq:DR-local-seed-transition-compactness} proves
\eqref{eq:DR-Laurent-full-frame-response-rate}.

For the divergence, the moment bound and Cauchy estimates on the
separated contour give
\[
 |\widehat m_j\pi_{j+2}|
 +|\widehat L_{-j-2}^C\widehat m_j|
 \leq C_A(j+1)q_A^j
 \quad\text{on }\overline U_A\times\Sigma_A.
\]
At $M=\mathrm{id}$ the bound is $C(j+1)q^j$, uniformly on $\Sigma$.
The integral identities \eqref{eq:DR-matter-contraction-proof} and
\eqref{eq:DR-variable-coefficient-trace} apply to $\widehat f$:
\[
 \sum_{j\geq0}\widehat m_j\pi_{j+2}
 =-(\mathcal S\widehat f,\mu)=\sigma_h(\mu),\qquad
 \sum_{j\geq0}\widehat L_{-j-2}^C\widehat m_j
 =\frac{26}{12}(\mathcal S\widehat f,\mu).
\]
Here $\mathcal S(M\circ f)=\mathcal Sf$.  Thus
\begin{equation}\label{eq:DR-raw-central-balance}
 -\frac{c_{\mathrm m}}{12}\sum_{j\geq0}\widehat m_j\pi_{j+2}
 -\sum_{j\geq0}\widehat L_{-j-2}^C\widehat m_j
 =\frac{26-c_{\mathrm m}}{12}\sigma_h(\mu)
 =\frac{c_{\mathrm L}}{12}\sigma_h(\mu).
\end{equation}
Summing the coefficient tails by \eqref{eq:DR-geometric-tail-estimate}
proves both divergence bounds.

It remains to control the two cutoffs.  Put
$\varepsilon_{N,A}:=(N+2)^2q_A^{N+1}$.
Since the physical vector of $\Delta_N$ is $-E_N^M$, the Cauchy
projections give, with $\widehat\psi=\widehat f^{-1}$,
\[
 \begin{aligned}
 \Delta_N\widehat f(z)
 &=\frac1{2\pi i}\oint_{|\xi|=r_{2,A}}
       K_{\widehat\psi}(\widehat f(z),\xi)E_N^M(\xi)\,d\xi,\\
 \Delta_N\widehat g(z)&=-E_N^M(\widehat g(z)).
 \end{aligned}
\]
The first formula holds near zero and extends analytically to
$\mathbb D$.  The marks and contour are uniformly separated.
Differentiating under the integral and using the inverted exterior
coefficient $w^2E_N^M(w^{-1})$ therefore yields
\begin{equation}\label{eq:DR-Laurent-normalization-jet-rate}
 \sup_{\overline U_A\times\Sigma_A}
 \bigl\{|du(\Delta_N)|+|dv(\Delta_N)|
        +|d\lambda_+(\Delta_N)|\bigr\}
 \leq C_A\varepsilon_{N,A}.
\end{equation}
Set $\lambda_-:=(\iota\circ\widehat g\circ\iota)'(0)$.
The exterior formula and its derivative control $\Delta_N\lambda_-$;
the reflected variation controls the conjugate jets.
Apply \eqref{eq:DR-group-projection-finite-jet-bound} to the real and
imaginary parts of $\Delta_N$ to obtain
\begin{equation}\label{eq:DR-Laurent-group-coordinate-rate}
 \sup_{\overline U_A\times\Sigma_A}\|dM(\Delta_N)\|
 \leq C_A\varepsilon_{N,A}.
\end{equation}
Substitution into \eqref{eq:DR-normalizer-local-coordinates} gives
\[
 g'(\infty)=\frac{(1-uv)^2}{\lambda_+\lambda_-},
 \qquad
 \frac{\Delta_Ng'(\infty)}{g'(\infty)}
 =-\frac{2(v\Delta_Nu+u\Delta_Nv)}{1-uv}
   -\frac{\Delta_N\lambda_+}{\lambda_+}
   -\frac{\Delta_N\lambda_-}{\lambda_-}.
\]
Since $k=\log|g'(\infty)|$, complex linear differentiation gives
\[
 \Delta_Nk=\frac12\left\{
 \frac{\Delta_Ng'(\infty)}{g'(\infty)}
 +\frac{\Delta_N\overline{g'(\infty)}}{\overline{g'(\infty)}}\right\}.
\]
The collar bounds and $e^{-A}\leq|g'(\infty)|\leq e^A$ bound
$\lambda_-$ and its reciprocal.  The preceding jet estimates imply
\begin{equation}\label{eq:DR-Laurent-capacity-rate}
 \sup_{\overline U_A\times\Sigma_A}|\Delta_Nk|
 \leq C_A\varepsilon_{N,A}.
\end{equation}
The product rule now gives
\begin{align}
 \Delta_N\widehat G_0
 &=\zeta(k)\Delta_N\widehat F
    +\zeta'(k)\widehat F\,\Delta_Nk,
 \label{eq:DR-Laurent-invariant-cutoff-and-group-rate}\\
 \Delta_N(\chi\widehat G_0)
 &=\chi\,\Delta_N\widehat G_0
    +\widehat G_0\,d\chi[dM(\Delta_N)].
 \label{eq:DR-Laurent-group-cutoff-product}
\end{align}
The bounds just proved establish
\eqref{eq:DR-Laurent-cutoff-full-frame-response-rate}.
\end{proof}

For a fixed $N\geq0$, set
\[
 H_j:=\widehat m_j\chi\widehat G_0,\qquad 0\leq j\leq N,
\]
and extend these tests by zero outside their localization.
In the next lemma, supremum norms are over $\mathcal F_{\rm fr}$.
For each $j$, let $Y$ range over the real and imaginary parts of the
interior source-frame anomaly for $p=j+2$ and its exterior analogue.

\begin{lemma}\label{lem:DR-localized-test-regularity}
\begin{equation}\label{eq:DR-raw-test-uniform-differentiability}
 \lim_{t\to0}\max_{\substack{0\leq j\leq N\\a\in\{1,2\}}}
 \left\|
 \frac{H_j\circ\widehat T_{j+2,t}^{(a),\mathrm{eq}}-H_j}{t}
 -\widehat L_{-j-2}^{(a),\mathrm{eq}}H_j
 \right\|_{\sup}=0.
\end{equation}
Each $H_j$ is $C^1$ along the fields $Y$, and
\[
 \max_{0\leq j\leq N}\ \max_Y\|YH_j\|_{\sup}<\infty.
\]
Both assertions also hold with $H_j$ replaced by $\chi\widehat G_0$.
\end{lemma}

\begin{proof}
The physical collars and compact source-coordinate ranges above satisfy
the hypotheses of Proposition~\ref{prop:DR-measurable-frame-lift}.
Choose the physical cutoffs to vanish near the compact ranges of both
marks.  Their pullbacks are supported in fixed annuli, by
\eqref{eq:DR-Laurent-group-inner-separation},
\eqref{eq:DR-Laurent-uniform-exterior-mark}, and
\eqref{eq:DR-local-seed-transition-compactness}.
For each mode, let $\beta_{x,t}$ denote the pullback of the physical
extension's Beltrami coefficient through the interior frame, or through
the exterior frame followed by inversion.
For some $0<q_0<1$ and sufficiently small $\epsilon_0>0$, uniformly
over the finitely many modes,
\[
 \sup_{x,\,|t|\leq\epsilon_0}
 \sum_{r=0}^2\|\partial_t^r\beta_{x,t}\|_\infty<\infty,
 \qquad
 \sup_{x,\,|t|\leq\epsilon_0}\|\beta_{x,t}\|_\infty\leq q_0<1.
\]
Here $x$ ranges over the localized frames; the time derivatives follow
from the smooth cutoff flows and the Beltrami composition formula.

For each principal-solution family used by the test and normalization,
let $K$ be its fixed compact coefficient support.  Denote the required
values and derivatives by $\operatorname{Jet}(\beta)$; their evaluation
points are uniformly separated from $K$.
The Ahlfors--Bers parameter theorem \cite[Theorem~11]{AB60} and
Cauchy's formula make this map holomorphic on the unit ball of the
closed subspace of $L^\infty(\mathbb C)$ supported in $K$.
Choose $q_0<q_1<1$.  Normal-family bounds and Cauchy estimates in this
subspace give
\[
 \sup_{\substack{\operatorname{supp}\beta\subset K\\
                  \|\beta\|_\infty\leq q_0}}
 \|\mathrm d^r\operatorname{Jet}(\beta)\|
 \leq\frac{C_r}{(q_1-q_0)^r}
       \sup_{\substack{\operatorname{supp}\beta\subset K\\
                        \|\beta\|_\infty\leq q_1}}
          |\operatorname{Jet}(\beta)|
 <\infty,\qquad r=1,2.
\]
The normalizations are smooth functions of these jets and their
conjugates, with denominators bounded away from zero by
Lemma~\ref{lem:DR-quantitative-normalization}.
For fixed $p\in\{2,\ldots,N+2\}$ and $a\in\{1,2\}$, write
$x_t=\widehat T_{p,t}^{(a),\rm eq}x$.
The chain rule gives, for each fixed $m\geq1$ and $\ell\in\mathbb Z$,
\[
 \sup_{x,\,|t|\leq\epsilon_0}\sum_{r=0}^2
 \left(|\partial_t^r\widehat{a_m}(x_t)|
       +|\partial_t^r\widehat{\chi_\ell}(x_t)|\right)<\infty.
\]
The source-defect formula \eqref{eq:DR-Schiffer-response-frame-defect}
also controls the raw fields on the compact source-coordinate ranges.
In particular,
\begin{equation}\label{eq:DR-finite-mode-absolute-frame-bounds}
 \max_{\substack{0\leq j\leq N\\a\in\{1,2\}}}
 \sup_{\overline U_A\times\Sigma_A}
 \left\{|\widehat L_{-j-2}^{(a),\mathrm{eq}}k|
 +\|dM(\widehat L_{-j-2}^{(a),\mathrm{eq}})\|
 +\|dM(\widehat L_{-j-2}^C)\|\right\}<\infty.
\end{equation}
The support margins
$\operatorname{supp}\chi\Subset U_A$ and
$\operatorname{supp}\zeta\Subset(-A,A)$, together with the explicit
inverse lifts in Proposition~\ref{prop:DR-measurable-frame-lift}, allow
$\epsilon_0$ to be decreased so that
\[
 \bigcup_{\substack{0\leq j\leq N,\ a\in\{1,2\}\\|t|\leq\epsilon_0}}
 \left\{\operatorname{supp}H_j
       \cup(\widehat T_{j+2,t}^{(a),\mathrm{eq}})^{-1}
                    (\operatorname{supp}H_j)\right\}
 \subset U_A\times\{h:|k(h)|<A\}.
\]

Write $\widehat f_t$ for the interior component of $x_t$.
Differentiation of the moment integral gives
\[
 \partial_t\widehat m_j(\widehat f_t)
 =\frac1\pi\int_{\mathbb D}\mu\left\{
 2\widehat f_t'\widehat f_t^{\,j}\partial_t\widehat f_t'
 +j(\widehat f_t')^2\widehat f_t^{\,j-1}\partial_t\widehat f_t
 \right\}\,dA.
\]
The second term is absent when $j=0$.  Cauchy's estimates on a compact
neighborhood of $\operatorname{supp}\mu$, followed by one more time
derivative, give
\[
 \max_{0\leq j\leq N}\sup_{x,\,|t|\leq\epsilon_0}
 \sum_{r=0}^2|\partial_t^r\widehat m_j(\widehat f_t)|<\infty.
\]

For averaged generators write the induced welding as
$l_t\circ h\circ\vartheta_t$, where $l_t,\vartheta_t$ are the circle
maps from the two source factorizations.  These factorizations are conformal on fixed
annuli around $\mathbb S^1$.  The same parameter estimates give
\[
 \sup_{x,\,|t|\leq\epsilon_0}\sum_{\ell=0}^2
 \left(\|\partial_t^\ell l_t\|_{C^2(\mathbb S^1)}
       +\|\partial_t^\ell \vartheta_t\|_{C^2(\mathbb S^1)}\right)<\infty.
\]
Their angular derivatives stay bounded away from zero for small time.
Write $\vartheta_t(e^{i\theta})=e^{i\varphi_t(\theta)}$ and set
\[
 a_t(u):=a(\varphi_t^{-1}(u))(\varphi_t^{-1})'(u),
 \qquad b_t:=b\circ l_t.
\]
Changing variables gives
\[
 J_{a,b}(l_t\circ h\circ \vartheta_t)
 =\frac1{2\pi}\int_0^{2\pi}a_t(u)b_t(h(e^{iu}))\,du,
\]
and, for $\ell=1,2$,
\[
 |\partial_t^\ell J_{a,b}(l_t\circ h\circ \vartheta_t)|
 \leq\sum_{r=0}^\ell\binom{\ell}{r}
       \|\partial_t^ra_t\|_\infty
       \|\partial_t^{\ell-r}b_t\|_\infty.
\]
These bounds are uniform in $h$.  Fixed additional compositions are
handled by \eqref{eq:DR-averaged-composition-change-variable} and the
fixed disk extensions from Appendix~\ref{app:uniform-response}.
The product and cylinder chain rules give the same bounds for the
chosen expression for $F$.

The derivatives of $\chi$ and $\zeta$ are uniformly continuous on
their compact coordinate ranges.  Hence, for any of the lifted families
$\widehat T_t$ and
$\widehat LH_j:=\left.\partial_t(H_j\circ\widehat T_t)\right|_0$,
\[
 \left\|\frac{H_j\circ\widehat T_t-H_j}{t}-\widehat LH_j\right\|_{\sup}
 \leq\sup_{|s|\leq|t|}
 \|\partial_s(H_j\circ\widehat T_s)-\widehat LH_j\|_{\sup}
 \longrightarrow0.
\]
Outside the common localization both the tests and their short
pullbacks vanish, so this is the required norm on the full frame space.

Finally, the source anomalies have bounded coefficients on the compact
source-coordinate ranges by
\eqref{eq:DR-Schiffer-response-frame-defect} and
\eqref{eq:DR-source-frame-vertical-field}.
Their action on coefficient and character generators is controlled by
the same finite jets.  On averaged generators it is given by the
preceding change of variables with $l_t,\vartheta_t$ equal to source
automorphisms.  Thus
\[
 YH_j=(Y\widehat m_j)\chi\widehat G_0
      +\widehat m_j\,d\chi[dM(Y)]\,\widehat G_0
      +\widehat m_j\chi\,Y\widehat G_0
\]
has uniformly bounded terms.  This proves the second assertion.
\end{proof}

\subsection{Descent to the welding measure}

Choose $\chi\in C_c^1(U_A)$ satisfying
\eqref{eq:DR-slice-cutoff-normalization}.
For a bounded slice function $H^\Sigma$ whose invariant lift
$\widehat H=H^\Sigma\circ s$ is differentiable along
$\widehat R_\mu^{\,\natural}$, set
\begin{equation}\label{eq:DR-ambient-slice-response}
 (R_\mu^\Sigma H^\Sigma)(h)
 :=(\widehat R_\mu^{\,\natural}\widehat H)(\id,h),
 \qquad h\in\Sigma.
\end{equation}
\begin{lemma}
\label{lem:DR-direct-Laurent-descent}
\begin{equation}\label{eq:DR-raw-quotient-IBP}
 \int R_\mu G_0\,d\widehat\nu_\kappa
 =\frac{c_{\mathrm L}}{12}
   \int G_0\sigma_h(\mu)\,d\widehat\nu_\kappa.
\end{equation}
\end{lemma}

\begin{proof}
On the slice, the Cauchy projection in the proof of
Lemma~\ref{lem:DR-Mobius-normalization-anomaly} gives
\begin{equation}\label{eq:DR-Cauchy-source-response-comparison}
 \widehat R_\mu^{\mathrm C}f=V_{\mu,f}\circ f+f'C_\mu,
 \qquad
 \widehat R_\mu^{\mathrm C}g=V_{\mu,f}\circ g.
\end{equation}
Together with the reflected variation, these are the derivatives of
$(\Psi_t\circ f\circ\Phi_t^0,\Psi_t\circ g)$ from the proof of
Lemma~\ref{lem:DR-physical-trace}.  For the finite Fourier representatives,
\eqref{eq:DR-Fourier-Cauchy-vector} and the decay at infinity give
\[
 C_\mu(0)=0,\qquad C_\mu(1)=C_\mu'(0).
\]
Let $\alpha(t)=\arg\Phi_t^0(1)$, using the branch with $\alpha(0)=0$.
The family $\widetilde\Phi_t$ fixing $0,1,\infty$ then satisfies
\[
 \Phi_t^0=r_{\alpha(t)}\circ\widetilde\Phi_t,
 \qquad
 \alpha(t)=-itC_\mu'(0)+i\bar t\,\overline{C_\mu'(0)}+o(|t|).
\]
Normalization of the target maps leaves the welding unchanged.
The uniform generator expansions in Appendix~\ref{app:uniform-response}
therefore give, on $\Sigma_{\rm r}$,
\begin{equation}\label{eq:DR-ambient-intrinsic-response-bridge}
 \begin{aligned}
 G_0^\Sigma&=G_0\circ\operatorname{weld},\\
 R_\mu^\Sigma G_0^\Sigma
 &=\bigl(R_\mu G_0-iC_\mu'(0)D_{\mathrm{rot}}G_0\bigr)
     \circ\operatorname{weld}.
 \end{aligned}
\end{equation}

We next prove the identity on the slice:
\begin{equation}\label{eq:DR-ambient-raw-quotient-IBP}
 \int_\Sigma R_\mu^\Sigma G_0^\Sigma\,d\widehat\nu_\kappa
 =\frac{c_{\mathrm L}}{12}
   \int_\Sigma G_0^\Sigma\sigma_h(\mu)\,d\widehat\nu_\kappa.
\end{equation}
The measure identity \eqref{eq:DR-raw-frame-measure} gives
\begin{equation}\label{eq:DR-K-cutoff-finite-measure}
 \widehat\nu_\kappa(\operatorname{supp}G_0^\Sigma)
 \leq\widehat\nu_\kappa(\Sigma_A)\leq e^{2A}<\infty.
\end{equation}
For each $N$, Lemma~\ref{lem:DR-localized-test-regularity} makes
$\chi\widehat G_0$ and its moment multiples admissible in
Lemma~\ref{lem:DR-raw-physical-IBP}.  Thus
\[
 \int\widehat R_{\mu,N}(\chi\widehat G_0)\,dM_\kappa^{\rm fr}
 =\int\chi\widehat G_0\,\delta_{\mu,N}\,dM_\kappa^{\rm fr}.
\]
The localization has measure at most
$H_G(\overline U_A)e^{2A}$.  By
\eqref{eq:DR-Laurent-cutoff-full-frame-response-rate} and
\eqref{eq:DR-Laurent-full-frame-divergence-rate},
\[
 \begin{aligned}
 &\int\left|
  (\widehat R_{\mu,N}-\widehat R_\mu^{\mathrm C})
       (\chi\widehat G_0)\right|\,dM_\kappa^{\rm fr}\\
 &\quad+\int|\chi\widehat G_0|
    \left|\delta_{\mu,N}-\frac{c_{\mathrm L}}{12}\sigma_h(\mu)\right|
       \,dM_\kappa^{\rm fr}
 \leq C_{\chi,F,\zeta,A}(N+2)^2q_A^{N+1}
 \longrightarrow0.
 \end{aligned}
\]
Passing to the limit in the finite identity gives
\begin{equation}\label{eq:DR-completed-Cauchy-full-frame-IBP}
 \int_{\mathcal F_{\rm fr}}
 \widehat R_\mu^{\mathrm C}(\chi\widehat G_0)\,dM_\kappa^{\rm fr}
 =\frac{c_{\mathrm L}}{12}
 \int_{\mathcal F_{\rm fr}}\chi\widehat G_0\sigma_h(\mu)\,
 dM_\kappa^{\rm fr}.
\end{equation}
To apply the decomposition \eqref{eq:DR-Cauchy-natural-anomaly-decomposition},
extend $dM$ complex linearly and set
$a_\mu(h):=dM_{(\mathrm{id},h)}(\widehat R_\mu^{\,\natural})$.
Let $V_{a_\mu(h)}^L$ be the corresponding left-invariant field on $G$,
and write $R_M(Q)=Q\circ M$ for right translation.
Equivariance and the anomaly identity give
\begin{align}
 dM(\widehat R_\mu^{\,\natural})(M,h)
 &=d(L_M)_{\mathrm{id}}a_\mu(h),
 \label{eq:DR-natural-left-invariant-path}\\
 dM(V_\mu^R)(M,h)
 &=d(R_M)_{\mathrm{id}}p_\mu(M,f).
 \label{eq:DR-anomaly-right-invariant-path}
\end{align}
The slice component is $R_\mu^\Sigma$, so
\begin{equation}\label{eq:DR-total-field-Weil-decomposition}
 \widehat R_\mu^{\,\natural}=R_\mu^\Sigma+V_{a_\mu(h)}^L.
\end{equation}
The coordinate maps and fields are measurable.  Equations
\eqref{eq:DR-off-slice-moment-bound},
\eqref{eq:DR-finite-mode-absolute-frame-bounds}, and
\eqref{eq:DR-Laurent-group-coordinate-rate}, with $M=\mathrm{id}$, give
\[
 \sup_{h\in\Sigma_A}\|a_\mu(h)\|<\infty.
\]
The compact support of $\chi$ lies in the uniform pole-free chart.
Thus \eqref{eq:DR-Mobius-anomaly-integrals}--
\eqref{eq:DR-Mobius-anomaly-polynomial} also give
\[
 \sup_{\operatorname{supp}\chi\times\Sigma_A}\|V_\mu^R\|<\infty.
\]
Together with \eqref{eq:DR-K-cutoff-finite-measure}, these bounds ensure
absolute integrability of the vertical derivatives.
Since $\widehat G_0$ is
$G$-invariant,
\begin{equation}\label{eq:DR-Cauchy-field-on-cutoff-lift}
 \widehat R_\mu^{\mathrm C}(\chi\widehat G_0)
 =\chi R_\mu^\Sigma G_0^\Sigma
  +\widehat G_0\{V_{a_\mu(h)}^L\chi+V_\mu^R\chi\}.
\end{equation}
Unimodularity and \eqref{eq:DR-Mobius-anomaly-Haar-zero} give
\begin{equation}\label{eq:DR-vertical-Haar-zero}
 \int_G V_{a_\mu(h)}^L\chi\,dH_G=0,
 \qquad
 \int_G V_\mu^R\chi\,dH_G=0.
\end{equation}
Fubini's theorem in
\eqref{eq:DR-full-frame-Weil-disintegration}, followed by
\eqref{eq:DR-slice-cutoff-normalization} and
\eqref{eq:DR-vertical-Haar-zero}, transforms
\eqref{eq:DR-completed-Cauchy-full-frame-IBP} into
\eqref{eq:DR-ambient-raw-quotient-IBP}.

It remains to remove the rotation term in
\eqref{eq:DR-ambient-intrinsic-response-bridge}.  Right rotation preserves
$k$ and $\widetilde\nu_\kappa$, hence also $\widehat\nu_\kappa$ by
\eqref{eq:DR-raw-frame-measure}.
Lemma~\ref{lem:DR-flow-core-uniform-differentiability}, with $v=1$, and
the compact support of $\zeta$ give
\[
 \begin{aligned}
 \int D_{\mathrm{rot}}G_0\,d\widehat\nu_\kappa
 &=\lim_{t\to0}\frac1t
   \left\{\int G_0\circ T_t^1\,d\widehat\nu_\kappa
                  -\int G_0\,d\widehat\nu_\kappa\right\}=0.
 \end{aligned}
\]
Since $C_\mu'(0)$ is constant and $\widehat\nu_\kappa$ is carried by
$\Sigma_{\rm r}$, the identification in Appendix~\ref{sec:DR-frame-spaces} and
\eqref{eq:DR-ambient-intrinsic-response-bridge} now turn
\eqref{eq:DR-ambient-raw-quotient-IBP} into
\eqref{eq:DR-raw-quotient-IBP}.
\end{proof}

\subsection{Final proof of Proposition~\ref{prop:DR-right-IBP}}
\label{app:final-right-IBP}

Retain $F\in\mathcal A_{\mathrm{fl}}$ and the finite Fourier
representative $\mu$.  Choose
$\zeta\in C_c^1(\mathbb R)$ such that
$0\leq\zeta\leq1$, $\zeta=1$ on $[-1,1]$, and
$\operatorname{supp}\zeta\subset[-2,2]$.  For $R\geq1$ and $s\in\mathbb R$, put
\begin{equation}\label{eq:DR-probability-cutoff}
 \zeta_R(s):=\zeta(s/R).
\end{equation}
The scalar function
\begin{equation}\label{eq:DR-tilted-compact-cutoff}
 \widetilde\zeta_R(s):=e^{-2s}\zeta_R(s)
 \end{equation}
belongs to $C_c^1(\mathbb R)$.  Apply
\eqref{eq:DR-raw-quotient-IBP} with
$\widetilde\zeta_R(k)F$ and use the measure identity
\eqref{eq:DR-raw-frame-measure},
$d\widehat\nu_\kappa=e^{2k}d\widetilde\nu_\kappa$.  The Leibniz rule gives
\begin{align}
 &\int\zeta_R(k)R_\mu F\,d\widetilde\nu_\kappa
 +\int\{\zeta_R'(k)-2\zeta_R(k)\}
      F R_\mu k\,d\widetilde\nu_\kappa\notag\\
 &=\frac{c_{\mathrm L}}{12}
   \int\zeta_R(k)F\sigma_h(\mu)\,
   d\widetilde\nu_\kappa,
 \label{eq:DR-tilt-before-capacity}\\
 \shortintertext{and hence}
 \int\zeta_R(k)R_\mu F\,d\widetilde\nu_\kappa
 &=\int\zeta_R(k)F\left\{
   \frac{c_{\mathrm L}}{12}\sigma_h(\mu)
   +2R_\mu k\right\}d\widetilde\nu_\kappa
 -\int\zeta_R'(k)F R_\mu k\,
  d\widetilde\nu_\kappa.
 \label{eq:DR-tilt-cutoff-expanded}
\end{align}
For the basic representative $\mu_{n,c}$, complex linearity in
\eqref{eq:DR-real-cos-response}--\eqref{eq:DR-real-sin-response} gives
\begin{equation}\label{eq:DR-complex-response-from-real-flows}
 R_{\mu_{n,c}}
 =\frac12\{D_{n,c}-D_{\mathrm{rot}}-iD_{n,s}\}.
\end{equation}
Lemma~\ref{lem:DR-flow-core-uniform-differentiability} and
\eqref{eq:DR-complex-response-from-real-flows} show that $R_\mu F$ is
bounded for every finite Fourier representative.  Let $K=\operatorname{supp}\mu$.  Koebe distortion on this fixed
compact subset of $\mathbb D\setminus\{0\}$ and
\eqref{eq:DR-capacity-complex-variation} give
\[
 |\sigma_h(\mu)|+2|R_\mu k|
 \leq\frac{\|\mu\|_\infty}{\pi}
 \int_K\left(|\mathcal Sf|
       +\left|\frac{f'^2}{f^2}-\frac1{z^2}\right|\right)dA
 \leq C_\mu,
\]
uniformly in $h$.  Moreover,
\begin{equation}\label{eq:DR-cutoff-error-bound}
 \left|\int\zeta_R'(k)F R_\mu k\,
 d\widetilde\nu_\kappa\right|
 \leq\frac{\|\zeta'\|_\infty}{R}
 \|F\|_{\sup}\|R_\mu k\|_{\sup}
 \longrightarrow0.
\end{equation}
Since $\zeta_R(k)\to1$ pointwise and $0\leq\zeta_R\leq1$,
dominated convergence as $R\to\infty$, followed by
$2R_\mu k=\tau_h(\mu)$ from
\eqref{eq:DR-capacity-complex-variation}, gives
\begin{align}
 \int R_\mu F\,d\widetilde\nu_\kappa
 &=\int F\left\{
   \frac{c_{\mathrm L}}{12}\sigma_h(\mu)
   +\tau_h(\mu)\right\}d\widetilde\nu_\kappa
 =\int F\rho_h(\mu)\,d\widetilde\nu_\kappa.
 \label{eq:DR-complex-right-IBP-proved}
\end{align}
Because the test algebra is stable under conjugation and
$\bar R_\mu F=\overline{R_\mu\overline F}$,
\[
 \int\bar R_\mu F\,d\widetilde\nu_\kappa
 =\overline{\int R_\mu\overline F\,d\widetilde\nu_\kappa}
 =\overline{\int\overline F\rho_h(\mu)\,d\widetilde\nu_\kappa}.
\]
Hence
\begin{equation}\label{eq:DR-conjugate-right-IBP-proved}
 \int\bar R_\mu F\,d\widetilde\nu_\kappa
 =\int F\overline{\rho_h(\mu)}\,d\widetilde\nu_\kappa.
\end{equation}

The Fourier pairings are given by \eqref{eq:DR-Fourier-A-values-proof}.
Lemma~\ref{lem:DR-flow-core-uniform-differentiability}, applied to
$v=1$, permits differentiation of right-rotation invariance:
\[
 \int D_{\mathrm{rot}}F\,d\widetilde\nu_\kappa
 =\lim_{t\to0}\frac1t
   \left(\int F\circ T_t^1\,d\widetilde\nu_\kappa
         -\int F\,d\widetilde\nu_\kappa\right)=0.
\]
Add
\eqref{eq:DR-complex-right-IBP-proved} and
\eqref{eq:DR-conjugate-right-IBP-proved}, first for $\mu_{n,c}$ and then
for $i\mu_{n,c}$, and use
\eqref{eq:DR-real-cos-response}--\eqref{eq:DR-real-sin-response}.  This
yields, for every $F\in\mathcal A_{\mathrm{fl}}$,
\begin{align}
 \int D_{n,c}F\,d\widetilde\nu_\kappa
 &=\int F\operatorname{Im}t_n\,d\widetilde\nu_\kappa,
 \label{eq:DR-coefficient-cos-IBP-proved}\\
 \shortintertext{and}
 \int D_{n,s}F\,d\widetilde\nu_\kappa
 &=\int F\operatorname{Re}t_n\,d\widetilde\nu_\kappa.
 \label{eq:DR-coefficient-sin-IBP-proved}
\end{align}
Real linearity and \eqref{eq:intro-score} prove
\eqref{eq:DR-hyp-IBP}.


\begin{thebibliography}{99}
\addcontentsline{toc}{section}{References}
\small

\bibitem{AB60}
L.~V.~Ahlfors and L.~Bers,
\emph{Riemann's mapping theorem for variable metrics},
Ann.\ of Math. (2) \textbf{72} (1960), 385--404.

\bibitem{AM01}
H.~Airault and P.~Malliavin,
\emph{Unitarizing probability measures for representations of Virasoro
algebra},
J. Math. Pures Appl. (9) \textbf{80} (2001), no.~6, 627--667.

\bibitem{AN08}
H.~Airault and Y.~A.~Neretin,
\emph{On the action of the Virasoro algebra on the space of univalent
functions},
Bull. Sci. Math. \textbf{132} (2008), no.~1, 27--39.

\bibitem{BJ24}
G.~Baverez and A.~Jego,
\emph{The CFT of SLE loop measures and the Kontsevich--Suhov conjecture},
arXiv:2407.09080, 2024, version 3 (July 2026).

\bibitem{BJ25}
G.~Baverez and A.~Jego,
\emph{Conformal welding and the matter--Liouville--ghost factorisation},
arXiv:2502.17076, 2025, version 2 (July 2025).

\bibitem{GQW25}
M.~Gordina, W.~Qian, and Y.~Wang,
\emph{Infinitesimal conformal restriction and unitarizing measures for
Virasoro algebra},
J. Math. Pures Appl. \textbf{195} (2025), Article 103669,
doi:10.1016/j.matpur.2025.103669; arXiv:2407.09426v3.

\bibitem{deBranges85}
L.~de Branges,
\emph{A proof of the Bieberbach conjecture},
Acta Math. \textbf{154} (1985), no.~1--2, 137--152.


\bibitem{FS25}
S.~Fan and J.~Sung,
\emph{Quasi-invariance for SLE welding measures},
arXiv:2502.15669, 2025.

\bibitem{FVW26}
S.~Fan, F.~Viklund, and Y.~Wang,
\emph{On the Loewner energy of a welding homeomorphism},
arXiv:2604.16737, 2026.

\bibitem{KMS22}
K.~Kavvadias, J.~Miller, and L.~Schoug,
\emph{Conformal removability of $\mathrm{SLE}_4$},
arXiv:2209.10532, 2022.

\bibitem{Nag92}
S.~Nag,
\emph{A period mapping in universal Teichm\"uller space},
Bull. Amer. Math. Soc. (N.S.) \textbf{26} (1992), no.~2, 280--287.

\bibitem{NS95}
S.~Nag and D.~Sullivan,
\emph{Teichm\"uller theory and the universal period mapping via quantum
calculus and the $H^{1/2}$ space on the circle},
Osaka J. Math. \textbf{32} (1995), no.~1, 1--34.

\bibitem{SW24}
J.~Sung and Y.~Wang,
\emph{Quasiconformal deformation of the chordal Loewner driving function
and first variation of the Loewner energy},
Math. Ann. \textbf{390} (2024), 4789--4812.

\bibitem{TT06}
L.~A.~Takhtajan and L.-P.~Teo,
\emph{Weil--Petersson metric on the universal Teichm\"uller space},
Mem. Amer. Math. Soc. \textbf{183} (2006), no.~861, viii+119.

\bibitem{Wang19}
Y.~Wang,
\emph{Equivalent descriptions of the Loewner energy},
Invent. Math. \textbf{218} (2019), no.~2, 573--621.

\bibitem{Zhan21}
D.~Zhan,
\emph{SLE loop measures},
Probab. Theory Related Fields \textbf{179} (2021), no.~1--2, 345--406.

\bibitem{KO24}
P.~Krupski and K.~Omiljanowski,
\emph{On hyperspaces of knots and planar simple closed curves},
arXiv:2401.13084, 2024.

\end{thebibliography}
\end{document}